\documentclass[12pt]{amsart}
\usepackage{tikz-cd}
\usepackage{amsmath}
\usepackage{fourier}
\usepackage{amssymb}
\usepackage{amscd}
\usepackage{amsthm}
\usepackage[centertags]{amsmath}
\usepackage{amsfonts}
\usepackage{newlfont}
\usepackage{graphicx}
\usepackage{amsfonts, amssymb}
\usepackage{mathrsfs}
\usepackage{latexsym}
\usepackage{tikz}
\usepackage{verbatim}
\usepackage[all]{xy}
\usepackage{enumitem}

\usepackage{setspace}

\usepackage[colorlinks=true,linkcolor=colorref,citecolor=colorcita,urlcolor=colorweb]{hyperref}
\definecolor{colorcita}{RGB}{21,86,130}
\definecolor{colorref}{RGB}{5,10,177}
\definecolor{colorweb}{RGB}{177,6,38}

\numberwithin{subsection}{section}

\newtheorem{theorem}{Theorem}[section]

\newtheorem{proposition}[theorem]{Proposition}
\newtheorem{corollary}[theorem]{Corollary}
\newtheorem{lemma}[theorem]{Lemma}

\theoremstyle{definition}
\newtheorem{remark}[theorem]{Remark}

\theoremstyle{remark}

\newcommand{\dis}{\displaystyle}

\newcommand{\Jj}{\mathcal J}

\newcommand{\NN}{\mathbb N}

\newcommand{\CC}{\mathbb C}

\newcommand{\ii}{\mathbf{i}}

\usepackage{mathtools}

\DeclareMathOperator{\mon}{mon}

\DeclareMathOperator{\id}{\mathrm{id}}

\newcommand{\bi}{\mathbf i}
\newcommand{\bj}{\mathbf j}

\newcommand{\chimon}{\chi_{\mon}}

\newcommand{\Pp}{\mathcal{P}}

\newcommand{\sinc}{\operatorname{sinc}}

\def\N{\mathbb{ N}}

\definecolor{miverde}{RGB}{42, 180, 59}

\begin{document}
\title[Local constants and Bohr's phenomenon for Banach spaces of analytic polynomials]{Local constants and Bohr's phenomenon for \\ Banach spaces of analytic polynomials}
 \author[Defant]{A.~Defant}
 \address{%
 Institut f\"{u}r Mathematik,
 Carl von Ossietzky Universit\"at,
 26111 Oldenburg,
 Germany}
 \email{defant@mathematik.uni-oldenburg.de}

\author[D. Galicer]{D.~Galicer}
\address{Departamento de Matem\'{a}ticas y Estad\'{\i}stica, Universidad T. Di Tella, Av. Figueroa Alcorta 7350 (1428), Buenos Aires, Argentina and IMAS-CONICET.} \email{daniel.galicer@utdt.edu}

 \author[Mansilla]{M.~Mansilla}
 \address{Departamento de Matem\'{a}tica,
 Facultad de Cs. Exactas y Naturales, Universidad de Buenos Aires and IAM-CONICET. Saavedra 15 (C1083ACA) C.A.B.A., Argentina}
 \email{mmansilla$@$dm.uba.ar}

 \author[Masty{\l}o]{M.~Masty{\l}o}
 \address{Faculty of Mathematics and Computer Science, Adam Mickiewicz University, Pozna{\'n}, Uniwersytetu \linebreak
  Pozna{\'n}skiego 4,
 61-614 Pozna{\'n}, Poland}
 \email{mieczyslaw.mastylo$@$amu.edu.pl}

 \author[Muro]{S.~Muro}
 \address{FCEIA, Universidad Nacional de Rosario and CIFASIS, CONICET, Ocampo $\&$ Esmeralda, S2000 Rosario, Argentina}
 \email{muro$@$cifasis-conicet.gov.ar}

\date{}


\begin{abstract}
The primary aim of this work is to develop methods that provide new insights into the relationships between fundamental constants
in Banach space theory--specifically, the projection constant, the unconditional basis constant and the Gordon-Lewis constant--for
the Banach space $\mathcal{P}_J(X_n)$ of multivariate analytic polynomials. This class consists of all polynomials whose monomial
coefficients vanish outside the set of multi-indices $J$, and it is equipped with the supremum norm on the unit sphere of the
finite-dimensional Banach space $X_n = (\mathbb{C}^n, \|\cdot\|)$. We establish a~general framework for proving quantitative results
on the asymptotic optimal behavior of these constants, which depend on both the dimension of the space and the degree of the polynomials.
Using the tools developed, we derive asymptotic estimates of the Bohr radius for general Banach sequence lattices. Additionally, we apply our results to the asymptotic study of local constants and the Bohr radius within finite-dimensional Lorentz sequence spaces, which requires a~refined analysis of the combinatorial structure of the associated \linebreak index sets. As a consequence, we obtain optimal results across a broad range of parameters.
\end{abstract}


\maketitle


\section*{Introduction} \label{Introduction}

The study of local properties of infinite-dimensional Banach spaces, which are intricately linked to the structure of finite-dimensional subspaces, is fundamental to both the local and general theory of Banach spaces. This area of research is driven by various aspects of modern analysis. One of the core challenges in local Banach space theory is understanding the asymptotic behavior of certain fundamental constants that have significant applications.

The goal of this work is to establish estimates for key constants in Banach space theory--specifically the projection constant, the unconditional basis constant, and the Gordon-Lewis constant--within the context of the finite-dimensional Banach spaces $\mathcal{P}_J(X_n)$ of polynomials. In this setting, $\mathcal{P}_J(X_n)$
denotes the Banach space of multivariate analytic polynomials $P$ defined on finite-dimensional Banach spaces $X_n = (\mathbb{C}^n, \|\cdot\|)$, where the monomial coefficients of $P$ vanish outside the set of multi-indices $J$. This space is equipped with the supremum norm over the unit sphere of $X_n$.  Additionally, we explore the relationships between the involved constants. Furthermore, we aim to unify these concepts in the study
of the Bohr radius for unit spheres of finite-dimensional Banach spaces and demonstrate the application of our general results to a significant class of finite-dimensional spaces generated by the scale of Lorentz spaces. To accomplish these objectives, we will employ techniques from various fields of analysis, including local Banach space theory, probability theory, complex analysis, Banach lattice theory, and combinatorics.

To clarify these aims, we will present key definitions and essential technical tools that naturally arise in our investigation.

\noindent{\bf The projection constant.}
The projection constant, in particular, is one of the most significant concepts in modern Banach space theory, closely related to the classical problem of describing complemented subspaces of Banach spaces—a problem that has been extensively studied since the inception of abstract operator theory.

Recall that if $X$ is a~closed subspace of a~Banach space $Y$, then the relative projection constant of $X$ in $Y$ is defined by
\begin{align*}
\boldsymbol{\lambda}(X, Y) & =  \inf\big\{\|P\|: \,\, P\in \mathcal{L}(Y, X),\,\, P|_{X} = \id_X\big\}\,,
\end{align*}
where $\id_X$ denotes the identity operator on $X$, and $\mathcal{L}(U,V)$ represents the Banach space of all bounded linear
operators between Banach spaces $U$ and $V$ with the supremum norm. Here, $\mathcal{L}(U):=\mathcal{L}(U,U)$, and we adopt the
convention that $\inf \varnothing = \infty$.

The (absolute) projection constant of $X$ is expressed as:
\[
\boldsymbol{\lambda}(X) := \sup \,\,\boldsymbol{\lambda}(I(X),Y)\,,
\]
where the supremum is taken over all Banach spaces $Y$ and isometric embeddings $I\colon X \to Y$. Note that if $X$ is
a~finite-dimensional Banach space and $X_1$ is a subspace of some $C(K)$-space isometric to $X$, then (see, e.g., \cite[III.B.5 Theorem] {wojtaszczyk1996banach}):
\begin{align*}
\boldsymbol{\lambda}(X) = \boldsymbol{\lambda}(X_1, C(K))\,.
\end{align*}
Thus, determining $\boldsymbol{\lambda}(X)$ is equivalent to finding the norm of a~minimal projection from $C(K)$ onto~$X_1$.

General bounds for projection constants of various finite-dimensional Banach spaces have been studied by many authors. The most
fundamental general upper bound is due to Kadets and Snobar~\cite{kadecsnobar}: For every finite-dimensional Banach space $X$,
we have
\begin{equation} \label{kadets1}
\boldsymbol{\lambda}(X) \leq \sqrt{\text{dim}X}.
\end{equation}

Regarding the Kadets-Snobar estimate~\eqref{kadets1}, we also mention a~remarkable result of\,Pisier
\cite[Corollary 10.8]{pisier1986factorization}, who constructed an infinite dimensional Banach space $X$ such that $\|P\| \geq \delta \sqrt{\text{rank}\,P}$ for some
$\delta \in (0, 1)$ and all  finite rank projections $P\colon X \to X$.
K\"onig \cite{konig1985spaces} proved that the Kadets-Snobar estimate~\eqref{kadets1} is asymptotically  best possible. More
precisely, he showed that there exists a~sequence $(X_{n_k})_{k=1}^{\infty}$ of finite-dimensional real Banach spaces such
that $\text{dim}(X_{n_k}) = n_k$, where $n_k\to \infty$ as $k\to \infty$, and
\[
\lim_{k\to \infty} \frac{\boldsymbol{\lambda}(X_{n_k})}{\sqrt{n_k}} = 1\,.
\]
Another surprising result in this direction is due to Figiel, Lindenstrauss and Milman \cite{FLM}. It states that there is
a~universal constant $c\in (0, 1)$ such that every $n$-dimensional Banach space $E$ contains a~subspace $F$ such that
$\boldsymbol{\lambda}(F) \geq c\,\sqrt{n}$.

\noindent{\bf The unconditional basis constant.}
The unconditional basis constant of a~basis $(e_i)_{i \in I}$ of a~Banach space $E$ is given by the infimum over all $K > 0$
such that for any finitely supported family $(\alpha_i)_{i \in I}$ of scalars and for any finitely supported family
$(\varepsilon_i)_{i \in I}$ with $|\varepsilon_i| =1, \,  i \in I$ we have
\begin{equation*}\label{unconditionality}
\Big\Vert  \sum_{i \in I} \varepsilon_i \alpha_i e_i \Big\Vert \leq K \Big\Vert \sum_{i \in I} \alpha_i e_i \Big\Vert\,.
\end{equation*}
We denote the unconditional basis constant of $(e_i)_{i\in I}$ by
$\boldsymbol{\chi}((e_i)_{i\in I})= \boldsymbol{\chi}((e_i)_{i\in I}; X)$. We also write  $\boldsymbol{\chi}((e_i)_{i \in I})
= +\infty$, whenever $(e_i)_{i\in I}$ is not unconditional, and say that $(e_i)_{i\in I}$ is a~$1$-unconditional basis, whenever
$\boldsymbol{\chi}((e_i)_{i\in I}) =1$. The unconditional basis constant $\boldsymbol{\chi}(X)$ of $X$ is defined to be the
infimum of $\boldsymbol{\chi}((e_i)_{i\in I})$ taken over all possible unconditional basis $(e_i)_{i\in I}$ of~$X$.

\noindent{\bf The  Gordon-Lewis constant.}
A significant milestone in the study of  unconditional basis constants was achieved by Gordon and Lewis in \cite{gordon1974absolutely},
who introduced the constant now known as the G-L~constant. In light of our study of this constant for the space of polynomials,
we need to introduce some additional concepts to properly present the original definition of this important  concept in the context of
local Banach space theory.

Given Banach spaces $X$, $Y$ and $1 \leq p \leq \infty$, an  operator $u\in \mathcal{L}(X,Y)$ is said to be $p$-factorable
whenever  there exist a~measure space $(\Omega, \Sigma, \mu)$ and operators $v \in  \mathcal{L}(X, L_p(\mu))$,
$w \in \mathcal{L}(L_p(\mu), Y^{\ast\ast})$, satisfying the following factorization
$
\kappa_Yu\colon X \stackrel{v} \longrightarrow L_p(\mu) \stackrel{w} \longrightarrow Y^{\ast\ast}\,;
$
here, as usual, $\kappa_{Y}\colon Y \to Y^{\ast\ast}$ is the canonical injection. In this case the $\gamma_p$-norm of the
$p$-factorable operator $u$ is given by
\begin{equation}\label{fac-p}
 \gamma_p(u) = \inf \|v\| \|w\|\,,
\end{equation}
where the infimum is taken over all possible factorizations. We are  mainly interested  in the  norms  $\gamma_p$ for operators acting
between  finite dimensional Banach spaces $X$ and $Y$. In this case, the infimum in \eqref{fac-p} is realized considering all possible
factorizations of the more simple form
\begin{equation*}
\begin{tikzcd}
X  \arrow[rd, "v"']  \arrow[rr, "u"] &  & Y \,\,,\\
& \ell_p^n   \arrow[ru, "w"'] &
\end{tikzcd}
\end{equation*}
where $n$ is arbitrary.
An operator $u\in \mathcal{L}(X,Y)$ is said to be $1$-summing  if there is a~constant $C>0$ such that
for each choice of finitely many $x_1, \ldots, x_N\in X$ one has
\[
\sum_{j=1}^N \|ux_j\|_Y  \leq C
\sup\Big\{\sum_{j=1}^{N} |x^\ast(x_j)|\,:\,\, \|x^\ast\|_{X^\ast} \leq 1\Big\}\,.
\]
By $\pi_{1}(u\colon X\to Y)$  we denote the least such $C$ satisfying this inequality. It is known that  for any $n$-dimensional
Banach space $X$, we have
\begin{equation*}
\boldsymbol{\lambda}(X)\,\pi_1(\id_X) \geq n\,,
\end{equation*}
and if $X$ has enough symmetries (i.e., any operator $u\colon X \to X$ which commutes with all isometries on $X$, is a~multiple
of the identity map), then
\begin{equation*}
\boldsymbol{\lambda}(X)\pi_1(\id_X)= n\,,
\end{equation*}
see, e.g., \cite{garlinggordon,konig1990bounds}. A Banach space $X$ has the Gordon-Lewis property if every $1$-summing operator
$u\colon X \to \ell_2$ is $1$-factorable. In this case, there is a~positive number $C$ such that for all $1$-summing operators
$u \colon X \to \ell_2$
\begin{equation*}
\label{eq: def G-L}
\gamma_1(u)\leq C\,\pi_1(u)\,,
\end{equation*}
and the best such $C>0$ is called the Gordon--Lewis constant of $X$ and denoted by $\boldsymbol{g\!l}(X)$.

We can further clarify the aforementioned comment about the crucial role of the Gordon-Lewis constant in the study of unconditionality
in Banach spaces. Specifically, Gordon and Lewis \cite{gordon1974absolutely} showed  that for every Banach space $X$ with
an unconditional basis~$(e_i)_{i \in I}$ the following estimate holds
\begin{equation} \label{gl-inequality}
\boldsymbol{g\!l}(X)\leq  \boldsymbol{\chi}(X) \leq \boldsymbol{\chi}( (e_i)_{i \in I})\,.
\end{equation}
It is worth noting that in contrast the unconditional basis constant, the Gordon-Lewis constant has the useful (ideal)
property that $\boldsymbol{g\!l}(X)\leq \|u\|  \|v\|\,\,\boldsymbol{g\!l}(Y)$ whenever $u\colon X \to Y$ and
$v\colon Y \to X$ are operators such that $\text{id}_{X} = uv$. In particular, if $X_0$ is an isometric subspace of $X$,
we get the following estimate
\begin{equation*}
\boldsymbol{g\!l}(X_0)\leq \boldsymbol{\lambda}(X_0,X)\,\, \boldsymbol{g\!l}(X)\,.
\end{equation*}

The projection constant of a~Banach space $X$ can be formulated in terms of the  $\infty$-factorization norm of the identity
operator $\id_X$.  More precisely, if $X$ is a~Banach space and $X_0$
is any subspace of some $L_{\infty}(\mu)$ isometric to $X$, then
\begin{align} \label{gammainfty}
\boldsymbol{\lambda}(X)= \gamma_\infty(\id_X)= \boldsymbol{\lambda}(X_0, L_{\infty}(\mu))\,.
\end{align}

In the general case of finite-dimensional spaces, determining these constants, or even finding their optimal asymptotic behavior
as a function of the dimension of the space, is indeed a challenging problem.

\noindent{\bf Summary of key results.}
Our focus is on spaces of  polynomials defined on $n$-dimensional Banach spaces.

We  mention three fundamental examples for projection constants of multivariate polynomials.
The first one is a theorem of  Lozinski and Kharshiladze (see \cite[IIIB. Theorem 22]{wojtaszczyk1996banach} and
\cite{natanson1961constructive}), which shows a  precise estimate of the projection constant of $\text{Trig}_{\leq m}(\mathbb{T})$,
the space of trigonometric polynomials on the torus $\mathbb{T} \subset \mathbb{C}$ of degree at most $m$ equipped with the sup-norm:
\begin{equation} \label{LoKa}
\boldsymbol{\lambda}\big(\text{Trig}_{\leq m}(\mathbb{T})\big)= \frac{4}{\pi^2} \log(m+1) + o(1).
\end{equation}

The second result is a beautiful formula from the 1980s, due to Ryll and Wojtaszczyk~\cite{ryll1983homogeneous} (see also
\cite[IIIB. Theorem 15]{wojtaszczyk1996banach}), which provides the exact value of the projection constant for the space
$\mathcal{P}_{m}(\ell^n_2)$ of all $m$-homogeneous polynomials on the $n$-dimensional complex Hilbert space $\ell_2^n$:
\begin{equation} \label{RWRW}
\boldsymbol{\lambda}\big(\mathcal{P}_{m}(\ell^n_2)\big) = \frac{\Gamma(n+m) \Gamma(1+\frac{m}{2})}{\Gamma(1+m)\Gamma(n+\frac{m}{2})}.
\end{equation}

This result was a key step in addressing several open questions of the time in the context of Hardy and Bloch spaces on the complex ball, including one posed by Rudin, who asked whether there exists an inner function on the open unit ball of Hilbert spaces (see, e.g.,  Rudin's monograph \cite{rudin1980} and Aleksandrov's article \cite{aleksandrov1984inner}).

Motivated by the study of multivariate variants of Bohr's famous power series theorem,  Defant and Frerick
\cite{defant2011bohr} (see also \cite[Section 22.2]{defant2019libro}) provided (implicitly) the following asymptotically sharp estimate
\begin{equation}\label{eq:proj const pols on lp}
\boldsymbol{\lambda}\big(\mathcal P_m(\ell_r^n)\big)  \sim_{C^m} \left( 1+\frac{n}{m} \right)^{m\left(1-\frac{1}{\min\{r,2\}}\right)}\,,
\end{equation}
where $1 \leq r \leq \infty$ and $\sim_{C^m}$ stands for equivalence up to $C^m$ with a constant $C$ only depending on $r$.

We note that the cases involving polynomial spaces on real and complex Hilbert spaces, Boolean cubes, as well as Dirichlet polynomials, have been explored in \cite{defant2024minimalarxiv},\cite{defant2024ryll}, \cite{defant2024asymptotic}, and \cite{defant2024projection}, respectively. These works, authored by Defant, Galicer, Mansilla, Masty{\l}o, and Muro, are part of a recent systematic study that also includes an analysis of the projection constant for the trace class \cite{defant2024trace}.

We now provide a brief overview of some of the main results of the present article. In Section $1$, we review the key definitions and notation used
throughout the paper.

In Section 2, we establish the relationships between  the Gordon-Lewis constant ${\boldsymbol{g\!l}}\big(\mathcal{P}_{J}(X_n)\big)$,
the the unconditional basis constant $\boldsymbol{\chimon}\big(\mathcal{P}_{J}(X_n) \big)$ (that is, the unconditional basis constant
of the monomials $(z^\alpha)_{\alpha \in J}$ in $\mathcal{P}_{J}(X_n)$) and the projection constant
$\boldsymbol{\lambda}\big(\mathcal{P}_{J}(X_n)\big)$ in the case that $X_n = (\mathbb{C}^n, \|\cdot\|)$ is a Banach lattice
and $J$ an index set. The main results here are Theorem~\ref{gl_versus_proj} and Theorem~\ref{thm: uncond cte vs proj cte}.

Applying a famous theorem of Lozanovski\v{\i}, we derive new  upper estimates for $\boldsymbol{\lambda}(\mathcal{P}_{J}(X_n)\big)$ in the
setting of Banach lattices $X_n = (\mathbb{C}^n, \|\cdot\|)$. We draw attention to Theorem~\ref{thm: bound for coef functional} and
Corollary~\ref{coro: bound proj constant tetra}. These results allow us to systematically extend estimates like~\eqref{eq:proj const pols on lp}
to a wider range of spaces $X_n$ and index sets $J$. To obtain lower estimates for $\boldsymbol{\chimon}\big(\mathcal{P}_{I}(X_n)\big)$, we use probabilistic methods. An application of the obtained estimates is Theorem~\ref{lower bound 2-convex for poly}, which states that in
a class of $2$-convex Banach lattices $X_n$ for a large class of index sets $J$, the lower and upper bounds of the aforementioned constants
$\boldsymbol{\chimon}\big( \mathcal{P}_{J}(X_n) \big)$ and $\boldsymbol{\lambda}\big(\mathcal{P}_{J}(X_n)\big)$ are asymptotically
correct.

We also revisit and expand on various results from the literature concerning Bohr radii in high dimensions in this section.
Recall that given an $n$-dimensional Banach lattice $X_n = (\mathbb{C}^n, \| \cdot \|)$, the Bohr radius $K(B_{X_n})$ of its open unit ball $B_{X_n}$ is the largest $0 < r < 1$ such that for every holomorphic function $f: B_{X_n} \to \mathbb{C}$, the following holds:
\[
\sup_{z \in r B_{X_n}} \sum_{\alpha \in \mathbb{N}_0^{(\mathbb{N})}} \left|\frac{\partial^\alpha f(0)}{\alpha!} z^\alpha \right| \le \sup_{z \in B_{X_n}} |f(z)|.
\]
Bohr's famous power series theorem in \cite{bohr1914theorem} shows, for the unit ball $\mathbb{D}$ of the Banach space
$\mathbb{C}$, that $K(\mathbb{D}) = \frac{1}{3}\,$. As usual, the Banach space of all bounded holomorphic functions
$f\colon B_{X_n} \to \mathbb{C}$ is denoted by $H_\infty(B_{X_n})$.
We will focus on a more general setting. Fixing  an  $n$-dimensional  Banach lattice
$X_n = (\CC^n, \|\cdot\|)$ together with 
an index set $J \subset \NN_0^{(\NN)}$, we consider the closed subspace
\[
H_{\infty}^{J}(B_{X_n}) :=  \left\lbrace f \in H^\infty(B_{X_n}) : \, c_\alpha(f) = 0 \text{ for } \, \alpha \, \notin J \right\rbrace
\]
of the Banach space of all bounded holomorphic functions $H_\infty(B_{X_n})$, and define
\begin{equation}\label{definitionB}
K(B_{X_n},J) := \sup \bigg\{r\in (0, 1) \colon \sup_{z \in r  B_{X_n}}
\sum_{\alpha \in J} \Big|\frac{\partial^\alpha f(0)}{\alpha!} z^\alpha\Big|
\le \|f\|_\infty \,\,\, \text{for all \,\, $f \in H^J_\infty (B_{X_n})$} \bigg\}\,,
\end{equation}
the Bohr radius of $B_{X_n}$ with respect to $J$, and  $K(B_{X_n}) := K\big(B_{X_n}, \NN_0^{(\NN)}\big)$.

Based on our findings, we describe the behavior of Bohr radii with respect to the unconditional basis and projection constants of
characteristics 'homogeneous' index sets, as well as the characteristic 'Homogeneous building blocks' $H^{J}_{\infty}(B_{X_n})$.
Additionally, we explore their relationship to the convexity of the underlying Banach sequence lattices. The key result in this
section  is Theorem~\ref{limits++} which shows for a large class of Banach sequence lattices $X$ and sets of indices $J$  that
$K(B_{X_n},J)\big/\sqrt{\log n /n} \to  1$ as $n \to \infty$. This result is inspired by the remarkable work of Bayart, Pellegrino, and Seoane, who in \cite{bayart2014bohr} established that the preceding limit holds for $X = \ell_{\infty}$ and the full set of multi-indices, with the asymptotic order initially determined by Defant, Frerick, Ortega-Cerd{\`a}, Ouna{\"{\i}}es, and Seip in \cite{defant2011bohnenblust}.

In Section 3, we focus on polynomials defined on finite-dimensional Lorentz spaces $\ell_{r,s}^n$ and supported on a given finite index
set $J$ of multi-indices. We provide asymptotically accurate estimates for the projection constant and the unconditional basis constant
of $\mathcal{P}_J(\ell_{r,s}^n)$, which depend on both the dimension $n$ and the structure of $J$. In particular, we show in
Theorem~\ref{t-final} that for the set of tetrahedral indices, the estimates for the space of homogeneous polynomials on $\ell_r^n$
(as given by \eqref{eq:proj const pols on lp}) remain valid for $\mathcal{P}_J(\ell_{r,s}^n)$ and for more general sets of indices $J$,
independently of the second parameter $s$. To establish similar  results for more general index sets (see the Theorems~\ref{t-finalII}, \ref{bound_similar_ell_r}, ~\ref{proj lorentz s<r}, and~\ref{proj lorentz s>r}), we overcome subtle technical challenges that are more
complex than those encountered in the case of $\ell_r^n$, requiring the use of most of the techniques previously developed. Finally, we
apply this to prove Theorem~\ref{thm: main bohr radii} on asymptotic estimates of Bohr radii in finite-dimensional Lorentz sequence spaces.
Specifically, assuming that the index set \( J \) has a certain mild structure (including the tetrahedral indices), we provide the optimal
asymptotic order of $K(B_{\ell_{r,s}^n}, J)$ for (almost) all $r$ and $s$.

For the entire index set $J = \N_{0}^{(\N)}$, this problem was addressed in \cite[Corollary 10]{defant2018bohr}. The results presented here
significantly advance those findings by proving the limit in certain cases and providing a new asymptotically optimal result over a range
of parameters, thus making a relevant progress on an open problem in the field. Our research contributes to the extensive recent literature
on both one-dimensional and multidimensional Bohr radii. For the one-dimensional case, see, for example, \cite{alkhaleefah2019bohr,
beneteau2004remarks,
bhowmik2018bohr,
bombieri2004remark,
ismagilov2020sharp,
kayumov2017bohr,
kayumov2018bohr,
paulsen2002bohr,
paulsen2004bohr}, and for the multivariate case
\cite{
bayart2012maximum,
bayart2014bohr,
boas2000majorant,
defant2011bohr,
defant2011bohnenblust,
defant2003bohr,
defant2018bohr,
khavinson1997bohr}.

\section{Background and notation}

We will require significant notation from (local) Banach space theory, as used in the monographs \cite{diestel1995absolutely, LT1,
pisier1986factorization, tomczak1989banach}. Unless otherwise specified, we consider complex Banach spaces. The dual Banach space
of $X$ is denoted by $X^\ast$.
The symbol $B_X$
(resp., $\overline{B}_X$) denotes the open (resp., closed) unit ball of $X$.

If $X$ and $Y$ are isomorphic spaces, that is, there is an invertible operator from $X$ onto $Y$, we write $X\simeq Y$. We use the
notion $X\equiv Y$ whenever $X$ and $Y$ may be identified isometrically. To indicate injective bounded mappings $T\colon X \to Y$,
we sometimes write $T\colon X \hookrightarrow Y$.

For isomorphic Banach spaces $X$ and $Y$, the Banach-Mazur  distance between $X$ and $Y$ is defined to be
\[
d(X, Y):=\inf \big\{\|T\|\,\|T^{-1}\|: \,\, \text{$T$ an isomorphism of $X$ onto $Y$} \big\}\,.
\]
If $X$ and $Y$ are not isomorphic, we let $d(X, Y)= +\infty$.

Given two sequences $(a_n)$ and $(b_n)$ of non-negative real numbers we write $a_n \prec_C b_n$, if there is a~constant $C>0$ such
that $a_n \leq C\,b_n$ for all $n\in \mathbb{N}$, while $a_n \sim_C b_n$ means that $a_n \prec_C b_n$ and $b_n \prec_C a_n$ holds.
In the case that an extra parameter $m$ is also involved, for two sequences of non-negative real numbers $(a_{n,m})$ and $(b_{n,m})$,
we write $a_{n,m} \prec_{C^m} b_{n,m}$ when there is a hypercontractive comparison, that is, there is a~constant $C>0$ (independent
of $n$ and $m$) such that $a_{n,m} \leq C^m b_{n,m}$ for all $n,m \in \mathbb{N}$. Of course, if $a_{n,m} \prec_{C_1^m} b_{n,m}$ and
$b_{n,m} \prec_{C_2^m} a_{n,m}$ we simply write $a_{n,m} \sim_{C^m} b_{n,m}$.

We denote the set of all sequences \(\alpha \in \mathbb{N}_0^{\mathbb{N}}\) with finite support as \(\mathbb{N}_0^{(\mathbb{N})}\).
As is customary, we refer to these sequences \(\alpha\) as multi-indices. It is clear that
\[
\mathbb{N}_0^{(\mathbb{N})} = \bigcup_{n \in \mathbb{N}} \mathbb{N}_0^n\,,
\]
where  $\mathbb{N}_0^n$ is  interpreted as a subset of $\mathbb{N}_0^{\mathbb{N}}$. We say  that   $\mathbb{N}_0^n$  forms the  set
of all
multi indices of length $n$.  For each multi-index \( \alpha = (\alpha_i) \in \mathbb{N}_0^{(\mathbb{N})} \), we define the order of
\( \alpha \) as
\[
|\alpha| = \sum \alpha_i\,.
\]
Using this notion, for \( m,n \in \mathbb{N} \), we define
\begin{align*}
\Lambda(m,n) &= \big\{ \alpha \in \mathbb{N}_0^n \colon |\alpha| = m \big\}, \\
\Lambda(\leq m,n) &= \big\{ \alpha \in \mathbb{N}_0^n \colon |\alpha| \leq m \big\},
\end{align*}
as well as
\[
\Lambda(m) := \bigcup_n \Lambda(m,n) \quad \text{and} \quad \Lambda(\leq m) := \bigcup_n \Lambda(\leq m,n).
\]
The following  formula and estimate
\begin{equation} \label{cardi}
|\Lambda(m,n)| = \dbinom{n+m-1}{m}\leq e^m \Big( 1 + \frac{n}{m}\Big)^m
\end{equation}
for the cardinality of $\Lambda(m,n)$ is crucial for our purposes.

For simplicity of notation and presentation throughout the paper, a subset of \(\mathbb{N}_0^{(\mathbb{N})}\) is called an \textit{index set}.
If \(J\) is an index set and \(m,n \in \mathbb{N}\), then we will frequently consider the following four building blocks of \(J\):
\[
J(m,n) : = J \cap \Lambda(m,n) \quad \text{ and } \quad  J(\leq m,n) : = J \cap \Lambda(\leq m,n)\,.
\]
\[
J(m) : = J \cap \Lambda(m) \quad \text{ and } \quad  J(\leq m) : = J \cap \Lambda(\leq m)\,.
\]
An index set $J$ is said to have degree at most $m$ whenever $J = J (\leq m)$, that is, $|\alpha| \leq m$ for all $\alpha \in J$.
Moreover,  $J$ is  called  $m$-homogeneous if  $J = J (m)$, so  $|\alpha| = m$ for all $\alpha \in J$.

A multi-index \(\alpha = (\alpha_i) \in \mathbb{N}_0^{(\mathbb{N})}\) is called tetrahedral if each entry \(\alpha_i\) is either \(0\)
or \(1\).  We denote the index set of all tetrahedral multi-indices by \(\Lambda_T\). The preceding definitions also provide precise
meanings for \(\Lambda_T(m,n)\), \(\Lambda_T(\leq m,n)\), \(\Lambda_T(m)\), and \(\Lambda_T(\leq m)\).

Given  $J \subset  \Lambda(m,n)$, we  need to consider  the index set
\[
J^\flat \subset \Lambda(m-1,n)\,,
\]
which consists of all $\alpha \in \Lambda(m-1,n)$ for which there is $1 \leq k \leq n$ and $\beta \in J$ such $\beta_i = \alpha_i$
for all $1 \leq i\neq  k\leq n$ and $\beta_k = \alpha_k +1$. We call $J^\flat$ the reduced index
set of $J$. Note that $\Lambda(m,n)^\flat=\Lambda(m-1,n)$.

We point out that it will sometimes be convenient to use an equivalent description of \(\Lambda(m,n)\). To describe it, let us denote
\[
\text{
$
\mathcal{M}(m,n)  := \{1, \ldots, n\}^{m}$ \quad  and \quad $ \mathcal{J}(m,n);  = \big\{\bj
= (j_1, \ldots, j_m) \in \mathcal{M}(m, n)\colon \, j_1 \leq \ldots \leq j_m\big\}\,,
$
}
\]
and observe that there is a~canonical bijection between $\mathcal{J}(m,n)$ and $\Lambda(m,n)$. Indeed, assign to
$\bj \in \mathcal{J}(m,n) $ the multi-index $\alpha \in \Lambda(m,n)$ given by $\alpha_r = |\{k \colon \bj_k = r\}|,\, 1 \leq r \leq n$,
and conversely for each $\alpha \in \Lambda(m,n)$ the index $\bj \in \mathcal{J}(m,n)$, where
\[
j_1 =\ldots =j_{\alpha_1} =1, \quad\, j_{\alpha_1 +1} =\ldots =j_{\alpha_1+\alpha_2} =2,\, \dots , \,j_{\alpha_{n-1} +1}
=\ldots =j_{\alpha_{n-1}+\alpha_n} = n\,.
\]

\noindent
On $\mathcal{M}(m,n)$ we consider the following equivalence relation: $\bi  \sim \bj$ if there is a permutation $\sigma$ on $\{1, \ldots,m\}$
such that $(i_1, \ldots, i_k) = (i_{\sigma(1)}, \ldots, i_{\sigma(m)})$. The equivalence class of $\bi \in \mathcal{M}(m,n)$ is
denoted by $[\bi]$, and its cardinality  by $|[\bi]|$. Provided that $\bj$ is associated with $\alpha$, we have that
\begin{equation*}
|[\alpha]| :=  |[\bj]| =\frac{m!}{\alpha!}\,.
\end{equation*}
Note that, given an index set \( J \subset \mathcal{J}(m,n) = \Lambda(m,n) \), the reduced set \( J^\flat \) in terms of the \(\mathbf{j}\)-mode is described as follows:
\begin{equation}\label{index sets}
J^\flat = \big\{ \bj \in \mathcal{J}(m-1,n)
\colon \exists 1 \leq k \leq n \,\, \text{  such that }\,\,\, (\bj,k)_\ast \in J \big\}\,,
\end{equation}
where we associate to each $\bi \in \mathcal{M}(m,n)$  the unique element $\bi_\ast\in \mathcal{J}(m,n)$ for which $\bi \in [\bi_\ast]$.

By $\mathcal{P}(\mathbb{C}^n)$ we denote the vector space of all finite  polynomials
\[
P(z)=\sum_{\alpha\in \mathbb{N}_0^n}c_\alpha(P)\,z^\alpha, \quad\, z=(z_1, \ldots, z_n)\in \mathbb{C}^n\,,
\]
where $c_\alpha(P)\in \mathbb{C}$ for each $\alpha \in \mathbb{N}_0^n$. More generally, for any  nonempty  finite index set $J$, we define
$$\mathcal{P}_J(\mathbb{C}^n)$$
to be the subspace of all $P \in \mathcal{P}(\mathbb{C}^n)$ for which $c_{\alpha}(P) = 0$ for all $\alpha \notin J$. Clearly,
$
\mathcal{P}_J(\mathbb{C}^n) = \mathcal{P}_{J\cap \mathbb{N}_0^n}(\mathbb{C}^n)\,.
$

For  $m \in \mathbb{N}$, we denote by  $\mathcal{P}_{\leq m}(\mathbb{C}^n)$ the space of all   polynomials
$P \in \mathcal{P}(\mathbb{C}^n)$ which have degree $\text{deg}\,P := \max\{|\alpha|:\, \alpha\in J \} \leq m$, and by $\mathcal{P}_{m }(\mathbb{C}^n)$ the space of all $m$-homoge\-neous
polynomials. Note that under the above notation
$\mathcal{P}_{\leq m}(\mathbb{C}^n)=  \mathcal{P}_{\Lambda(\leq  m)}(\mathbb{C}^n)$
and
$\mathcal{P}_{m}(\mathbb{C}^n)=  \mathcal{P}_{\Lambda(m)}(\mathbb{C}^n)$.

Given a nonempty  finite  index set $J$ and a~Banach space $X_n := (\mathbb{C}^n,\|\cdot\|)$, we equip $\mathcal{P}(\mathbb{C}^n)$
with the sup norm
\[
\|P\|_{B_{X_n}} := \sup_{z\in B_{X_n}} |P(z)|\,,
\]
and  denote the resulting Banach space by $\mathcal{P}_J(X_n)$. This is, in fact, the main object of our interest.

We  also need a few facts on polynomials defined on arbitrary Banach spaces $X$ (and not only $X_n:= (\mathbb{C}^n, \|\cdot\|)$).
For all relevant information we refer to  \cite{defant1992tensor,defant2019libro,dineen1999complex,floret1997natural}.

Let $X$ be a   Banach space $X$ over the field $\mathbb{K}$,  where $\mathbb{K} \in \{\mathbb{R}, \mathbb{C}\}$, and $m \in \mathbb{N}$.
A mapping $P\colon X\to \mathbb{K}$ is said to be a (bounded) $m$-homogeneous  polynomial if there exists a~(bounded) $m$-linear form
$T\colon X\times\cdots \times X \to \mathbb{K}$  such that $P(z)=T(z,\dots,z)$ for all $z\in X$. In fact, given a~(bounded) $m$-homogeneous
polynomial $P$ on $X$, by the usually called polarization formula
\begin{equation}\label{formula de polarizacion}
T(z^{(1)}, \dots, z^{(m)}) = \frac{1}{m!2^m} \sum_{\varepsilon_i = \pm 1} \varepsilon_1 \dots \varepsilon_m P\left(\sum_{i=1}^m \varepsilon_i z^{(i)}\right),
\end{equation}
there is a unique symmetric (bounded) $m$-linear form $T$ with the property that $P(z)=T(z,\dots,z)$ for all $z\in X$. We as usual denote
this unique form by $\overset{\vee}{P}=T$. The space of all bounded $m$-linear forms over $X$ is denoted by $\mathcal L_m(X)$, whereas we
write $\mathcal{P}_m(X)$ for all bounded $m$-homogeneous polynomials on $X$. Endowed with the norms
\[
\Vert T \Vert = \sup_{z^{(i)} \in B_{X}} \vert T(z^{(1)}, \dots, z^{(m)}) \vert \,\,\,\,\,\,\,\,\,\,\text{resp.,} \,\,\,\,\,\,\,\,\,\,
\Vert P \Vert = \sup_{z \in B_{X}} \vert P(z) \vert \,,
\]
each gets a~Banach space. Furthermore, we need the notion of a polynomial of degree less than or equal to $m$ on a given Banach space $X$.
By $\mathcal{P}_{\leq m}(X)$ we denote the linear space of all functions $P\colon X \to \mathbb{K}$ having the  form $P = P_0 + \sum_{k=1}^m P_k$,
where $P_k \in \mathcal{P}_k(X), \, 1 \leq k \leq m$ and $P_0 \in \mathbb{C}$. As before, $\mathcal{P}_{\leq m}(X)$ together with the supremum
norm (of $P$ on $B_X$) leads to a~Banach space.

We recall the notion of usually called polarization constants, which appear naturally  when relating homogeneous polynomials with their associated
symmetric multilinear forms. From  \eqref{formula de polarizacion} we easily obtain the  following polarization inequality
\begin{equation*}
\Vert \overset{\vee}{P} \Vert \leq \frac{m^m}{m!} \Vert P \Vert\,,\,\,\,\,\,\,\,\,P \in \mathcal P_m(X)\,.
\end{equation*}

Given  a finite dimensional Banach space $X$, we at a few occasions  need to represent  $\mathcal L_m(X)$ and $\mathcal P_m(X)$
as injective tensor products. Recall that in this case the canonical identities
\[
\mathcal{L}_m(X) = \otimes_\varepsilon^m X^\ast
\,\,\,\,\,\,\,\,\,\,\, \text{ and } \,\,\,\,\,\,\,\,\,\,\,
\mathcal{P}_m(X) = \otimes_{\varepsilon_s}^m X^\ast
\]
hold algebraically as well as isometrically, where $\otimes_\varepsilon^m$ stands for the $m$th full injective tensor product and $\otimes_{\varepsilon_s}^m$ for the $m$th symmetric  injective tensor product (see, e.g., \cite{floret1997natural}). We also need
the $m$th full projective  tensor product $\otimes_\pi^m$ (resp., the $m$th symmetric  projective tensor product $\otimes_{\pi_s}^m$),
which is dual to $\otimes_\varepsilon^m$ (resp., $\otimes_{\varepsilon_s}^m$).

\vspace{15 mm}

Important geometric concepts from the theory of Banach lattices will be used, and we will repeat some of them here. For more details, we refer to \cite{LT1, LT}.

A~Banach function lattice $X$ over a~measure space $(\Omega, \mathcal{A}, \mu)$ is said to be $p$-convex, $1 \leq p \leq \infty$, respectively $q$-concave,
$1 \leq q \leq \infty$, if there is a~constant $C>0$ such that for every choice of finitely many $x_1,\ldots, x_N \in~X$
\begin{align*}
\Big\|\Big(\sum_{k=1}^{N} |x_k|^p \Big)^{1/p}\Big\|_X \le C \, \Big (\sum^{N}_{k=1}\|x_k\|_X^p \Big)^{1/p}\,,
\end{align*}
resp.,
\begin{align*}
\Big (\sum^{N}_{k=1}\|x_k\|_X^q \Big) ^{1/q}
\le C \, \Big\|\Big(\sum_{k=1}^{N}|x_k|^q \Big)^{1/q}\Big\|_{X}
\end{align*}
(with the usual modification whenever $p= \infty$ or $q= \infty$). We define the $p$-convexity constant $M^{(p)}(X)$ (resp.,
$q$-concavity constant $M_{(p)}(X)$) to be the least constant $C>0$ satisfying the above inequality. In case that $X$ is not
$p$-convex (resp., not $q$-concave), then we write $M^{(p)}(X) = \infty$ (resp., $M_{(q)}(X)~=~\infty$).
Clearly, every Banach function lattice $X$ is $1$-convex and $\infty$-concave (with constants $1$). We also note that, if
$r < p < s$ and  $X$ is $p$-convex (resp., $p$-concave), then $X$ is $r$-convex (resp., $s$-concave) with $M^{(r)}(X) \le M^{(p)}(X)$
(resp., $M_{(s)}(X) \le M_{(p)}(X)$). For details we refer to \cite[Proposition 1.d.5]{LT1}. If a~Banach lattice is $r$-convex,
for some $1<r<\infty$, then it is $p$-convex, respectively $r$-concave implies $q$-concave, for every $1 < p < r < q <\infty$
(see \cite[Theorem 1.f.7]{LT1}).

(Quasi-)~Banach sequence lattices $X$ are of special interest for our purposes. We recall that a (quasi-) Banach function lattice
over the counting measure space $\big(\Omega, 2^{\Omega}, \mu\big)$ is called a~(quasi-)~Banach sequence lattice over $\Omega$.  A~Banach
sequence lattice $X$ is said to be maximal (or $X$ has the Fatou property) whenever the closed unit ball of $X$ is closed in the
topology of pointwise convergence.

We are mainly interested in the case $\Omega:=\{1, \ldots, n\}$ and $\Omega:=\mathbb{N}$. The fundamental function of a~(quasi-) Banach
sequence lattice $X$ is defined by
\[
\varphi_X(k):= \Big\|\sum_{j=1}^k e_j\Big\|_X \quad , \, k\in \mathbb{\Omega}\,.
\]
A (quasi-)~Banach sequence lattice $X$ is said to be symmetric whenever we have that $(x_{\sigma(n)}) \in X$ with $\|x_\sigma\|_X =\|x\|_X$
for every $x = (x_n)\in X$ and every permutation $\sigma: \mathbb{N} \to \mathbb{N}$.

Notice that if $X$ is a~separable Banach sequence lattice over $\mathbb{N}$, then the Banach dual $X^{*}$ can be identified in
a~natural way with the K\"othe dual $X'$ of $X$,that is,
\begin{align*}
X' := \bigg\{(x_k): \, \sum_{k=1}^{\infty}\, |x_k y_k| < \infty \, \, \, \, \,\mbox{for all \,\,\, $(y_k) \in X$} \bigg\}
\end{align*}
equipped with the norm
\[
\|(x_k)\|_{X'} = \sup\bigg\{\sum_{k=1}^{\infty}\,|x_k y_k|: \, \, \|(y_k)\|_X \leq 1 \bigg\}\,.
\]
Clearly, $X'$ is a~maximal Banach sequence lattice. We also note that  a Banach sequence lattice $X$
is  maximal if and only if  $X \equiv X''$.

Given two (quasi-) Banach sequence lattices $X,Y$, one may define  the pointwise product $x y~=~(x_k y_k)$ of $x = (x_k) \in X $  and
$y  = (y_k) \in Y$. This leads to the definition of the pointwise product space
$$X \circ Y : =\big\{ x y : \,\, x \in X, \,\, y \in  Y \big\}\,,$$
equipped with the (quasi-) norm $\|z\|_{X \circ Y}:= \inf \big\{\|x\|_{X} \|y\|_{Y}: \, z = x y,\,\, x\in X,\,\, y \in Y \big\}\,$.
By $M(X, Y)$ we denote the space of all multipliers from $X$ to $Y$, that is, all sequences $ \xi$ such that $ \xi x \in Y$
for all $x \in X$, equipped with the (quasi)-norm $\| \xi \|_{M(X, Y)} := \sup \big\{\| \xi x\|_{Y}:\, \|x\|_{X} \leq 1\big\}\,)$.

If $X$ is a (quasi-) Banach sequence lattice over $\mathbb{N}$ and $n \in \mathbb{N}$, then
\[
\|z\|:= \Big\|\sum_{k=1}^n |z_k| e_k\Big\|_{X}, \quad\, z=(z_1,\ldots, z_n)\in \mathbb{C}^n
\]
defines  a~lattice (quasi-) norm on $\mathbb{C}^n$, where $e_k$ stands for $k$th standard unit vectors of $c_0$. This defines the
(quasi-)~Banach sequence lattice
\[
X_n = (\mathbb{C}^n, \|\cdot\|)\,,
\]
which we call the $n$th section of $X$. Obviously, the collection $\{e_k\}_{k=1}^n$ of all unit basis vectors forms
a~$1$-unconditional basis of $X_n$.

Note that the projection constant is bounded by the Banach-Mazur distance to $\ell_\infty^n,$ i.e. $\boldsymbol{\lambda}(X_n) \leq d\big(X_n, \ell_\infty^n \big)$,
for any $n$-dimensional Banach lattice $X_n$, so it follows that
\begin{equation} \label{Carsten1}
\boldsymbol{\lambda}(X_n) \leq d(\ell^n_\infty, X_n) \leq \varphi_{X_n}(n), \quad\, n\in \mathbb{N}\,,
\end{equation}
provided that $\|\id\colon X_n \to \ell_{\infty}^n\| \leq 1$, or equivalently  $\|e_k\|_{X_n}\leq 1\,, \, 1 \leq k \leq n$.
Conversely, Sch\"{u}tt proved in \cite{schutt1978projection} that
\begin{equation} \label{schuett}
\varphi_{X_n}(n) \leq \sqrt{2} \, \big\|\id\colon \ell_{2}^n \to  X_n\big\| \, \boldsymbol{\lambda}(X_n)\,.
\end{equation}

\section{Unconditionality, projection and Gordon-Lewis constants} \label{Unconditionality}

Given a~Banach sequence lattice $X_n:=(\mathbb{C}^n, \|\cdot\|)$ and a finite  index set $J \subset \mathbb{N}_0^{(\mathbb{N})}$,
our goal in this section is to estimate the projection and unconditionality constant of the Banach space $\mathcal{P}_{J}(X_n)$.
These estimates will be expressed as functions of the dimension $n$ of the space $X_n$  and the degree
$m = \max\{|\alpha|: \, \alpha \in J\}$ of the set $J$. To achieve this, we will explore the relationships between the invariants
of the local theory of Banach spaces as discussed in the Introduction.

\subsection{Building blocks}
We start by examining how the projection constant varies when considering subspaces derived from corresponding subsets of indices.

Let $X_n = (\mathbb C^n, \|\cdot\|)$ be a Banach space. Then for each pair of  (finite) index sets $I,J \subset  \mathbb{N}_0^n$
with $I \subset J$ we define the projection
\begin{equation} \label{anni}
\mathbf{Q}_{J,I}:\mathcal{P}_J(X_n) \to \mathcal{P}_I(X_n)\,,\,\,\,
P \mapsto \sum_{\alpha \in  I} c_\alpha(P) z^\alpha\,.
\end{equation}

This projection annihilates the coefficients of a polynomial in $\mathcal{P}_J(X_n)$ whose indices are in the complement of $I$.

\begin{remark} \label{simple}
Given $I,J \subset  \mathbb{N}_0^n$ with $I \subset J$ and a  Banach space $X_n= (\mathbb{C}^n, \|\cdot\|)$, we have
\[
\boldsymbol{\lambda}\big(\mathcal{P}_{I}(X_n)
\big) \leq \|\mathbf{Q}_{J,I}\, \| \boldsymbol{\lambda}\big(\mathcal{P}_J(X_n)\big)\,.
\]
Indeed, factorize $\id_{\mathcal{P}_I(X_n)} = \mathbf{Q}_{J,I}\circ j_{I,J}$ \, through the\,canonical embedding
$j_{I,J}\colon \mathcal{P}_I(X_n)\hookrightarrow \mathcal{P}_J(X_n)$ and \,the
projection $\mathbf{Q}_{J,I}: \mathcal{P}_J(X_n)\hookrightarrow \mathcal{P}_I(X_n)$, and use \eqref{gammainfty}.
\end{remark}

We will start with a simple but crucial case for further discussion

\begin{proposition} \label{prop: homogeneous part proj constant}
Let $J \subset \mathbb{N}_0^n$ be an index set of degree $m$. Then, for each $0 \leq k \leq m$ and each Banach space
$X_n = (\mathbb{C}^n, \|\cdot\|)$,
\[
\big\|\mathbf{Q}_{J,J(k)}\colon \mathcal{P}_J(X_n) \to \mathcal{P}_{J(k)}(X_n)\big\| =1\,,
\]
so in particular
\[
\boldsymbol{\lambda}\big(\mathcal{P}_{J(k)}(X_n)\big) \leq  \boldsymbol{\lambda}\big(\mathcal{P}_J(X_n)\big)\,.
\]
\end{proposition}

\begin{proof}
Given  $P \in \mathcal{P}_J(X_n)$ consider
$
P = \sum_{k=0}^{m} \mathbf{Q}_{J,J(k)}(P)$, the unique Taylor expansion of $P$. Cauchy's inequality gives the bound
$\|\mathbf{Q}_{J,J(k)}(P)\| \leq \|P\|$ for all $0 \leq k \leq m$ (see, e.g., \cite[Proposition 15.33]{defant2019libro}).
\end{proof}

Less standard is the result we now prove for subsets of tetrahedral indices $\alpha$, i.e., $\alpha \in \Lambda_T( \leq m,n)$.
The following tool is crucial for our purposes and is essentially due to Ortega-Cerd\`a, Ouna\"{\i}es and Seip. Since this result
appears in the unpublished manuscript \cite{ortega2009sidon} and we need a slight improvement, we provide a detailed proof for
completeness. For clarity, we will first define the relevant constant

\begin{equation*}
\kappa:=\bigg(\prod_{k=1}^\infty\ \sinc{\frac{\pi}{\mathfrak{p}_k}}\bigg)^{-1}=2.209\ldots\,,
\end{equation*}
where  $(\mathfrak{p}_n)_{n=1}^{\infty}$ stands for the  sequence of prime numbers and $\sinc x:=(\sin x)/x$.

\begin{theorem}\label{OrOuSe}
Let $X = (\CC^n, \| \cdot \|)$ be a Banach lattice  and  $J \subset \mathbb{N}_0^n$ an  index set of degree at most $m$. Then
\[
\big\|\mathbf{Q}_{J,J_T}\colon \mathcal{P}_J(X_n) \to \mathcal{P}_{J_T}(X_n)\big\|\,
\leq \kappa^m\,,
\]
where $J_T = J \cap \Lambda_T(\le\! m,n)$. In particular,
\[
\boldsymbol{\lambda}\big(\mathcal{P}_{J_T }(X_n)\big) \leq \kappa^m \boldsymbol{\lambda}\big(\mathcal{P}_J(X_n)\big)\,.
\]
\end{theorem}

\begin{proof}

As usual, we write $\pi(x)$ for the counting function of the prime numbers. Now, given
\[
t=(t_1,\ldots,t_{\pi(m)}) \in  \pmb{R} :=[0,1]^{\pi(m)}\,,
\]
define
\[
r_m(t)=c_m \exp \left(2\pi i\Bigl(\frac {t_1}2 + \frac {t_2}3+\cdots+ \frac
{t_{\pi(m)}}{\mathfrak{p}_{\pi(m)}}\Bigr)\right)\,,
\]
where
\[
c_m=\prod_{k=1}^{\pi(m)} \left(\frac {\mathfrak{p}_k}{2\pi i}
\Bigl(e^{\frac {2\pi i}{\mathfrak p_k}} -1\Bigr)\right)^{-1}\,.
\]
Note that the function $r_m:\pmb{R}\to \mathbb{C}$ has the following properties:
\begin{enumerate}
\item[(i)] $\int_{\pmb{R}} r_m(t)\,d\mu(t)=1$,
\item[(ii)] $\int_{\pmb{R}} r_m^k(t)\, d\mu(t)=0$\ \ for each \, $2 \leq k \leq m$,
\item[(iii)] $|r_m(t)|\le \kappa$\ \ for all $t \in \pmb{R}$.
\end{enumerate}
Here $\mu$ denotes the  Lebesgue measure on $\pmb{R}$. Indeed, (i) and (ii) are trivial and follow by the definition
of the function, and (iii) holds because $|r_m(t)| = |c_m|$ and
\[
|c_m|^{-2}= \prod_{k=1}^{\pi(m)} \frac{\mathfrak{p}_k^2}{(2\pi)^2} \Bigl|e^{\frac{2\pi i}{\mathfrak{p}_k}}-1\Bigr|^2
 = \prod_{k=1}^{\pi(m)} \sinc^2\frac{\pi}{\mathfrak{p}_k}\,.
\]
Given a polynomial $P \in \mathcal P_J(X_n)$, note that by the properties (i) and (ii) we have the representation
\[
\mathbf{Q}_{J,J_T}P(z)=\int_{\pmb{R}^n} P(z_1r_m(t^1),\ldots, z_n r_m(t^n))\,
d\mu(t^1)\cdots d\mu(t^n)\,,\,\,\, z \in X_n\,.
\]
Since $X_n$ is a~Banach lattice we by (iii) deduce that  $|P(z_1 r_m(t^1),\ldots, z_n r_m(t^n))|
\le \kappa^m \|P\|_{\mathcal{P}_{J}(X_n)}$ for every $z~\in~B_{X_n}$, and therefore
\[
\Vert \mathbf{Q}_{J,J_T}(P) \Vert_{\mathcal{P}_{\Lambda_T}(X_n)} \leq \kappa^m \,\Vert P \Vert_{\mathcal{P}_{J}(X_n)}\,.
\]
This proves the first statement, the second one is an immediate consequence of the observation from Remark~\ref{simple}.
\end{proof}

Observe that by the Kadets-Snobar theorem from \eqref{kadets1} we know that for any Banach space $X_n = (\mathbb{C}^n, \|\cdot\|) $
and any finite index set $J \subset \mathbb{N}_0^n $ we have
\begin{equation*}\label{tincho}
\boldsymbol{\lambda}\big(\mathcal{P}_{J}(X_n)\big) \leq \sqrt{|J|} \leq \sqrt{ m+1} \max_{0 \leq k \leq m} \sqrt{|J(k)|}\,.
\end{equation*}
We can now provide estimates of the projection constant of $\mathcal{P}_{J}(X_n)$ from above or below by the projection constants
of its 'homogeneous building blocks' $\mathcal{P}_{J(k)}(X_n), \, 0 \leq k \leq m$.

\begin{theorem} \label{degree-homo}
Let $X_n = (\mathbb{C}^n, \|\cdot\|) $ be a Banach space, and let $J \subset \mathbb{N}_0^n$ be an index set of degree at most $m$. Then
\[
\max_{0 \leq k \leq m} \boldsymbol{\lambda}\big(\mathcal{P}_{J(k)}(X_n)\big) \,\,\leq \,\,
\boldsymbol{\lambda}\big(\mathcal{P}_{J}(X_n)\big)
\,\,\leq \,\,(m+1) \max_{0 \leq k \leq m} \boldsymbol{\lambda}\big(\mathcal{P}_{J(k)}(X_n)\big)\,.
\]
In particular, for any  Banach sequence lattice $X$
\[
\lim_{m \to \infty}  \sup_{n\in \mathbb{\mathbb{N}}}
\frac{\sqrt[m]{\boldsymbol{\lambda}\big(\mathcal{P}_{\leq m}(X_n)\big)}}{\sqrt[m]{\max_{0 \leq k \leq m} \boldsymbol{\lambda}\big(\mathcal{P}_{k}(X_n)\big)}} = 1\,.
\]
\end{theorem}

\begin{proof}
Note first that by Proposition~\ref{prop: homogeneous part proj constant} for all $0 \leq k \leq m$ it holds
$ \boldsymbol{\lambda}\big(\mathcal{P}_{J(k)}(X_n)\big) \,\,\leq \,\,
\boldsymbol{\lambda}\big(\mathcal{P}_{J}(X_n)\big)$, so that it remains  to  check the second estimate.
We  use  that each $P \in \mathcal{P}_{J}(X_n)$ has a unique Taylor series expansion
$P = \sum_{k=0}^m P_k$ with $P_k \in \mathcal{P}_{J(k)}(X_n)$, and from the Cauchy inequality we  know that
$\|P_k\| \leq \|P\| $ for all $0 \leq k \leq m$. Consequently,  the two operators
\begin{align*}
&
U\colon \mathcal{P}_{J}(X_n) \to  \bigoplus_\infty \mathcal{P}_{J(k)}(X_n)\,,\,\,\,\, P \mapsto (P_k)_{k=1}^m
\\&
V\colon  \bigoplus_1 \mathcal{P}_{J(k)}(X_n) \to \mathcal{P}_{J}(X_n) \,,\,\,\,\, (Q_k)_{k=1}^m  \mapsto \sum_{k=1}^m Q_k\,,
\end{align*}
 both have norms less than or equal to $1$.
 Now fix some $\varepsilon >0$, and choose for each $1 \leq k \leq m$ an appropriate  factorization
\[
\xymatrix
{\mathcal{P}_{J(k)}(X_n) \ar[r]^{\id} \ar[d]_{u_k}
& \mathcal{P}_{J(k)}(X_n) \\
\ell_\infty^{M_k} \ar[ur]_{v_k}}
\]
such that $\|u_k\|\leq 1$ and $\|v_k\|\leq (1 + \varepsilon) \boldsymbol{\lambda} \big( \mathcal{P}_{J(k)}(X_n) \big)$.
Then we arrive at  the following commutative diagram
\[
\xymatrix
{\mathcal{P}_{J}(X_n) \ar[r]^{\id} \ar[d]_{U}
& \mathcal{P}_{J}(X_n)
\\
\bigoplus_\infty \mathcal{P}_{J(k)}(X_n)
\ar[d]_{\bigoplus u_k}
&
\bigoplus_1 \mathcal{P}_{J(k)}(X_n) \ar[u]_{V}
\\
\bigoplus_\infty \ell_\infty^{M_k} \ar[r]^{\Phi}
&
\bigoplus_1 \ell_\infty^{M_k}  \ar[u]_{\bigoplus v_k}\,,
}
\]
where $\Phi$ stands for the identity map which here obviously has norm $\leq m+1$.
But
\[
\big\|\bigoplus u_k\big\| \leq \max{\|u_k\|} \leq 1\,,
\]
as well as
\[
\big\|\bigoplus v_k\big\| \leq \max{\|v_k\|} \leq (1+\varepsilon)
\max_{0 \leq k \leq m} \boldsymbol{\lambda}\big(\mathcal{P}_{J(k)}(X_n)\big)\,.
\]
This finally gives
\begin{align*}
\boldsymbol{\lambda}\big(\mathcal{P}_{J}(X_n)\big)
\leq \|U\| \,\big\|\bigoplus u_k\big\| \,\,\boldsymbol{\lambda}\big(\bigoplus_\infty \ell_\infty^{M_k}\big)\,\,
\|\Phi\|\,\big\|\bigoplus v_k\big\| \, \|V\|
\leq (m+1) \max_{0 \leq k \leq m} \boldsymbol{\lambda}\big(\mathcal{P}_{J(k)}(X_n)\big)\,,
\end{align*}
the conclusion.
\end{proof}




\subsection{The Gordon-Lewis bridge}

In this subsection, we use the Gordon-Lewis constant to bridge the gap between the unconditionality and projection constants.
We start with the following result which relates the Gordon-Lewis constant ${\boldsymbol{g\!l}}\big(\mathcal{P}_{J}(X_n)\big)$
with the unconditional basis constant $\boldsymbol{\chimon}\big( \mathcal{P}_{J}(X_n) \big)$ of the monomial basis
$(z^\alpha)_{\alpha \in J}$ of $\mathcal{P}_{J}(X_n)$.

\begin{theorem} \label{gl-versus-unc}
Let $X_n = (\mathbb{C}^n,\|\cdot\|)$  be a~Banach  lattice, and  let $J\subset~\mathbb{N}_0^n$ be a~finite index
set of degree at most $m$. Then
\[
\boldsymbol{g\!l}\big( \mathcal{P}_{J}(X_n)\big) \,\leq\,\boldsymbol{\chi}\big(\mathcal{P}_{J}(X_n) \big)
\,\leq\, \boldsymbol{\chimon}\big( \mathcal{P}_{J}(X_n) \big) \le 2^m \boldsymbol{g\!l}\big( \mathcal{P}_{J}(X_n) \big)\,.
\]
\end{theorem}

When $J = \Lambda(m,n)$, we recover the result stated in \cite[Proposition3.1]{defant2011bohr}. Theorem \ref{gl-versus-unc}
extends Theorem~21.11 from \cite{defant2019libro} to any index set $J$ with degree at most $m$; its proof can be adapted almost
verbatim from the one provided there. Note that the first estimate in Theorem \ref{gl-versus-unc} follows directly from the
Gordon-Lewis inequality \eqref{gl-inequality}, while the second estimate is straightforward.

\bigskip
The following result estimates the Gordon-Lewis constant of $\mathcal{P}_{J}(X_n)$, where $J$ is a finite  index set of degree
$m$, in terms of  the projection constants of the spaces $\mathcal{P}_{J(k)^\flat}(X_n),\, 1 \leq  k \leq m$.

\begin{theorem} \label{gl_versus_proj}
Let $X_n = (\mathbb{C}^n,\|\cdot\|)$  be a Banach  lattice and  let an index set $J \subset \Lambda(m,n)$. Then
\[
\boldsymbol{g\!l}\big( \mathcal{P}_{J}(X_n)\big) \,\le \,e  \|\mathbf{Q}_{\Lambda(m,n),J}\|\,\,
\boldsymbol{\lambda}\big( \mathcal{P}_{J^\flat}(X_n)\big)\,.
\]
Moreover, if $J\subset~\mathbb{N}_0^n$ is an index set of degree  $m$, then
\[
\boldsymbol{g\!l}\big( \mathcal{P}_{J}(X_n)\big) \le e (m+1) \,\max_{1 \leq  k \leq m}\|\mathbf{Q}_{\Lambda(k,n),J(k)}\|\,\,
\max_{1 \leq  k \leq m} \boldsymbol{\lambda}\big( \mathcal{P}_{J(k)^\flat}(X_n)\big)\,.
\]
\end{theorem}

The case $J= \Lambda(m,n)$ was previously established in \cite[Proposition 4.2]{defant2011bohr}, and a detailed alternative
presentation of the proof for this homogeneous case is provided in \cite[Proposition 22.1]{defant2019libro}. In this context,
we will show how to adapt those arguments to tackle the general situation, as it involves more complex technical subtleties.

We start with an  elementary  observation  taken from \cite[Lemma 22.2]{defant2019libro}, which up to polarization covers the
case $J=\Lambda(2,n)$ of Theorem~\ref{gl_versus_proj} (take $Y= X^\ast$): For every finite dimensional Banach lattice~$X$,
and every finite dimensional Banach space $Y$ one has
\begin{equation}\label{supo}
\boldsymbol{g\!l}\big( \mathcal{L}(X,Y)\big) \leq \boldsymbol{\lambda}(Y)\,.
\end{equation}

We now proceed with the proof of Theorem~\ref{gl_versus_proj}. For the first part of the statement, we will factorize the identity of $\mathcal{P}_{J}(X_n)$ through $\mathcal{L}\big(X_n, \mathcal{P}_{J^\flat}(X_n)\big)$ and apply the previous inequality. In the second part, we will decompose into homogeneous indices and use the factorization from the first part for each homogeneous piece.

\begin{proof}[Proof of Theorem~\ref{gl_versus_proj}]
To see the first statement, we  consider the following commutative diagram:
\begin{equation}\label{diagram}
\xymatrix
{
\mathcal{P}_{J}(X_n) \ar[rr]^{id} \ar[d]^{\text{$U_m$}} &  & \mathcal{P}_{J}(X_n)\\
\mathcal{L}\big(X_n,\mathcal{P}_{J^\flat}(X_n)\big)
\ar[r]_{I_m\;\;\;\;\;\;\;} & \mathcal{L}\big(X_n,\mathcal{P}_{m-1}( X_n)\big) \ar[r]_{\;\;\;\;\;\;\;\;\;\;\;\;\;\;V_m}
& \mathcal{P}_{m}(X_n), \ar[u]^{\mathbf{Q}_{\Lambda(m,n),J}}\,,\\
}
\end{equation}
where $I_m$ is the canonical inclusion map and
\begin{align*}
&
\big(U_m(P)x\big)(u) := \check{P}( u, \ldots, u
,x)
\,\,\, \text{ for  } \,\,\, P \in  \mathcal{P}_{J}(X_n) \,\,\, \text{ and  } \,\,\, x,u \in X_n\,,
\\[1ex]&
V_m(T)(y) := (Ty)y \,\,\, \text{ for  } \,\,\,  T \in \mathcal{L}\big(X_n,\mathcal{P}_{m-1}(X_n)\big)
\,\,\, \text{ and  } \,\,\, y \in X_n.
\end{align*}
We show that $U_m$, as an operator from $\mathcal{P}_{J}(X_n)$ into $\mathcal{L}\big(X_n, \mathcal{P}_{J^\flat}(X_n)\big)$,
is well-defined. Indeed,  define
\[
a_\bi(\check{P})= \frac{c_\bj(P)}{|[\bj]|}\,\,\,\, \text{ for $\bj \in \mathcal{J}(m,n)$ and $\bi \in [\bj]$} \,.
\]
Then, given $P \in  \mathcal{P}_{J}(X)$ and $x,u \in X_n$,
\begin{align*}
\check{P}( u, \ldots, u,x)
&
=\sum_{\bi \in \mathcal{M}(m,n)} a_\bi(\check{P})
u_{i_1}\ldots u_{i_{m-1}} x_{i_m}
\\&
=\sum_{\bi \in \mathcal{M}(m-1,n)} \sum_{\ell=1}^n a_{(\bi,\ell)}(\check{P})
u_\bi x_{\ell}
=
\sum_{\bj \in \mathcal{J}(m-1,n)} \sum_{\bi \in  [\bj]} \,\, \sum_{\ell=1}^n a_{(\bi,\ell)}(\check{P})
u_\bi x_{\ell}
\\&
=
\sum_{\bj \in \mathcal{J}(m-1,n)}
\sum_{\ell=1}^n
\,\,
\sum_{\bi \in  [\bj]}
 a_{(\bi,\ell)}(\check{P})
 u_\bi x_{\ell}
   =
\sum_{\bj \in \mathcal{J}(m-1,n)}
\sum_{\ell=1}^n
\,\,
\sum_{\bi \in  [\bj]}
 \frac{c_{(\bi,\ell)_\ast}(P)}{|[(\bi,\ell)_\ast]|}
  u_\bi x_{\ell}
   \\&
  =
\sum_{\bj \in \mathcal{J}(m-1,n)}
\sum_{\ell=1}^n
\bigg[
 \frac{c_{(\bj,\ell)_\ast}(P)}{|[(\bj,\ell)_\ast]|}  u_\bj \bigg] |[\bj]|
 x_{\ell}
     =
\sum_{\bj \in \mathcal{J}(m-1,n)}
\bigg[
\sum_{\ell=1}^n
  \frac{c_{(\bj,\ell)_\ast}(P)}{|[(\bj,\ell)_\ast]|} |[\bj]|
  x_{\ell} \bigg]  u_\bj
  \\&
  =
\sum_{\bj \in \mathcal{J}(m-1,n)}
\bigg[\sum_{\substack{1 \leq \ell \leq n \\ (\bj,\ell)_\ast \in J}}
 \frac{c_{(\bj,\ell)_\ast}(P)}{|[(\bj,\ell)_\ast]|} |[\bj]|
  x_{\ell} \bigg]  u_\bj
    =
\sum_{\bj \in J^\flat}
\bigg[\sum_{\substack{1 \leq \ell \leq n \\ (\bj,\ell)_\ast \in J}}
 \frac{c_{(\bj,\ell)_\ast}(P)}{|[(\bj\grave{},\ell)_\ast]|} |[\bj]|
  x_{\ell} \bigg]  u_\bj\,,
  \end{align*}
which shows that $U_m(P)x \in \mathcal{P}_{ J^\flat}(X_n)$ for every $x \in X_n$.
By the Harris polarization formula (see, e.g., \cite[Proposition 2.34]{defant2019libro}) we have $\|U_m\| \leq e$,
and moreover trivially  $\|V_m\| \leq~1$. Hence by the ideal properties of the involved ideal norms we see that
\[
\boldsymbol{g\!l}\big( \mathcal{P}_{J}(X_n)\big) \leq e \, \|\mathbf{Q}_{\Lambda(m,n),J}\| \,
\boldsymbol{g\!l}\big(\mathcal{L}\big(X_n,\mathcal{P}_{J^\flat}(X_n)\big)\big) \leq  e \, \|\mathbf{Q}_{\Lambda(m,n),J}\|\, \boldsymbol{\lambda}\big(\mathcal{P}_{J^\flat}(X_n)\big)\,,
\]
where for the last estimate we use  \eqref{supo}. This proves  the first claim.

For the second claim we have   to handle an  index set $J$ of degree $m$, and consider the following commutative diagram
\begin{equation*}
    \xymatrix
{
\mathcal{P}_{J}(X_n)
  \ar[d]^{\text{$\mathbf{O}\oplus\bigoplus\mathbf{Q}_{J,J(k)}$}}
  \ar[r]^{\text{$ \id_{\mathcal{P}_{J}(X_n)}$}}
 & \mathcal{P}_{J}(X_n)
 \\
  \mathbb{C} \oplus_\infty \bigoplus_\infty \mathcal{P}_{J(k)}(X_n)
  \ar[d]^{\text{$ \id_\mathbb{C}\oplus
  \bigoplus U_k$}}
  \ar[r]^{\text{$\id_\mathbb{C}\oplus\bigoplus \id_{ \mathcal{P}_{J(k)}(X_n)}$}}
 &
 \mathbb{C} \oplus_1\bigoplus_1 \mathcal{P}_{J(k)}(X_n) \ar[u]^{\sum }  & \mathbb{C} \oplus_1\bigoplus_1 \mathcal{P}_{k}(X_n) \ar[l]_{\text{$\id_\mathbb{C}\oplus\bigoplus\mathbf{Q}_{\Lambda(k,n),J(k)}$}}
 \\
 \mathbb{C} \oplus_\infty\bigoplus_\infty  \mathcal{L}\big(X_n,\mathcal{P}_{J(k)^\flat}(X_n)\big)
\ar[r]^{\id_\mathbb{C}\oplus\bigoplus I_k}
&
 \mathbb{C} \oplus_\infty\bigoplus_\infty \mathcal{L}\big(X_n,\mathcal{P}_{k-1}(X_n)\big) \ar[r]^{\Phi}  & \mathbb{C} \oplus_1\bigoplus_1 \mathcal{L}\big(X_n,\mathcal{P}_{k-1}(X_n)\big)
 \ar[u]^{\text{$\id_\mathbb{C}\oplus\bigoplus V_k$}}\,.
}
\end{equation*}

Let us explain our notation in this diagram: Here  $U_k$, $V_k$ and $I_k$ for $1 \leq k \leq m$ are the operators from \eqref{diagram}.
If $P = a_0 +\sum_{k=1}^m P_k$ is the Taylor decomposition of $P \in \mathcal{P}_{J}(X_n)$, then
$\mathbf{O}(P) = a_0$, and hence  $\big(\mathbf{O}\oplus\bigoplus\mathbf{Q}_{J,J(k)}\big)(P) = \big(a_0, (P_k)_{k=1}^m\big)$.
Additionally, $\sum$ is the mapping which assigns to every $\big(a_0, (P_k)_{k=1}^m\big)$ the polynomial $P = a_0 +\sum_{k=1}^m P_k$,
and $\Phi$ stands for the identity map -- whereas the notation for the rest of the maps is self-explaining. Obviously, this gives that
\[
\boldsymbol{g\!l}\big( \mathcal{P}_{J}(X_n)\big)
\le e (m+1) \,\max_{1 \leq  k \leq m}\|\mathbf{Q}_{\Lambda(k,n),J(k)}\|\,\,
\max_{1 \leq  k \leq m} 
\boldsymbol{g\!l}\big( \bigoplus_\infty  \mathcal{L}\big(X_n,\mathcal{P}_{J(k)^\flat}(X_n)\big)\big)\,,
\]
so it remains to prove the following claim
\begin{equation*}
\boldsymbol{g\!l}\big( \bigoplus_\infty  \mathcal{L}\big(X_n,\mathcal{P}_{J(k)^\flat}(X_n)\big)\big)
\leq \max_{1 \leq  k \leq m} \boldsymbol{\lambda}\big( \mathcal{P}_{J(k)^\flat}(X_n)\big)\,.
\end{equation*}
Indeed, using standard properties of $\varepsilon$- and $\pi$-tensor products, we have
\begin{align*}
\bigoplus_\infty  \mathcal{L}\big(X_n,\mathcal{P}_{J(k)^\flat}(X_n)\big)
& \hookrightarrow \bigoplus_\infty  \mathcal{L}\big(X_n,\bigoplus_\infty \mathcal{P}_{J(k)^\flat}(X_n)\big) \\
& = \ell_\infty^m \otimes_\varepsilon \big[ X_n^\ast \otimes_\varepsilon \bigoplus_\infty \mathcal{P}_{J(k)^\flat}(X_n)\big] \\
& = \big[\ell_\infty^m \otimes_\varepsilon  X_n^\ast\big] \otimes_\varepsilon \bigoplus_\infty \mathcal{P}_{J(k)^\flat}(X_n) \\
& = \big(\ell_1^m \otimes_\pi X_n\big)^\ast   \otimes_\varepsilon \bigoplus_\infty \mathcal{P}_{J(k)^\flat}(X_n)
= \mathcal{L}\big(  \ell_1^m(X_n), \bigoplus_\infty \mathcal{P}_{J(k)^\flat}(X_n)\big)\,,
\end{align*}
where the first space in fact is $1$-complemented in the second one, and all other identifications are isometries.
Then we deduce from \eqref{supo} that
\[
\boldsymbol{g\!l}\big( \bigoplus_\infty  \mathcal{L}\big(X_n,\mathcal{P}_{J(k)^\flat}(X_n)\big)\big)
\leq \boldsymbol{g\!l}\big(\mathcal{L}\big(\ell_1^m(X_n), \bigoplus_\infty \mathcal{P}_{J(k)^\flat}(X_n)\big)\big)
\leq \boldsymbol{\lambda} \big(\bigoplus_\infty \mathcal{P}_{J(k)^\flat}(X_n) \big)\,.
\]
Since obviously
\[
\boldsymbol{\lambda} \big(\bigoplus_\infty \mathcal{P}_{J(k)^\flat}(X_n) \big)
= \gamma_\infty \big(\id_{\bigoplus_\infty \mathcal{P}_{J(k)^\flat}(X_n)} \big)
\leq \max_{1 \leq  k \leq m} \gamma_\infty \big(\id_{\mathcal{P}_{J(k)^\flat}(X_n)} \big)
= \max_{1 \leq  k \leq m} \boldsymbol{\lambda} \big(\mathcal{P}_{J(k)^\flat}(X_n) \big)\,,
\]
the proof is complete.
\end{proof}

From the Theorems ~\ref{gl-versus-unc} and~\ref{gl_versus_proj}, the following result follows.

\begin{theorem}\label{thm: uncond cte vs proj cte}
Let $X_n = (\mathbb{C}^n,\|\cdot\|)$  be a Banach  lattice, and let $J\subset~\mathbb{N}_0^n$ be an index set of degree at most~$m$. Then
\[
\boldsymbol{\chimon}\big( \mathcal{P}_{J}(X_n)\big)
\,\le \,e (m+1) 2^m \,\max_{1 \leq  k \leq m}\|\mathbf{Q}_{\Lambda(k,n),J(k)}\|\,\,
\max_{1 \leq  k \leq m} \boldsymbol{\lambda}\big( \mathcal{P}_{J(k)^\flat}(X_n)\big)\,.
\]
In addition,
\[
\boldsymbol{\chimon}\big( \mathcal{P}_{J}(X_n)\big) \,\le \,e 2^m \, \|\mathbf{Q}_{\Lambda(m,n),J}\|
\,\, \boldsymbol{\lambda}\big( \mathcal{P}_{J^\flat}(X_n)\big)\,,
\]
whenever  $J \subset \Lambda(m,n)$.
\end{theorem}

\begin{corollary} \label{main3A}
Let $m \in \mathbb{N}$ with $m\geq 2$ and let $X_n = (\mathbb{C}^n,\|\cdot\|)$  be a Banach  lattice. Then
\[
\boldsymbol{\chimon}\big( \mathcal{P}_{m}(X_n)\big) \,\le \, \boldsymbol{\chimon}\big( \mathcal{P}_{\leq m}(X_n)\big)
\,\le \,e (m+1) 2^m  \,\,\max_{1 \leq  k \leq m-1} \boldsymbol{\lambda}\big( \mathcal{P}_{k}(X_n)\big)\,,
\]
and
\[
\boldsymbol{\chimon}\big( \Pp_{\Lambda_T(m)}(X_n)\big) \,\le \,\boldsymbol{\chimon}\big(\Pp_{\Lambda_T(\le m)}(X_n)\big)
 \,\le \,e (m+1) 2^m \kappa^m  \,\,\max_{1 \leq  k \leq m-1} \boldsymbol{\lambda}\big( \Pp_{\Lambda_T(k)}(X_n)\big)\,.
\]
\end{corollary}

The previous results, comparing Gordon-Lewis constants, unconditionality constants, and projection constants, allow us to
transfer the bound provided by the Kadets-Snobar theorem~\eqref{kadets1} for the projection constant to the unconditionality
constant.

\begin{corollary}\label{K-S for poly}
Let  $X$ be a Banach sequence lattice, and let $\big(J_m\big)$ be a~sequence of index sets, each with degree at most $m$.
Then the following estimates hold{\rm:}
\begin{align} \label{mmm}
\boldsymbol{\lambda}\big(\mathcal{P}_{J_m}(X_n)\big)  \prec_{C^m}\Big(1+\frac{n}{m}\Big)^{\frac{m}{2}}
\end{align}
and
\begin{align*}
\boldsymbol{\chimon}\big(\mathcal{P}_{J_m}(X_n)\big)  \prec_{C^m}\Big(1+ \frac{n}{m}\Big)^{\frac{m-1}{2}}\,,
\end{align*}
where $C > 0$ is universal.
\end{corollary}

\begin{proof}
Using \eqref{cardi}, a~simple calculation shows that
\begin{equation*}
|J_m|\le |\Lambda(\leq m,n) | \prec_{C^m} \Big( 1+\frac{n}{m}\Big)^{m}\,.
\end{equation*}
Then estimate \eqref{mmm} follows from the Kadets-Snobar inequality~\eqref{kadets1}. Hence, using Corollary
\ref{main3A}, we conclude that
\[
\boldsymbol{\chimon}\big( \mathcal{P}_{J_m}(X_n)\big)\,\le \,
\boldsymbol{\chimon}\big( \mathcal{P}_{\leq m}(X_n)\big)
\,\le \,e (m+1) 2^m  \,\,\max_{1 \leq  k \leq m-1} \boldsymbol{\lambda}\big( \mathcal{P}_{k}(X_n)\big)\, \prec_{C^m}\Big(1+ \frac{n}{m}\Big)^{\frac{m-1}{2}}\,. \qedhere
\]
\end{proof}

\subsection{Convexity and concavity}

Assuming that the Banach sequence lattice is $2$-convex, we can demonstrate that the previous estimate provides the correct
asymptotics. Before presenting and proving this result, we introduce general estimates that relate the unconditional basis
constants $\boldsymbol{\chimon}\big(\mathcal{P}_{J}(X_n)\big)$ and $\boldsymbol{\chimon}\big(\mathcal{P}_{J}(Y_n)\big)$, provided that $X$
and $Y$ satisfy certain general lattice properties, such as convexity and concavity. These results, combined with those
discussed in the previous subsection, often enable us to obtain accurate asymptotic estimates for the local constants
under consideration for the space of analytic polynomials on finite-dimensional Banach lattices.

Much of what we prove below is grounded in Lozanovski\v{\i}'s profound factorization theorem \cite{Loz}. Since we will need
to reference it again later, we prefer to state it explicitly here. It states that for any maximal Banach function lattice
$X$ over a~complete and $\sigma$-finite measure space $(\Omega, \mathcal{A}, \mu)$, we have $\text{supp}\,X= \Omega$, we
have $X\circ X' \equiv L_1(\mu)$. More precisely, for every $f\in L_1(\mu)$, there exist $g\in X$ and $h\in X'$ such that
\begin{equation*}\label{Lozanovskii}
f= gh \quad\, \text{and} \,\,\,\, \|f\|_{L_1(\mu)} = \|g\|_{X}\,\|h\|_{X'}\,.
\end{equation*}
Clearly, $g$ and $h$ can be chosen positive whenever $f$ is positive. Moreover, if $\|f\|_{L_1(\mu)} =1$,
then we can find $g \in S_{X}$ and $h \in S_{X'}$.

The following lemma serves as our starting point.

\begin{lemma} \label{YM(Y,X)}
Let $X_n = (\mathbb{C}^n, \|\cdot\|)$ and $Y_n = (\mathbb{C}^n, \|\cdot\|)$ be two Banach lattices such that
\[
\|\id \colon X_n \to Y_n\circ M(Y_n, X_n)\| \leq 1\,.
\]
Then, for any index set $J \subset \mathbb{N}_{0}^n$, we have
\begin{align*}
\boldsymbol{\chimon}\big(\mathcal{P}_{J}(X_n)\big) \leq \boldsymbol{\chimon}\big(\mathcal{P}_{J}(Y_n)\big)\,.
\end{align*}
\end{lemma}

\begin{proof}
Fix some $z\in B_{X_n}$. Then by assumption there exist $\xi, v\in \mathbb{C}^n$ such that  $z= v \xi$, $\|v\|_{Y_n} \leq 1$,
and $\|D_\xi \colon Y_n \to X_n\|\leq 1$. Hence, for every $P\in \mathcal{P}_J(X_n)$, we have
\begin{align*}
\Big| \sum_{\alpha \in J} |c_\alpha(P)|\,z^\alpha\Big| = \Big|\sum_{\alpha \in J} |c_\alpha(P\circ D_\xi)|\,v^\alpha\Big|
& \leq \boldsymbol{\chimon}\big(\mathcal{P}_J(Y_n)\big)\,\|P\circ D_\xi\|_{\mathcal{P}_J(Y_n)} \\
& \leq \boldsymbol{\chimon}\big(\mathcal{P}_J(Y_n)\big)\,\|P\|_{\mathcal{P}_J(X_n)}\,,
\end{align*}
which completes the proof.
\end{proof}

An immediate consequence is the following result, which shows that for any set $J\subset \mathbb{N}_{0}^n$, the minimum
(resp., maximum) value of the unconditional basis constant $\boldsymbol{\chimon}\big(\mathcal{P}_{J}(X)\big)$ for the
class of all $n$-dimensional Banach lattices is attained by $\ell_{1}^n$ (resp., $\ell_{\infty}^n$).

\begin{corollary} \label{C1C2}
Let $X_n = (\mathbb{C}^n, \|\cdot\|)$  be a Banach lattice. Then, for any index set $J\subset \mathbb{N}_{0}^n$, the following
estimates holds{\rm:}
\[
\boldsymbol{\chimon}\big(\mathcal{P}_{J}(\ell^n_1)\big) \le
\boldsymbol{\chimon}\big(\mathcal{P}_{J}(X_n)\big) \le \boldsymbol{\chimon}\big(\mathcal{P}_{J}(\ell^n_\infty)\big)\,.
\]
\end{corollary}

\begin{proof}
We apply Lemma~\ref{YM(Y,X)}. Clearly, $\ell^n_\infty \circ M(\ell^n_\infty, X_n ) \equiv X_n$, which  gives the
second inequality. On the other hand, from Lozanowski's theorem, we obtain
\[
X_n  \circ M( X_n ,\ell^n_1 ) \equiv X_n  \circ  X_n'  \equiv\ell^n_1\,,
\]
which leads to the first estimate.
\end{proof}

Under the assumptions of convexity and concavity of the underlying Banach lattice, the previous results can be further improved.

\begin{theorem}
\label{PropAppl}
Let $X_n = (\mathbb{C}^n, \|\cdot\|)$ and $Y_n = (\mathbb{C}^n, \|\cdot\|)$ be two Banach lattices
such that $M_{(r)}(X_n ) = M^{(r)}(Y_n) = 1$, where $1 <r <\infty$. Then, for any index set $J\subset \mathbb{N}_{0}^n$, we have
\begin{equation*}
\boldsymbol{\chimon}\big(\mathcal{P}_{J}(X_n)\big) \leq \boldsymbol{\chimon}\big(\mathcal{P}_{J}(Y_n)\big)\,.
\end{equation*}
\end{theorem}

In view of Lemma~\ref{YM(Y,X)}, the proof is an consequence of \cite[Theorem 3.8,(i)]{schep2010products} which implies
that $X_n \equiv Y_n\circ M(Y_n, X_n)$.

\begin{corollary} \label{C1}
Let $X_n = (\mathbb{C}^n, \|\cdot\|)$ be a Banach lattice such that $M_{(r)}(X_n ) =  1$, where $1 <r <\infty$. Then,
for any index set $J \subset \mathbb{N}_0^n$, we have
\[
\boldsymbol{\chimon}\big(\mathcal{P}_{J}(X_n)) \leq \boldsymbol{\chimon}\big(\mathcal{P}_{J}(\ell_r^n)\big)\,.
\]
\end{corollary}

As already shown in Corollary~\ref{C1C2}, this result also holds for $r = \infty$ (note that $M_{(\infty)}(X_n ) =  1$
for any~$X_n$).

\begin{corollary} \label{C2}
Let $Y_n = (\mathbb{C}^n, \|\cdot\|)$ be a  Banach lattice such that $M^{(r)}(Y_n ) =  1$, where $1 <r <\infty$. Then,
for any index set $J\subset \mathbb{N}_{0}^n$, we have
\[
\boldsymbol{\chimon}\big(\mathcal{P}_{J}(\ell^n_r)\big) \leq \boldsymbol{\chimon}\big(\mathcal{P}_{J}(Y_n)\big)\,.
\]
\end{corollary}

Since  $M^{(1)}(Y_n ) =  1$ holds for any $Y_n$, the case $r=1$ is again covered by Corollary~\ref{C1C2}.

Finally, we show that, if in the above corollaries  the concavity constant of $X_n$ and the  convexity constant of
$Y_n$ differ from  $1$, then, at least in the  homogeneous case (i.e., $J \subset \Lambda(m,n)$ for some $m$), the
estimates hold up to a~constant $C^m$, where $C$ depends on the $r$-concavity and the $r$-convexity constants.

To understand this, we recall  the usually called renorming theorem of Figiel and Johnson (see, e.g., \cite[Proposition~1.d.8]{LT}):
If a~Banach function lattice $X$ on $(\Omega, \mathcal{A}, \mu)$ is $p$-convex and $q$-concave for some $1\leq p\leq q\leq \infty$,
then on $X$ there is an equivalent lattice norm $\|\cdot\|_0$ such that both the $p$-convexity and $q$-concavity constants are
equal to $1$, and furthermore the following inequality holds
\[
\|x\|_X \leq \|x\|_0 \leq M^{(p)}(X)M_{(q)}(X) \|x\|_{X}\,, \quad\, x\in X\,.
\]

\begin{corollary} \label{appl}
Let $X$ and $Y$ be Banach sequence lattices such that  $X$ is $r$-concave and $Y$ is $r$-convex for $1 <r <\infty$. Then, for
any index set $J \subset \Lambda(m,n)$, we have
\[
\boldsymbol{\chimon}\big(\mathcal{P}_{J}(X_n)\big)
\le (M_{(r)}(X) M^{(r)}(Y))^m\,\,\boldsymbol{\chimon}\big(\mathcal{P}_{J}(Y_n)\big)\,.
\]
\end{corollary}

\begin{proof}
Observe first that if $E:= (\mathbb{C}^n, \|\cdot\|_E)$ and $F:=(\mathbb{C}^n, \|\cdot\|_F)$ are
$n$-dimensional Banach lattices such that for some $\gamma \geq 1$
\[
\|z\|_E \le \|z\|_F \leq \gamma \|z\|_E, \quad\, z\in \mathbb{C}^n\,,
\]
then for any polynomial $P\in \mathcal{P}_J(\mathbb{C}^n)$, we have
$\gamma^{-m} \|P\|_{\mathcal{P}_J(E)} \leq \|P\|_{\mathcal{P}_J(F)} \leq \|P\|_{\mathcal{P}_J(E)}$.
Clearly, this yields
\[
\gamma^{-m}\,\boldsymbol{\chimon}\big(\mathcal{P}_J(E)\big) \leq  \boldsymbol{\chimon}\big(\mathcal{P}_J(F)\big)
\leq \gamma^m\,\boldsymbol{\chimon} \big(\mathcal{P}_J(E)\big)\,.
\]
Applying the renorming theorem mentioned above, we conclude that there exist equivalent lattice norms
$\|\cdot\|_{\tilde X}$ and $\|\cdot\|_{\tilde Y}$ on $X$ and $Y$, respectively, such that the Banach lattices
$\widetilde{X}:= (X, \|\cdot\|_{\widetilde{X}})$ and $\widetilde{Y} := (Y, \|\cdot\|_{\widetilde{Y}})$ satisfy
$M_{(r)}(\widetilde{X}) = M^{(r)}(\widetilde{Y})=1$ as well as
\begin{equation*}
\|x\|_X \le \|x\|_{\widetilde{X}} \le M_{(r)}(X) \|x\|_X, \quad\, x\in X\,,\,\,\,\,\,\,
\text{and}\,\,\,\,\,\,\,
\|y\|_Y \leq \|y\|_{\widetilde{Y}} \le M^{(r)}(Y) \|y\|_Y, \quad\, y\in Y\,.
\end{equation*}
The above facts combined with Proposition \ref{PropAppl} yield
\begin{align*}
M_{(r)}(X)^{-m}\,\boldsymbol{\chimon}
\big(\mathcal{P}_J(X_n)\big) & \le  \boldsymbol{\chimon}\big(\mathcal{P}_J(\widetilde{X}_n)\big)
\le \boldsymbol{\chimon}\big(\mathcal{P}_J(\widetilde{Y}_n)\big)
\le M^{(r)}(Y)^{m}\,\boldsymbol{\chimon}\big (\mathcal{P}_J(Y_n)\big)
\end{align*}
and this completes the proof of the  claim.
\end{proof}

We are now prepared to prove the result concerning the correct asymptotic estimates stated in Corollary
\ref{K-S for poly}.

\begin{theorem}\label{lower bound 2-convex for poly}
Let  $X$ be a \,$2$-convex Banach sequence lattice, and let $\big(J_m\big)$ be a sequence of index sets,
each with degree at most~$m$. Then, for all $m$ such that $\Lambda_T(m) \subset J_m$, we have
\begin{equation} \label{mmm 2-conv}
\boldsymbol{\lambda}\big(\mathcal{P}_{J_m}(X_n)\big)  \sim_{C^m}\Big( 1+\frac{n}{m}\Big)^{\frac{m}{2}}
\end{equation}
and
\begin{equation} \label{MMM 2-conv}
\boldsymbol{\chimon}\big(\mathcal{P}_{J_m}(X_n)\big)  \sim_{C^m}\Big(1+ \frac{n}{m}\Big)^{\frac{m-1}{2}}\,,
\end{equation}
where $C > 0$ only depends on the $2$-convexity constant of $X$.
\end{theorem}

\begin{proof}
The upper bounds are a special case of Corollary \ref{K-S for poly}. To prove the lower bounds, we start with
that of \eqref{MMM 2-conv}.  From Corollary \ref{appl} with $r=2$, it follows that for all $m$
with $\Lambda_T(m) \subset J_m$,
\[
\boldsymbol{\chimon}\big(\mathcal{P}_{\Lambda_{T}(m)} (\ell_2^n)\big) \leq M^{(2)}(X)^m
\boldsymbol{\chimon}\big(\mathcal{P}_{\Lambda_{T}(m)}(X_n) \big) \leq M^{(2)}(X)^m \boldsymbol{\chimon}
\big(\mathcal{P}_{J_m}(X_n)\big)\,,
\]
Consequently by Proposition~\ref{toblach} (in the Hilbert space case) for $m\le n$, we obtain
\[
\Big( 1+ \frac{n}{m}\Big)^{\frac{m-1}{2}} \leq 2^{\frac{m-1}{2}} \Big(\frac{n}{m}\Big)^{\frac{m-1}{2}}
\prec_{C^m} \boldsymbol{\chimon} \big(\mathcal{P}_{\Lambda_T(m)}(\ell_2^n)\big) \prec_{C^m} \boldsymbol{\chimon}
\big(\mathcal{P}_{J_m}(X_n)\big)\,.
\]
Clearly, for all $m \ge n$, we have
\begin{align} \label{obvious}
\Big( 1+ \frac{n}{m}\Big)^{\frac{m-1}{2}} \leq 2^{\frac{m-1}{2}}  \prec_{C^m} \boldsymbol{\chimon}\big(\mathcal{P}_{J_m}(X_n)\big)\,,
\end{align}
so the lower bound in \eqref{MMM 2-conv} follows.

Finally, it remains to check the  lower bound in \eqref{mmm 2-conv}. We apply  Theorem~\ref{thm: uncond cte vs proj cte},
Theorem~\ref{OrOuSe} (twice), and  Proposition~\ref{prop: homogeneous part proj constant}, which all together prove that
for $m $ with $\Lambda_T(m) \subset J_m$, we get
\begin{align}\label{verflucht}
\begin{split}
\boldsymbol{\chimon}\big( \mathcal{P}_{\Lambda_T(m+1)}(X_n)\big) & \, \leq \,e 2^{m+1}
\|\mathbf{Q}_{\Lambda(m+1,n),\Lambda_T(m+1,n)}\|\,\, \boldsymbol{\lambda}\big( \mathcal{P}_{\Lambda_T(m)}(X_n)\big) \\
& \, \leq \,e 2^{m+1}  \kappa^{m+1} \boldsymbol{\lambda}\big( \mathcal{P}_{\Lambda_T(m)}(X_n)\big) \\
& \, \leq \,e 2^{m+1}  \kappa^{2m+1} \boldsymbol{\lambda}\big(\mathcal{P}_{J_m(m)}(X_n)\big)
\, \leq \,e 2^{m+1} \kappa^{2m+1} \boldsymbol{\lambda}\big(\mathcal{P}_{J_m}(X_n)\big)\,;
 \end{split}
\end{align}
for the third estimate note that that $\Lambda_T(m) = J_m(m)\cap  \Lambda_T(\leq m) $.
 Then for $m $ with $\Lambda_T(m) \subset J_m$, we obtain the lower bound in \eqref{mmm 2-conv} from
\eqref{MMM 2-conv}. 
\end{proof}

\subsection{Estimates by norms of coefficient functionals} \label{Part: Polynomial projection constants}

In this subsection we introduce another useful technique to estimate the projection constant of spaces $\mathcal{P}_J(X_n)$
of multivariate polynomials on finite dimensional Banach lattices $X_n$ supported on a finite index set $J$.

Given an index set $J \subset \mathbb{N}_0^n$ and Banach space $X_n = (\mathbb{C}^n,\|\cdot\|)$, we  define the isometrical
embedding $I \colon \mathcal{P}_{J}(X_n) \to \ell_\infty(B_{X_n})$ by
\[
I(P) := P|_{B_{X_n}}, \quad\, P\in \mathcal{P}_{J}(X_n)\,.
\]
For each $\alpha \in J$ \, let $c_{\alpha}^{*}  \colon \mathcal{P}_J(X_n) \to \mathbb{C}$ be the linear functional, given, for
any polynomial $P = \sum_{\alpha \in J} c_\alpha(P) e_\alpha$ by $c_{\alpha}^{*}(P) := c_{\alpha}(P)$. From the Hahn-Banach
theorem, it follows that for any $c_{\alpha}^{*}$, we can find a~norm-preserving extension
\[
\widetilde{c}_{\alpha}^{*} \in \ell_\infty(B_{X_n})^{*}\,,
\]
that is, for all $P\in \mathcal{P}_{J}(X_n)$, we have $\|\widetilde{c}_{\alpha}^{*}\| = \|c_{\alpha}^{*}\|$
and $\widetilde{c}_{\alpha}^{*}(P)= c_\alpha(P)$.

Consider the following obvious factorization of the identity map $\id \colon \mathcal{P}_J(X_n) \to
\mathcal{P}_J(X_n)$,
\[
\id \colon \mathcal{P}_J(X_n) \stackrel{I} \longrightarrow  \ell_\infty(B_{X_n}) \stackrel{Q}
\longrightarrow \mathcal{P}_J(X_n)\,,
\]
where the finite dimensional operator $Q$ is given by the formula
\begin{equation}\label{QQQ}
Qf := \sum_{\alpha \in J} \widetilde{c}_{\alpha}^{*}(f) \,e_\alpha, \quad\, f \in \ell_\infty(B_{X_n})\,.
\end{equation}
Clearly, this implies that
\[
\boldsymbol{\lambda}\big(\mathcal{P}_J(X_n)\big) \leq \big\|Q\colon \ell_\infty(B_{X_n}) \to\ell_\infty(B_{X_n})\big\|\,.
\]
Although determining the norm of this projection $Q$ is generally a challenging task, we find upper estimates which, under certain
restrictive assumptions on the index set $J$ and the Banach space $X_n$, lead to asymptotically optimal bounds for $\boldsymbol{\lambda}\big(\mathcal{P}_J(X_n)\big)$. More precisely, we introduce what we call the polynomial projection constant:
\[
\widehat{\boldsymbol{\lambda}}\big(\mathcal P_J(X_n)\big):= \sup_{z \in B_{X_n}} \sum_{\alpha \in J} c_{X_n}(\alpha) \vert z^{\alpha}\vert\,,
\]
where $c_{X_n}(\alpha)$, the  characteristic of the multi-index $\alpha$, is the reciprocal of the norm of the monomial $z^{\alpha}$ in $\mathcal{P}_J(X_n)$, that is,
\[
c_{X_n}(\alpha):\, = \,\frac{1}{\sup_{z \in B_{X_n}} |z^\alpha|}\,.
\]
An interesting fact is that we can actually bound the projection constant of the space $\mathcal{P}_J(X_n)$ by this quantity, as shown
in Theorem~\ref{lambda-dash}.  Therefore, having accurate upper bounds for the  characteristic $c_{X_n}(\alpha)$ is crucial for our
purposes. In fact,  it turns out that the polynomial projection $\widehat{\boldsymbol{\lambda}}\big(\mathcal{P}_J(X_n)\big)$, facilitates
the systematic and  comfortable extension of estimates such as \eqref{eq:proj const pols on lp} to wider ranges of index sets $J$ and
spaces $X_n$. At the same time, it adheres to a~reasonable abstract theory. In many concrete situations, this estimate is sufficient;
the advantage here is that this new parameter is more manageable compared to $\boldsymbol{\lambda}\big(\mathcal{P}_J(X_n)\big)$.

The norm of the functional $c_{\alpha}^*$ can be bounded by the characteristic of the index $\alpha$.
Indeed, the following lemma is just a rephrasing of the Cauchy inequalities.
\begin{lemma} \label{sup}
Let $X_n = (\mathbb{C}^n, \|\cdot\|)$ be a Banach lattice and $J \subset \mathbb{N}_0^n$. Then, for every $P \in \mathcal{P}_J(X_n)$ and
$\alpha \in J $, we have
\[
|c_\alpha(P)| \leq c_{X_n}(\alpha) \Vert P \Vert_{\mathcal{P}_J(X_n)}\,;
\]
in other terms, $\,\big\|c_\alpha^{*}\big\|_{\mathcal{P}_J(X_n)^\ast} \leq  c_{X_n}(\alpha)$\,.
\end{lemma}


Coming back to the discussion from the introduction of  this chapter, we get an estimate of the projection constant by the
polynomial projection constant.

\begin{theorem} \label{lambda-dash}
Let $X_n = (\mathbb{C}^n,\|\cdot\|)$ be a~Banach lattice and  $J \subset \mathbb{N}_0^n$. Then
\[
\boldsymbol{\lambda}\big(\mathcal{P}_{J}
(X_n)\big) \leq \widehat{\boldsymbol{\lambda}}\big(\mathcal{P}_{J}(X_n)\big)\,.
\]
\end{theorem}

\begin{proof}
As noted above, the mapping $Q \colon \ell_\infty(B_{X_n}) \to \ell_\infty(B_{X_n})$ defined by the formula \eqref{QQQ} is
the projection of $\ell_\infty(B_{X_n})$ onto $I(\mathcal{P}_{J}(X_n))$. Then by Lemma \ref{sup}, it follows that
\begin{align*}
\|Q\| &=  \sup_{\|f\|_{\ell_{\infty}(B_{X_n})} \leq 1} \Big\|
\sum_{\alpha  \in J} \widetilde{c}^{\ast}_{\alpha} (f) z^\alpha\Big\|_{\mathcal{P}_J(X_n)}
\leq \sup_{z \in B_{X_n}} \sum_{\alpha\in J} \big\|\widetilde{c}^{\ast}_{\alpha} \big\|_{\ell_{\infty}(B_{X_n})^{*}}\,|z^\alpha|
\\
& = \sup_{z \in B_{X_n}} \sum_{\alpha\in J}  \big\|c^{\ast}_{\alpha}\big\|_{\mathcal P_J(X_n)^\ast}|z^\alpha|
\leq \sup_{z \in B_{X_n}} \sum_{\alpha\in J}  c_{X_n}(\alpha)|z^\alpha|
\\
& = \widehat{\boldsymbol{\lambda}}\big(\mathcal{P}_J(X_n)\big)\,,
\end{align*}
and since $\boldsymbol{\lambda}\big(\mathcal{P}_J (X_n)\big) \leq \|Q\|$, the desired estimate follows.
\end{proof}

It is well-known that for every $\alpha \in \Lambda(m,n)$ and $1 \leq r \leq \infty$, we have
\begin{equation} \label{dineen}
c_{\ell_{r}^n}(\alpha) = \Big(\frac{m^m}{\alpha^\alpha}\Big)^{1/r}\,;
\end{equation}
for an  elementary proof see \cite[Lemma 1.38]{dineen1999complex}. Observe as an immediate consequence that for  tetrahedral indices
$\alpha \in \Lambda_T(m,n)$ with $m \leq n$, we get
\begin{equation} \label{dineen2}
c_{\ell_{r}^n}(\alpha) = m^{\frac{m}{r}}\,.
\end{equation}
By Theorem~\ref{lambda-dash} accurate upper bounds for $c_{X_n}(\alpha)$  provide good upper bounds for the
projection constant of $\mathcal{P}_J(X_n)$.

We will explore how to derive general variants of the estimates mentioned above. As in the previous subsection, we will apply
Lozanovski\v{\i}'s theorem discussed earlier.

\begin{theorem} \label{thm: bound for coef functional}
The following statements are true for any Banach lattice $X_n=(\mathbb{C}^n, \|\cdot\|_{X_n})$, any $J\subset \mathbb{N}_0^n$,
and any  $\alpha \in \Lambda(m,n)${\rm:}
\begin{itemize}
\item[(i)] $c_{X_n}(\alpha)\,c_{X_{n}'}(\alpha)  \leq \frac{m^m}{\alpha^{\alpha}}$\,.
\item[(ii)] $c_{X_n}(\alpha) \le \frac{\|\alpha\|^m_{X_n}}{\alpha^\alpha}$\,.
\end{itemize}
\end{theorem}

\begin{proof}
(i)
Clearly, \(X_n\) has the Fatou property because it is finite-dimensional. Lozanovski\v{\i}'s factorization theorem (summarized
in Equation \eqref{Lozanovskii}) implies that for every \(z \in S_{\ell_1^n}\), we can select \(u \in S_{X_n}\) and \(v \in S_{X_{n}'}\)
such that \(|u||v| = |z|\). Thus,  we have:
\[
c_{X_n}(\alpha)\,c_{X_{n}'}(\alpha) \le \frac{1}{\dis\sup_{u \in S_{X_n}}|u^\alpha|} \frac{1}{\dis\sup_{v \in S_{X_n'}}|v^\alpha|}
\le \frac{1}{\dis\sup_{z\in S_{\ell_1^n}}|z^\alpha|} \le \frac{m^m}{\alpha^\alpha}\,,
\]
where the last inequality holds taking $z= \frac{\alpha}{m}$.

(ii) For  $z \in S_{\ell_1^n}$ let  $u\in S_{X_n}$, $v\in S_{X_n'}$ be positive and such that $|z|=uv$. Since
\[
(v^{\alpha})^{1/m}= \big(v_1^{\alpha_1} \cdots v_n^{\alpha_n}\big)^{1/m}
\leq \frac{1}{m}(\alpha_1 v_1 + \ldots + \alpha_n v_n)\,,
\]
it follows that
\[
(v^{\alpha})^{1/m} \leq \frac{1}{m} \|(\alpha_1, \ldots, \alpha_n)\|_{X_n}\,\|(v_1, \ldots, v_n)\|_{X_{n}'}
= \frac{1}{m}\,\|(\alpha_1, \ldots, \alpha_n)\|_{X_n}\,,
\]
and hence
\[
c_{X_{n}'}(\alpha) \le \frac{1}{|u^\alpha|} =  \frac{|v^\alpha|}{|z^\alpha|}
\le \frac{1}{m^m}\,\frac{\|(\alpha_1, \ldots, \alpha_n)\|^m_{X_n}}{|z|^{\alpha}}\,.
\]
Then taking the supremum in $z \in S_{\ell_1^n}$, we obtain
\[
c_{X_{n}'}(\alpha) \le \frac{1}{m^m}\,\|(\alpha_1, \ldots, \alpha_n)\|^m_{X_n}\frac{m^m}{\alpha^\alpha}
= \frac{\|(\alpha_1, \ldots, \alpha_n)\|^m_{X_n}}{\alpha^\alpha}\,. \qedhere
\]
\end{proof}

\smallskip

\begin{corollary}\label{coro: bound proj constant tetra}
\label{lambda}
Let $X_n= (\mathbb{C}^n, \|\cdot\|)$~be a symmetric Banach lattice and  $J~\subset~\Lambda_T(m,n)$ some index set, where
$m \leq n$. Then
\[
\widehat{\boldsymbol{\lambda}}\big(\mathcal{P}_J(X_n)\big)
\le e^{m} \bigg(\frac{\varphi_{X_{n}'}(n)}{\varphi_{X_{n}'}(m)}\bigg)^{m}\,.
\]
\end{corollary}

\begin{proof}
In case $\alpha$ is tetrahedral, $\|(\alpha_1, \ldots, \alpha_n)\|_{X_n}=\varphi_X(m) = m/\varphi_{X_{n}'}(m)$ and
thus by Theorem \ref{thm: bound for coef functional},
\[
c_{X_n}(\alpha) \le \frac{\|\alpha\|^m_{X_n}}{\alpha^\alpha} \le
\frac{m^m}{\alpha^\alpha}\frac1{\varphi_{X_{n}'}(m)^m}\le e^m\frac{m!}{\alpha!}\frac1{\varphi_{X_{n}'}(m)^m}\,.
\]
Thus, we obtain that
\begin{align*}
\widehat{\boldsymbol{\lambda}}\big(\mathcal{P}_{J}(X_n)\big)
& \leq e^m\frac1{\varphi_{X_{n}'}(m)^m}\sup_{z\in B_{X_n}}\,\,\sum_{\alpha \in J}  \frac{m!}{\alpha!} \,|z^\alpha | \\
& \le e^m\frac1{\varphi_{X_{n}'}(m)^m}\sup_{z\in B_{X_n}}(|z_1|+\dots+|z_n|)^m
= e^m \bigg(\frac{\varphi_{X_{n}'}(n)}{\varphi_{X_{n}'}(m)}\bigg)^{m}\,. \qedhere
\end{align*}
\end{proof}

We conclude this subsection with a result that is interesting in its own right. It is motivated by the surprising fact
noted in \cite{ryll1983homogeneous} (as a simple consequence of \eqref{RWRW}) that

\begin{align*} 
\boldsymbol{\lambda}\big(\mathcal{P}_m(\ell_2^n)\big) \leq 2^{n-1}, \quad\, m,n\in \mathbb{N}\,.
\end{align*}
In fact, the sequence $\big(X_m\big)_{m \geq 1}$ with $X_m := \mathcal{P}_m(\ell_2^2)$ was the first known nontrivial example of
a~sequence of finite-dimensional Banach spaces for which $\lim_{m \to \infty} \dim X_m = \infty$, even though
\[
\sup_m \boldsymbol{\lambda}(X_m) < \infty\,.
\]
It is worth mentioning that Bourgain \cite{bourgain1989} gave an affirmative solution to a~problem considered in \cite{ryll1983homogeneous},
showing that the sequence $\big(d\big(X_m, \ell_\infty^{\dim X_m}(\mathbb{C})\big)\big)_m$ of Banach-Mazur distances is bounded.

\begin{theorem} \label{Xl1}
Let $X$ be a Banach sequence lattice. Then the following statements are equivalent{\rm:}
\begin{itemize}
\item[{\rm(i)}] $\sup_{m,n} \boldsymbol{\lambda}\big(\mathcal{P}_{ \leq m}(X_n)\big)^\frac{1}{m} < \infty$\,.
\item[{\rm(ii)}] $\sup_{m,n} \boldsymbol{\lambda}\big(\mathcal{P}_{ m}(X_n)\big)^\frac{1}{m} < \infty$\,.
\item[{\rm(iii)}] $X_n = \ell_1^n$ uniformly\,.
\item[{\rm(iv)}] $X=\ell_1$ whenever $X$ is separable\,.
\end{itemize}
\end{theorem}

\begin{proof}
By Theorem~\ref{degree-homo} the first two statements are equivalent. Again Theorem~\ref{degree-homo} combined with
the well known estimate $\boldsymbol{\lambda}\big(\mathcal{P}_{ m}(\ell_1^n)\big) \leq  e^m$ yields
$\boldsymbol{\lambda}\big(\mathcal{P}_{ \leq m}(\ell_1^n)\big) \leq  (m+1) e^{m}$. Thus we need only show that (ii)
implies (iii). Under the assumption (ii) we have that
\[
\sup_{n} \gamma_1(X_n)  = \sup_{n} \gamma_\infty(X_n^\ast) =\sup_{n} \boldsymbol{\lambda}(X_n^\ast) \leq C:= \sup_{m,n} \boldsymbol{\lambda}\big(\mathcal{P}_{  m}(X_n)\big)^\frac{1}{m} < \infty\,.
\]
We prove that there is $K >0$ such that for any diagonal operator $D_\lambda: X_n  \to \ell_n^2, \,(e_i) \mapsto (\lambda_i e_i)$
we have
\[
\sup_n \pi_1(D_\lambda\colon X_n\to \ell^n_2) \leq  K \|D_\lambda\|\,.
\]
Indeed, take a factorization $\id = uv$, where $u\colon X_n \to L_1(\mu)$ and $v\colon  L_1(\mu) \to X_n $ for some measure $\mu$
and  $\|u\| \|v\| \leq  2  \gamma_1(X_n)$. From the classical Grothendieck's theorem, it follows that
\[
\pi_1 (D_\lambda \circ v) \leq K_G  \|D_\lambda \circ v\|\,,
\]
and hence
\begin{align*}
\pi_1(D_\lambda\colon X_n\to \ell^n_2) & =  \pi_1(D_\lambda\circ v \circ u)
\leq \pi_1(D_\lambda\circ v) \|u\| \\
& \leq K_G  \|D_\lambda\| \|v\| \|u\|\leq  K_G  \|D_\lambda\| 2  \gamma_1(X_n) \leq  2 K_G C  \|D_\lambda\|\,.
\end{align*}
Then we deduce from \cite[Proposition 33.6]{tomczak1989banach} (a result which essentially goes  back to Lindenstrauss and
Pe{\l}czy\'nski) that
\[
\|\id\colon  X_n \to \ell_1^n \| \leq K=(2 K_G C)^2\,,
\]
and this completes the proof.
\end{proof}

Determining whether, in Theorem \ref{Xl1}, the projection constant of $\mathcal{P}_{m}(X_n)$ can be replaced by the unconditional
basis constant of the monomial basis remains an open problem, known as 'Lempert's problem' in \cite{defant2019libro}. Specifically,
it is unknown if $X \equiv \ell_1$ if and only if $\sup_{m,n} \boldsymbol{\chimon}\left(\mathcal{P}_{  m}(X_n)\right)^{\frac{1}{m}}< \infty$.

We recall a significant partial solution provided in \cite{bayart2012maximum} (see also \cite[Theorem 20.26]{defant2019libro}),
which essentially shows that in this case $X$ is very similar to $\ell_1$: there is $D=D(\varepsilon) >0$ such that
$\|\id\colon X_n \to \ell_1^n \| \,\,\leq D \,\,(\log \log n)^\varepsilon\,$ for all $n$.


As we have observed, the projection constant and the unconditional basis constant are closely related. This indicates that, in
a~certain sense, the results of Theorem \ref{Xl1} can be viewed as a solution to a weaker version of Lempert's problem.

\subsection{Probabilistic estimates}

\label{Probabilistic estimates}

For future reference, we establish several probabilistic lower bounds for the unconditional basis constant of the monomial basis
$(z^\alpha)_{\alpha\in J}$ in the space $\mathcal{P}_{J}(X_n)$ of multivariate polynomials supported on $J$. These bounds are
derived under specific conditions on the underlying Banach sequence lattice $X_n$ and the index set $J \subset \mathbb{N}_0^{n}$.

We closely follow the methods from \cite{bayart2012maximum}, \cite{boas2000majorant}, \cite{defant2004maximum}, \cite{defant2020subgaussian},
and \cite{mastylo2017kahane}, which primarily address the homogeneous case $J = \Lambda(m,n)$. In particular, we will rely on two key
theorems from Bayart’s work in \cite{bayart2012maximum}, which we enclose below for completeness. It is important to note  that the subtle
proofs of these results utilize different techniques: the "coverage method" and the "entropy method." For all the necessary preliminary tools
and proofs of these results, we refer to \cite{defant2019libro}, specifically Corollary~17.5 and Corollary~17.22.

\begin{theorem} \label{Anquetil alfa}
Given  $1 \leq r \leq 2$, there is a constant $C >0$ such that for  each $m \geq 2$, for every Banach space $X_n = (\mathbb{C}^n,\|\cdot\|)$,
and for every choice of scalars $(c_{\alpha})_{\alpha \in \Lambda(m,n)}$ there exists a choice of signs $\varepsilon_{\alpha} = \pm1, \, \alpha \in \Lambda(m,n)$ such that
\[
\sup_{z \in B_{X_n}} \Big\vert \sum_{\alpha \in \Lambda(m,n)} \varepsilon_{\alpha} c_{\alpha} z^{\alpha} \Big\vert
\leq C (n \log m)^{\frac{1}{r'}}  \sup_{\alpha}  \Big(\vert c_{\alpha} \vert \Big( \frac{\alpha !}{m!}\Big)^{\frac{1}{r}} \Big)
\sup_{z \in B_{X_n}} \Big( \sum_{k=1}^{n} \vert z_{k} \vert^{r} \Big)^{\frac{m}{r}}\,.
\]
\end{theorem}

\medskip

\begin{theorem} \label{Hinault alfa}
Given  $1 \leq r \leq 2$, there is a constant $C >0$ such that for each  $m \geq 2$, for every Banach space $X_n = (\mathbb{C}^n,\|\cdot\|)$,
and for every choice of scalars $(c_{\alpha})_{\alpha \in \Lambda(m,n)}$ there exists a choice of signs $\varepsilon_{\alpha} = \pm1, \, \alpha \in \Lambda(m,n)$ such that
\begin{align*}
\sup_{z \in B_{X_n}} \Big\vert \sum_{\alpha \in \Lambda(m,n)} \varepsilon_{\alpha} c_{\alpha} z^{\alpha} \Big\vert
\leq C m (\log n)^{1+\frac{1}{r'}}   \sup_{\alpha}  \bigg( \vert c_{\alpha} \vert \Big( \frac{\alpha !}{m!} \Big)^{\frac{1}{r}} \bigg)
\sup_{z \in B_{X_n}} \Big( \sum_{k=1}^{n} \vert z_{k} \vert^{r} \Big)^{\frac{m-1}{r}}  \sup_{z \in B_{X_n}} \sum_{k=1}^{n} \vert z_{k} \vert\,.
\end{align*}
\end{theorem}

\smallskip

We present some lower bounds for $\boldsymbol{\chimon}\big(\mathcal{P}_J(X_n)\big)$, where $X_n$ is the $n$th section of a Banach sequence
lattice $X$ and the index set $J \subset \mathbb{N}_0^{(\mathbb{N})}$ contains all $m$-homogeneous tetrahedral multi-indices  of length $m\leq n$. The
proofs of these estimates are based on Theorem~\ref{Anquetil alfa}, Theorem~\ref{Hinault alfa} and  the following lemma.

\begin{lemma} \label{innichenA}
For each  $1 \leq r \leq 2$, there is a constant $C >0$ such that for every Banach  lattice $X_n = (\mathbb{C}^n, \|\cdot \|)$
and for each $m \le n$ one has
\[
\frac{1}{\|\id\colon X_n\to \ell_r^n\|^m}\frac{|\Lambda_T(m,n)|}{\varphi_{X_n}(n)^m n^{\frac{1}{r'}} m ^{-\frac{m}{r}}}
\,\,\leq\,\, C e^{\frac{m}{r}} (\log m)^{1/r'} \boldsymbol{\chimon}\big(\mathcal{P}_{\Lambda_T(m)}(X_n)\big)\,.
\]
\end{lemma}

\begin{proof}
Clearly, taking $z = (\varphi_{X_n}(n)^{-1}, \ldots,\varphi_{X_n}(n)^{-1}) \in B_{X_n}$, we get
\[
\frac{|\Lambda_T(m,n)|}{\varphi_{X_n}(n)^m} \leq \sup_{z \in B_{X_n}} \Big| \sum_{\alpha \in \Lambda_T(m,n)} z^\alpha\Big|\,.
\]
Then, it follows from Theorem~\ref{Anquetil alfa} that we  find signs $\varepsilon_\alpha = \pm 1, \alpha \in \Lambda_T(m,n)$
for which for all  $m \leq n$
\begin{align*}
\sup_{z \in B_{X_n}} \Big| \sum_{\alpha \in \Lambda_T(m,n)} \varepsilon_\alpha  z^\alpha\Big|
& \leq C \,\,  (n \log m)^{1/r'} \sup_{\alpha \in \Lambda_T(m,n)} \Big(\frac{\alpha!}{m!}\Big)^{1/r}
\sup_{z \in B_X} \Big(\sum_{k=1}^n |z_k|^r\Big)^{m/r} \\
& \leq C \,\, (\log m)^{1/r'} n ^{1/r'}m!^{-1/r} \sup_{z \in B_X} \Big(\sum_{k=1}^n |z_k|^r\Big)^{m/r}   \nonumber\\
& =  C \,\, (\log m)^{1/r'} n ^{1/r'}m!^{-1/r} \|\id \colon X_n\to \ell_r^n\|^m\,,  \nonumber
\end{align*}
where $C >0$ is a constant only depending on $r$\,. Using  that $m^m \leq e^m m!$, we finally arrive at
\begin{align*}
\frac{|\Lambda_T(m,n)|}{\varphi_{X_n}(n)^m} & \leq
\sup_{z \in B_{X_n}} \Big| \sum_{\alpha \in  \Lambda_T(m,n)} \varepsilon_\alpha \varepsilon_\alpha z^\alpha \Big| \\
& \leq  \boldsymbol{\chimon} \big(\mathcal{P}_{ \Lambda_T(m,n)}\big)
\sup_{z \in B_{X_n}} \Big| \sum_{\alpha \in \Lambda_T(m,n)} \varepsilon_\alpha z^\alpha\Big| \\
& \leq \boldsymbol{\chimon} \big(\mathcal{P}_{\Lambda_T(m,n)}\big)
\,C\,(\log m)^{1/r'}n^{1/r'} (m!)^{-1/r}\|\id \colon X_n\to \ell_r^n\|^m \\
& \leq \boldsymbol{\chimon} \big(\mathcal{P}_{ \Lambda_T(m,n)}\big)
\,C\, e^{\frac{m}{r}}(\log m)^{1/r'} n^{1/r'} (m!)^{-m/r} \|\id \colon X_n \to \ell_r^n\|^{m}\,. \qedhere
\end{align*}
\end{proof}

We are now ready to prove the aforementioned lower bounds.

\begin{proposition} \label{innichen1}
Let $1 \leq r \leq 2$, and let $X$ be a Banach sequence lattice such that $\varphi_X(n)\prec n^{1/r}$. Then, for any
sequence $\big(J_m\big)$ of index sets, each with degree at most $m$ and $\Lambda_T(m) \subset J_m$, the following
estimate holds{\rm:}
\[
\frac{1}{\|\mathrm{id}\colon X_n\to \ell_r^n\|^m} \Big( \frac{n}{m} \Big)^{\frac{m-1}{r'}}
\prec C^{m} \boldsymbol{\chimon}\big(\mathcal{P}_{J_m}(X_n) \big)\,,
\]
where $C \ge 1$ is a constant depending only on $X$.
\end{proposition}

\begin{proof}
We only have to consider the case $J = \Lambda_T(m,n)$. Note first that
\begin{equation} \label{binom}
\Big(\frac{n}{m}\Big)^m \leq \binom{n}{m} = |\Lambda_T(m,n)|\,.
\end{equation}
Combining this  with the assumption  $\varphi_X(n) \leq C n^{1/r}$, we get
\[
\frac{|\Lambda_T(m,n)|}{\varphi_{X_n}(n)^m n^{\frac{1}{r'}} m ^{-\frac{m}{r}}} \ge
\frac{n^m}{C^m n^{\frac{m}{r}} n^{\frac{1}{r'}}  m^{\frac{m}{r'}}}
=  \Big( \frac{n}{m} \Big)^{\frac{m-1}{r'}}  \frac{1}{C^m m^\frac{1}{r'}},
\]
and hence the claim is an immediate consequence of Lemma~\ref{innichenA}.
\end{proof}

\begin{proposition} \label{toblach}
Let $X$ be a Banach sequence lattice  such that
\begin{equation} \label{M}
\text{
$\varphi_{X_n}(n)\,\varphi_{X_{n}'}(n) \prec n$
\quad and \quad
$\|\id \colon X_n \to \ell_2^{n}\| \prec \frac{1}{\sqrt{n}}\,\|\id\colon X_n \to \ell_1^{n}\|\,.$}
\end{equation}
Then, for any sequence $\big(J_m\big)$ of index sets, each with degree at most $m$
and $\Lambda_T(m)\subset J_m$, we have{\rm:}
\[
\Big( \frac{n}{m} \Big)^{\frac{m-1}{2}} \prec_{C^m}
\boldsymbol{\chimon}\big(\mathcal{P}_{J_m}(X_n)\big)\,,
\]
where $C \ge 1$ is a constant depending only on $X$.
\end{proposition}

\begin{proof}
Since for $x =\frac{e_1}{\|e_1\|_X}$ one has $\|x\|_X=1$, it follows that
\[
\gamma:=\frac{1}{{\|e_1\|_X}} \leq \|\id\colon X_n \to \ell_2^{n}\| \prec \frac{1}{\sqrt{n}}\,\|\id\colon X_n \to \ell_1^{n}\|\,.
\]
Using the estimate $\varphi_{X_n}(n)\,\varphi_{X_{n}'}(n) \prec n$, we get
\[
\gamma \prec \frac{1}{\sqrt{n}}\,\|\id\colon X_n \to \ell_1^{n}\| = \frac{1}{\sqrt{n}}\,\varphi_{X_{n}'}(n)
\prec \frac{\sqrt{n}}{\varphi_{X_n}(n)}\,,
\]
and hence $\varphi_{X_n}(n)\prec \sqrt{n}$. This combined with the above estimate \eqref{binom} yields
\begin{align*}
\frac{1}{\|\id \colon X_n\to \ell_2^n\|^m}\frac{|\Lambda_T(m,n)|}{\varphi_{X_n}(n)^m n^{\frac{1}{2}} m ^{-\frac{m}{2}}}
& \succ \frac{n^\frac{m}{2}}{\varphi_{X_{n}'}(n)^m} \Big(\frac{n}{m}\Big)^m \frac{1}{n^\frac{1}{2}
\varphi_{X_n}(n)^m  m^{-\frac{m}{2}}}
\sim \frac{n^{\frac{m-1}{2}}}{m^{\frac{m}{2}}} = \frac{1}{m^{\frac{1}{2}}}\Big( \frac{n}{m} \Big)^{\frac{m-1}{2}}\,,
\end{align*}
and hence  the conclusion follows (as in the preceding proof) from Lemma~\ref{innichenA}.
\end{proof}

Applying Theorem~\ref{Hinault alfa}, we obtain an additional lower bound that will be crucial later on.

\begin{proposition} \label{innichen}
Let $1 \leq r \leq 2$, and let $X$ be a Banach sequence lattice $X$  such that $\varphi_{X_n}(n) \varphi_{X_{n}'}(n)~\prec~n$.
Then, there exists a~constant $C = C(r,X)$ such that for each $m\le n$, the following estimate holds{\rm:}
\[
\frac{\Big( \frac{n}{n-m} \Big)^{n-m} \sqrt{\frac{n}{m(n-m)}}}{C m\, (\log n)^{1+1/r'} e^{\frac{m}{r}} m^{m/r'}}
\,\,\,\,\left(  \frac{\|\id \colon X_n\to \ell_1^n\|}{\|\id \colon X_n\to \ell_r^n\|}\right)^{m-1}
\,\leq\,
\boldsymbol{\chimon}\big(\mathcal{P}_{\Lambda_T(m)}(X_n)\big)\,.
\]
\end{proposition}

\begin{proof}
Note first that for each $m,n \in \mathbb{N}$
\begin{equation} \label{puchner}
\frac{m!^{1/r}}{C_1m \,  (\log n)^{1+1/r'}}\frac{|\Lambda_T(m,n)|}{ \|\id \colon X_n\to \ell_r^n\|^{m-1}\varphi_{X_n}(n)^m \varphi_{X_{n}'}(n)}
\,\leq\,
\boldsymbol{\chimon}\big(\mathcal{P}_{\Lambda_T(m)}(X_n)\big)\,,
\end{equation}
where $C_1 = C_1(r)$ only depends on $r$. Indeed, choose signs $\varepsilon_\alpha = \pm 1, \alpha \in \Lambda_T(m,n)$ such that
\begin{align*} \label{poly2}
\begin{split}
\sup_{z \in B_{X_n}} \Big| \sum_{\alpha \in \Lambda_T(m,n)} \varepsilon_\alpha  z^\alpha\Big|
\leq \,C_1 m \,  (\log n)^{1+1/r'} m!^{-1/r}  \|\id\colon X_n\to \ell_r^n\|^{m-1}  \|\id \colon X_n\to \ell_1^n\|\,,
\end{split}
\end{align*}
where $C_1 = C_1(r)$ is the constant from Theorem~\ref{Hinault alfa}. Since $\|\id\colon X_n\to \ell_1^n\| = \varphi_{X_{n}'}(n)$,
we may now proceed exactly as in the proof of Lemma~\ref{innichenA} to get \eqref{puchner}. We use now
Stirling's formula,
\begin{equation*}
\sqrt{2 \pi}  n^{n + \frac{1}{2}} e^{-n} \leq n! \leq   \frac{12}{11}\sqrt{2 \pi} n^{n + \frac{1}{2}} e^{-n}, \quad\, n\in \mathbb{N}\,,
\end{equation*}
in  \eqref{puchner} to conclude that there is a uniform constant $\gamma \ge 1$ such that for each $m \leq n$
\begin{equation*}
|\Lambda_T(m,n)| = \binom{n}{m}  \ge \gamma  \Big( \frac{n}{m} \Big)^{m}  \Big( \frac{n}{n-m} \Big)^{n-m} \sqrt{\frac{n}{m(n-m)}}\,.
\end{equation*}
Estimating now \eqref{puchner}, we get with $C_2 = C_1/\gamma$ that for all $m \leq n$
\begin{align*}
\boldsymbol{\chimon}\big(\mathcal{P}_{\Lambda_T(m)}(X_n)\big) \geq
\frac{  m!^{1/r} }{C_2m   (\log n)^{1+1/r'}  }
\, \, \,  \frac{n^m\Big( \frac{n}{n-m} \Big)^{n-m} \sqrt{\frac{n}{m(n-m)}}}{m^m\|\id
\colon X_n\to \ell_r^n\|^{m-1}\varphi_{X_n}(n)^m \varphi_{X_{n}'}(n)}\,.
\end{align*}
Finally, since
$$\displaystyle \frac{n^{m-1}}{\|\id \colon X_n\to \ell_r^n\|^{m-1}\varphi_{X_n}(n)^{m-1}} =
\left(\frac{\|\id\colon X_n\to \ell_1^n\|}{\|\id \colon X_n\to \ell_r^n\|}\right)^{m-1}$$
and
$1 \prec \frac{n}{{\varphi_{X_n}(n)\varphi_{X_n'}(n)}}$, we have
\begin{align*} \boldsymbol{\chimon}\big(\mathcal{P}_{\Lambda_T(m)}(X_n)\big) \geq
\frac{\Big( \frac{n}{n-m} \Big)^{n-m} \sqrt{\frac{n}{m(n-m)}}}{Cm   (\log n)^{1+1/r'} m^m m!^{-1/r}}
\, \, \,  \left(  \frac{\|\id\colon X_n\to \ell_1^n\|}{\|\id \colon X_n\to \ell_r^n\|}\right)^{m-1}\,,
\end{align*}
where now $C \ge 1$ depends on $r$ and $X$. Since $m^m \leq e^m m!$, we have that
$(m!)^{-\frac{1}{r}} \leq e^{\frac{m}{r}} m^{-\frac{m}{r}} $, and so the proof is complete.
\end{proof}

\subsection{Impact on Bohr radii} \label{Part: Bohr radii}

In this section, we focus in the study of multivariate Bohr radii. As already defined in the introduction in Equation
\eqref{definitionB},
$K(B_{X_n},J)$ stands for 
the Bohr radius of the open unit ball $B_{X_n}$ 
of a Banach lattice  ${X_n} = (\CC^n, \| \cdot \|)$
with respect to an index set $J \subset \mathbb{N}_0^{(\NN)}$.

It is immediate that $K(B_{X_n},J) =  K(B_{X_n},J \cap \mathbb{N}_0^n)$ for any index sets $J$, and for $J$, $J' \subset \mathbb{N}_0^{n}$ satisfying $J \subset J'$, the following monotonicity estimate holds:
\begin{equation}\label{rem: bohr radii monotony}
K(B_{X_n},J') \le K(B_{X_n},J)\,.
\end{equation}
Moreover, for any Banach lattice  ${X_n} = (\CC^n, \| \cdot \|)$ and index set  $J$, we have
\begin{equation}\label{lowinfty}
K(B_{\ell_\infty^n},J) \leq K(B_{X_n},J)\,.
\end{equation}
To see the proof of this inequality, one only needs to follow the argument
in \cite[Proposition 19.2]{defant2019libro}.

We write $ K_m(B_{X_n}):= K(B_{X_n},\Lambda(m))$,
and call this number  the $m$-homogeneous Bohr radius of $B_{X_n}$. Note that in this case $J=\Lambda(m)$, so in \eqref{definitionB},
we only consider $m$-homogeneous  polynomials $f \in \mathcal{P}_m( X_n) = H^{\Lambda(m)}_\infty (B_{X_n})\,.$

Typically, $X_{n} = (\CC^n, \| \cdot \|)$ will be the $n$-th section of a Banach sequence lattice  $X$ with a specific geometrical
structure (e.g., a $2$-convex space), or a concrete spaces $X$ (e.g., a Lorentz spaces $\ell_{r,s}$). Additionally,  we focus on various
index sets $J$ in with particular  structures (e.g., index sets which consist of indices  of degree at most $m$, or index sets formed
by tetrahedral indices,
or sets of  multi indices $\alpha$ generated by the prime number decompositions $n = \mathfrak{p}^\alpha$ of certain subsets of natural
numbers).

To illustrate our objectives, we recall a couple of results that serve as motivation for what follows.

From \cite{defant2011bohr}
(see also \cite[Theorem 19.1]{defant2019libro}) we know that for each $1  \leq r \leq \infty$
\[
K(B_{\ell_r^n}) \sim_C \Big(\frac{\log n}{n}\Big)^{\min \big\{ \frac{1}{2},\frac{1}{r'} \big\}}\,.
\]
This result is intimately connected with the topic of the previous section -- the asymptotic determination of
$\boldsymbol{\lambda}\big(\mathcal{P}_J(X_n)\big)$ and $\boldsymbol{\chimon}\big(\mathcal{P}_J(X_n)\big)$.

In fact, in Theorem~\ref{thm: main bohr radii}, we will integrate many of the results obtained so far to characterize the asymptotic
decay of $K(B_{\ell_{r,s}^n}, J)$ as $n$ tends to infinity. This will cover (almost) all possible values of $r$~and $s$, and as well
as a~broad class of index sets $J$.


\subsubsection{\bf Bohr radii vs unconditionality and projection constants}

We begin this subsection by presenting two results  that generalize the cases for the entire index set $J = \NN_0^n$,
first noted in \cite{defant2003bohr} (see also \cite[Proposition~19.4 and Lemma 19.5]{defant2019libro}). The proofs
of these results are straightforward extensions. They are crucial for studying multidimensional Bohr radii in Banach
spaces and demonstrate a close connection with the study of unconditionality in spaces of multivariate polynomials
on Banach spaces.

\begin{proposition}\label{prop: Km vs uncond in Pm}
Let ${X_n} = (\CC^n, \| \cdot \|)$ be a~Banach lattice and  $J \subset \NN_0^n$ an index set. Then for each $m$, we have
\[
K_m(B_{X_n},J) = \frac{1}{\boldsymbol{\chimon}\big(\Pp_{J(m)}( X_n)\big)^{1/m}}\,.
\]
\end{proposition}

\begin{theorem}\label{thm: Bohr vs unc}
Let ${X_n} = (\CC^n, \| \cdot \|)$ be a  Banach lattice and  $J \subset \NN_0^n$ an index set. Then
\[
\frac{1}{3} \inf_{ m \in \mathbb{N} } K_m(B_{X_n},J)  \le K(B_{X_n},J)  \le \inf_{m \in \mathbb{N}} K_m(B_{X_n},J)\,.
\]
\end{theorem}

We now present a direct relationship between Bohr radii and the projection constants of spaces of homogeneous polynomials,
which will be crucial for the proof of Theorem~\ref{applyhedral} (and subsequently for Theorem~\ref{thm: main bohr radii}).

Given $I,J \subset  \mathbb{N}_0^n$ with $I \subset J$, recall the definition of the projection
$
\mathbf{Q}_{J,I} \colon \mathcal{P}_J(X_n) \to \mathcal{P}_I(X_n)
$
from \eqref{anni}, and from~\eqref{index sets} the definition of the reduced index sets $J^\flat$.

\begin{theorem} \label{applyproj}
Let $X$ be a Banach sequence lattice. Then, for every $n$ and each  index set $J \subset \NN_0^n$, it holds that
\[
K(B_{X_n},J)   \ge  \frac{1}{6} \,\,\,\,\inf_{m \in \mathbb{N} }
\bigg( \frac{1}{\sqrt[m]{e\big\|\mathbf{Q}_{\Lambda(m,), J(m)}\big\|}}
\frac{1}{\sqrt[m]{\boldsymbol{\lambda}(\mathcal{P}_{J(m)^\flat}(X_n))}} \bigg)\,.
\]
\end{theorem}
Observe that on the right side of this inequality we take the $m$-root, even though $J(m)^\flat$ is by definition
contained in $\Lambda(m-1,n)$.

\begin{proof}
From Theorem~\ref{thm: Bohr vs unc} we deduce that for all $n$
\[
\frac{1}{3} \inf_{ m \in \mathbb{N} } K_m(B_{X_n},J)  \le K(B_{X_n},J)\,.
\]
On the other hand, by Theorem \ref{thm: uncond cte vs proj cte}  for all $m$ and $n$
\[
\boldsymbol{\chimon}\big( \mathcal{P}_{J(m)}(X_n) \big)\,\, \le\,\, e 2^{m}
\big\|\mathbf{Q}_{\Lambda(m,n),J(m)} \big\|\boldsymbol{\lambda}\big( \mathcal{P}_{J(m)^\flat}(X_n) \big)\,.
\]
The proof is then completed by taking the $m$-th root and using Proposition \ref{prop: Km vs uncond in Pm}.
\end{proof}

As an application, we establish both upper and lower bounds for multivariate Bohr radii in the tetrahedral case. First,
let us note the following fact, which will be repeatedly used in the subsequent results.

\begin{remark}\label{rem: maximum of f}
Given  $n \in \mathbb{N}$, the function $f_n$ defined by  $f_n(t) = \frac{1}{t n^{1/t}}$ for all $t > 0$ attains
its maximum at $t = \log n$. In particular,
\[
\dis \max_{m \in \mathbb{N}}\,f_n(m) \sim_C \frac{1}{\log n}\,.
\]
\end{remark}

\begin{theorem} \label{applyhedral}
Let $X$ be a symmetric Banach sequence lattice. Then
\[
\inf_{m\leq n} \bigg(\frac{\varphi_{X'}(m-1)}{\varphi_{X'}(n)}\bigg)^{\frac{m-1}{m}}
\prec_C\,\, K(B_{X_n},\Lambda_T)\,,
\]
and for  $1 \leq r \leq 2$,
\[
K(B_{X_n},\Lambda_T) \prec_C
\frac{(\log n)^{\frac{1}{r'}}}{\varphi_{X'}(n)}\|\id\colon X_n \to \ell_r^n\|\,.
\]
\end{theorem}

\begin{proof}
Let us start with the proof of the first claim. From Theorem~\ref{OrOuSe} we know that for all $m,n$
\[
\big\|\mathbf{Q}_{\Lambda(m,n),\Lambda_T(m,n)}:
\mathcal{P}_{\Lambda(m,n)}(X_n) \to \mathcal{P}_{\Lambda_T(m,n)}(X_n)\big\|
\leq \kappa^m\,.
\]
Moreover, by Theorem~\ref{lambda-dash} and Corollary~\ref{lambda}, we have
\begin{align*}
\boldsymbol{\lambda}\big(\mathcal{P}_{\Lambda_T(m,n)^\flat}(X_n)\big)
& \le \widehat{\boldsymbol{\lambda}}\big(\mathcal{P}_{\Lambda_T(m,n)^\flat}(X_n)\big)
\le e^{m-1} \Big(\frac{\varphi_{X_n'}(n)}{\varphi_{X_n'}(m-1)}\Big)^{m-1}\,,
\end{align*}
where the last equality holds as $m \le n$ (recall that $\Lambda_T(m,n) = \emptyset$ when $m > n$).
Implementing the preceding two inequalities into Theorem~\ref{applyproj}, yields the assertion.

For the proof of the upper estimate, note first that by Theorem \ref{thm: Bohr vs unc} and Proposition~\ref{prop: Km vs uncond in Pm}
\[
K(B_{X_n},\Lambda_T) \leq  \inf_{ m \in \mathbb{N} } K(B_{X_n},\Lambda_T(m,n))
= \frac{1}{\sup_{ m \in \mathbb{N} } \boldsymbol{\chimon}\big(\mathcal{P}_{\Lambda_T(m,n)}(X_n)\big)^{1/m}}\,.
\]
But applying Lemma~\ref{innichenA}, the fact that $\varphi_X(n)\varphi_{X'}(n)=n$ by the symmetry of $X$ and
\eqref{binom}, we see that
\begin{align*}
\sup_{ m \in \mathbb{N} } \boldsymbol{\chimon}\big(\mathcal{P}_{\Lambda_T(m,n)}(X_n)\big)^{1/m}
&
\le  \frac{1}{\|\id \colon X_n\to \ell_r^n\|}\sup_{ m \in \mathbb{N}}\frac{|\Lambda_T(m,n)|^{1/m}}{\varphi_X(n) n^{\frac{1}{mr'}} m ^{-\frac{1}{r}}} \\
&
\le  \frac{\varphi_{X'}(n) }{ n \|\id \colon X_n\to \ell_r^n\|}\sup_{m \in \mathbb{N} }\frac{  n}{n^{\frac{1}{mr'}} m ^{\frac{1}{r'}}}
\le \frac{\varphi_{X'}(n) }{ \|\id \colon X_n\to \ell_r^n\|}\sup_{ m \in \mathbb{N} }\frac{1}{n^{\frac{1}{mr'}} m ^{\frac{1}{r'}}}\,,
\end{align*}
so that the conclusion again follows using Remark~\ref{rem: maximum of f}.
\end{proof}

\subsubsection{\bf Bohr radii and convexity}

We now use the geometric concept of convexity in Banach sequence spaces to provide estimates of Bohr radii. For $2$-convex spaces,
we present an explicit characterization of the asymptotic growth of the Bohr radius for a~broad class of index sets.

\begin{theorem}\label{thm: 2 convex bohr radiiA}
Let $X$ be a $\,2$-convex Banach sequence space 
and  $J$ an index set such that $\Lambda_{T} \subset~J$.  Then, with a~constant
$C \geq 1$ only depending on $X$,
\[
K_m(B_{X_n},J) \sim_C \Big(\frac{m}{n+m}\Big)^{\frac{m-1}{2m}}\,.
\]
Moreover, with a constant $C \geq 1$ only depending on $X$,
\[
K(B_{X_n},J) \sim_C \sqrt{\frac{\log n}{n}}.
\]
\end{theorem}

\begin{proof}
The first claim follows using  Proposition~\ref{prop: Km vs uncond in Pm} and the fact that, by Theorem
\ref{lower bound 2-convex for poly} for all $m, n$
\[
\left( 1+\frac{n}{m}\right)^{\frac{m-1}{2}}\sim_{C^m} \boldsymbol{\chimon}\big(\Pp_{J(m)}(X_n)\big)\,.
\]
Note that this in particular means that
\[
\text{$K_m(B_{X_n},J) \sim_C \Big(\frac{m}{n}\Big)^{\frac{m-1}{2m}}$ \,  for $m \leq n$ \quad and \quad
$K_m(B_{X_n},J) \sim_C 1$  \, for $m \geq n$\,.}
\]
As a consequence, we deduce from Theorem~\ref{thm: Bohr vs unc}  that for some constant $C \ge 1$
\begin{align*}
\frac{1}{3} \min \left\{ \frac{1}{C} ,\inf_{m \leq n}\,\Big(\frac{m}{n}\Big)^{\frac{m-1}{2m}} \right\}
& \leq \frac{1}{3} \min \left\{  \inf_{n \leq m } K_m(B_{X_n},J),\inf_{m \leq n} K_m(B_{X_n},J) \right\} \\
& = \frac{1}{3} \inf_{ m \in \mathbb{N} } K_m(B_{X_n},J) \le K(B_{X_n},J)\\
& \le \inf_{m \in \mathbb{N}} K_m(B_{X_n},J) \le C \inf_{m \leq n}\Big(\frac{m}{n}\Big)^{\frac{m-1}{2m}}\,.
\end{align*}
Using the fact that $\Big(\frac{m}{n}\Big)^{\frac{m-1}{2m}} \sim_C \Big( \frac{m n^{\frac{1}{m}} }{n}\Big)^{\frac{1}{2}}$,
we conclude from  Remark \ref{rem: maximum of f}  and the first statement of the theorem (already proved) that
\[
K(B_{X_n},J) \sim_C K_{[\log n]}(B_{X_n},J) \sim_C \sqrt{\frac{\log n}{n}}\,. \qedhere
\]
\end{proof}

For the rest of his section the aim is to establish a~far reaching extension of the remarkable result by Bayart, Pellegrino, and Seoane
\cite{bayart2014bohr} which states that
\begin{spacing}{1}
\begin{equation}\label{BPS}
\lim_{n \to \infty} \frac{K(B_{\ell_\infty^n})}{\sqrt{\frac{\log n}{n}}}  = 1\,.
\end{equation}
\end{spacing}

We prove that this limit formula is valid for a class of Banach lattices $X_n$ distinct from $\ell_\infty^n$. Moreover, we consider
more general sets of indices beyond the full set $J = \mathbb{N}_0^{(\mathbb{N})}$, including all tetrahedral indices. In particular,
we will demonstrate that a limit analogous to \eqref{BPS} holds for a~class of Lorentz sequence spaces (see Theorem \ref{thm: main bohr radii}).

\begin{theorem}\label{limits++}
Let $X$ be a  Banach sequence lattice such that
\[
\varphi_X(n)\,\varphi_{X'}(n) \prec n\quad \text{ and } \quad
\|\id\colon X_n \to \ell_2^{n}\| \leq \frac{1}{\sqrt{n}}\,\|\id\colon X_n \to \ell_1^{n}\|\,.
\]
Then, for every index set  $J$ for which $\Lambda_{T} \subset~J$, one has
\[
\lim_{n \to \infty} \frac{K(B_{X_n},J)}{\sqrt{\frac{\log n}{n}}}  = 1\,.
\]
\end{theorem}

\begin{proof}
Combining \eqref{lowinfty} and \eqref{BPS} with the monotonicity of Bohr radii in the index set, yields
\begin{align*}
1 = \lim_{n \to \infty}  \frac{K(B_{\ell_\infty^n})}{\sqrt{\frac{\log n}{n}}}
= \liminf_{n \to \infty}  \frac{K(B_{X_n})}{\sqrt{\frac{\log n}{n}}}
&
\leq
\liminf_{n \to \infty} \frac{K(B_{X_n},J)}{\sqrt{\frac{\log n}{n}}}
\leq
\limsup_{n \to \infty} \frac{K(B_{X_n},J)}{\sqrt{\frac{\log n}{n}}}
\end{align*}
and this is less than or equal to $
\limsup_{n \to \infty} \frac{K(B_{X_n},\Lambda_T)}{\sqrt{\frac{\log n}{n}}}.$
Thus to complete the proof, it remains to show that the last term is $\leq 1$. To do so we apply Proposition~\ref{innichen}
with $r=2$, and get for each $m\leq n$
\begin{align*}
\frac{1}{\boldsymbol{\chimon}\big(\mathcal{P}_{\Lambda_T(m,n)}(X_n)\big)^\frac{1}{m}} & \leq
\left(  \frac{\|\id:X_n\to \ell_2^n\|}{\|\id:X_n\to \ell_1^n\|}\right)^{\frac{m-1}{m}}
\frac{\Big( Cm\,  (\log n)^{\frac{3}{2}} \Big)^{\frac{1}{m}}    \Big(e^{\frac{m}{2}} m^{m/2}  \Big)^{\frac{1}{m}}}
{\Big(\Big( \frac{n}{n-m} \Big)^{n-m}\Big)^{\frac{1}{m}} \Big(  \sqrt{\frac{n}{m(n-m)}}\Big)^{\frac{1}{m}}} \,,
\end{align*}
and consequently for each $m\leq n$
\begin{align*}
\frac{1}{\boldsymbol{\chimon}\big(\mathcal{P}_{\Lambda_T(m,n)}(X_n)\big)^\frac{1}{m}}
&
\leq n^{-\frac{m-1}{2m}} \frac{\Big( Cm\,  (\log n)^{\frac{3}{2}} \Big)^{\frac{1}{m}} \Big(e^{\frac{m}{2}}
m^{m/2}  \Big)^{\frac{1}{m}}}{\Big(\Big( \frac{n}{n-m} \Big)^{n-m}\Big)^{\frac{1}{m}}
\Big(\sqrt{\frac{n}{m(n-m)}}\Big)^{\frac{1}{m}}} \\
& = \Big(  \frac{m}{n} \Big)^\frac{1}{2} n^\frac{1}{2m} m^{-\frac{1}{2}}\frac{\Big( Cm\,(\log n)^{\frac{3}{2}} \Big)^{\frac{1}{m}}\Big(e^{\frac{m}{2}}m^{m/2}  \Big)^{\frac{1}{m}}}
{\Big(\Big( \frac{n}{n-m} \Big)^{n-m}\Big)^{\frac{1}{m}} \Big(  \sqrt{\frac{n}{m(n-m)}}\Big)^{\frac{1}{m}}}
\\
& = \Big(  \frac{m}{n} \Big)^\frac{1}{2} n^\frac{1}{2m}  e^{\frac{1}{2}}    \Big( \frac{n-m}{n} \Big)^{\frac{n-m}{m}}
\frac{\Big( Cm\,  (\log n)^{\frac{3}{2}} \Big)^{\frac{1}{m}}}
{\Big(\sqrt{\frac{n}{m(n-m)}}\Big)^{\frac{1}{m}}} \\
& =\Big(  \frac{m}{n} \Big)^\frac{1}{2} \left[n^\frac{1}{2m}  e^{\frac{1}{2}}    \Big( 1-\frac{m}{n} \Big)^{\frac{n}{m}}\right] \, \left[
\frac{n}{n-m}\frac{\Big( Cm\,  (\log n)^{\frac{3}{2}} \Big)^{\frac{1}{m}}}
{\Big(\sqrt{\frac{n}{m(n-m)}}\Big)^{\frac{1}{m}}}\right]\,.
\end{align*}
Taking $m=\lceil\log n\rceil$, the last factor in the previous chain of inequalities tends to $1$ as $n \to \infty$. Since
\[
\lim_{n \to \infty}
n^\frac{1}{2\lceil\log n\rceil}  e^{\frac{1}{2}}\Big( 1-\frac{\lceil\log n\rceil}{n}\Big)^{\frac{n}{\lceil\log n\rceil}}
= e^{\frac{1}{2}}e^{\frac{1}{2}} e^{-1}= 1\,,
\]
we from Theorem~\ref{thm: Bohr vs unc} and Proposition~\ref{prop: Km vs uncond in Pm} deduce
\begin{align*}
\limsup_{n \to \infty}\frac{K(B_{X_n})}{\sqrt{\frac{\log n}{n}}} &
=  \limsup_{n \to \infty} \frac{\sqrt{\frac{n}{\log n}}}{\dis\sup_{m\in \mathbb{N}}  \boldsymbol{\chimon}\big(\mathcal{P}_m(X_n)\big)^\frac{1}{m}}
\leq   \limsup_{n \to \infty} \frac{\sqrt{\frac{n}{\log n}}}{ \boldsymbol{\chimon}\big(\mathcal{P}_{\lceil\log n\rceil}(X_n)\big)^\frac{1}{\lceil\log n\rceil}} \le 1\,,
\end{align*}
which completes the proof.
\end{proof}

Note that with a similar proof, but applying Proposition \ref{toblach} instead of Proposition \ref{innichen}, we can establish
the following, more general result for the entire index set.

\begin{theorem}\label{limits}
Let $X$ be a  Banach sequence lattice such that $\frac{\|\id\colon X_n \to \ell_2^{n}\|}{\|\id\colon X_n \to \ell_1^{n}\|}
\leq \frac{1}{\sqrt{n}}$ for each $n$. Then
\[
\lim_{n \to \infty} \frac{K(B_{X_n})}{\sqrt{\frac{\log n}{n}}} = 1\,.
\]
\end{theorem}

We conclude by presenting a class of Banach sequence lattices $X$ for which  \eqref{BPS} holds when $\ell_\infty$ is replaced
by $X$, and the entire index set is replaced by any arbitrary index set that includes all tetrahedral indices. To achieve this,
we show how to construct Banach sequence lattices that satisfy the assumptions of Theorem~\ref{limits}.

First, we note that the $2$-convexification $X:=E^{(2)}$ of any Banach sequence lattice $E$ is $2$-convex with $M^{(2)}(X)=1$.
If, in addition, $X$ satisfies
\begin{equation*}
\varphi_X(n)\,\varphi_{X'}(n) \prec n\,,
\end{equation*}
then  $X$ is indeed an example we seek. This estimate clearly holds whenever $E$ is symmetric. However, it is important to point
out that it is straightforward to show that this estimate for $X:=E^{(2)}$ is valid even when  $E$ is an arbitrary Banach sequence
lattice satisfying $\varphi_E(n)\,\varphi_{E'}(n) \prec n$, which includes non-symmetric spaces  $E$.

\section{Applications for polynomials on Lorentz spaces $\ell_{r,s}^n$}

In this section, we concentrate on polynomials equipped with supremum norms on unit spheres within a significant class of symmetric
finite-dimensional Banach spaces, specifically Lorentz sequence spaces. We provide some essential definitions and properties that
will be used throughout without further reference. For a more in-depth exploration of Lorentz sequence spaces, we refer to
\cite{creekmore1981type, Jameson, LT1, Reisner}.

For $1 \leq p, q \leq \infty$ the space $\ell_{p,q}$ consists of those (complex) sequences $z$ for which (we use the convention $\frac{1}{\infty}:=0$)
\[
\Vert z \Vert_{\ell_{p,q}} := \Big\Vert \big( z^{*}_{n} n^{\frac{1}{p} - \frac{1}{q}} \big)_{n=1}^{\infty}  \Big\Vert_{\ell_{q}} < \infty\,,
\]
where $z^\ast = (z_k^\ast) $ as usual denotes the decreasing rearrangement of $|z|$.
Observe that, in general, this is a quasi-norm and only defines a norm for \( 1 \le q \le p \le \infty \). For $z \in \ell_{p,q}$ we define
\[
\Vert z \Vert_{\ell_{p,q}}^* := \left( \sum_{n=1}^\infty n^{\frac{q}{p} -1} \left( \frac{1}{n} \sum_{k=1}^n z_k^* \right)^q  \right)^{1/q}\,.
\]
It should be noted (see \cite[Lemma~4.5]{bennett1988interpolation}) that for $1 \le p,q \le \infty$ and $z \in \ell_{p,q}$, it holds
\begin{equation*}\label{eq: equivalencia normas ell_p,q}
\Vert z \Vert_{\ell_{p,q}} \le \Vert z \Vert_{\ell_{p,q}}^* \le p' \Vert z \Vert_{\ell_{p,q}}\,.
\end{equation*}
Thus, we can consistently work with the quasi-norm $\Vert \cdot \Vert_{\ell_{p,q}}$ and treat $(\ell_{p,q}, \Vert \cdot \Vert_{\ell_{p,q}})$
as a~Banach sequence space, albeit with the trade-off of $p'$ (the conjugate exponent of $p$) as a cost each time we do so. Henceforth, if
$X = \ell_{p,q}$, then $X_n$ is denoted by $\ell_{p,q}^n$.

The fundamental function for $\ell_{p,q}$ satisfies the equivalence $\varphi_{\ell_{p,q}}(n) \sim_C n^{\frac{1}{p}}$. Note that if
$1\leq q<p$, then $\ell_{p, q}$ is a~symmetric Banach space.

We heavily use the fact that the spaces $\ell_{p,q}$ are ordered lexicographically:
\begin{align*}
&
\ell_{p,q} \hookrightarrow \ell_{r, s},  \quad\,\,\,\,\, \text{for \, $p < r$} \\
&
\ell_{p, q} \hookrightarrow \ell_{r, s}, \quad\,\,\,\,\, \text{for \, $p=r$ \, and \, $q < s$\,.}
&
\end{align*}
It is easy to check (see, e.g., \cite[Lemma 22]{defant2006norms}) that, for $1 \leq r,s,t \leq \infty$
\begin{equation}\label{eq: norm identity lorentz}
\big\|{\id} \colon \ell_{r,t}^n \to \ell_{r,s}^n\big\| \prec_C\, \max\big\{1,(1 + \log n)^{\frac{1}{s}-\frac{1}{t}}\big\}\,.
\end{equation}
Moreover, notice that for $1 \le q \leq p \le \infty $ it holds that
\begin{equation}\label{eq: norma id lorentz q < p}
\|\id \colon \ell_{p,q} \to \ell_p \| = 1\,.
\end{equation}
Indeed, since $(k^{\frac{q}{p}-1})_k$ is a decreasing sequence, then
\[
z_n^* n^{1/p} = \left(n(z_n^*)^qn^{\frac{q}{p}}\right)^{\frac1{q}}
\le \left( \sum_{k=1}^n (z_k^*)^q k^{\frac{q}{p}-1} \right)^{1/q} = \| z  \|_{\ell_{p,q}}\,,
\]
hence for $z \in B_{\ell_{p,q}}$,
\[
\|z\|_{\ell_p}^p = \sum_{n \ge 1} (z_n^*)^p = \sum_{n \ge 1} (z_n^*)^q (z_n^*)^{p-q}
\le \sum_{n \ge 1} (z_n^*)^q n^{\frac{q}{p}-1} \le 1\,.
\]

Also we will need the fact that $\ell_{p, q}$ is $2$-convex whenever $2 <p < \infty$ \,and \, $2 \leq q < \infty$, and
$\ell_{p, q}$ is $2$-concave whenever $1 \leq p < 2$ \,and \,$1 \leq q \leq 2$. We refer to \cite{Jameson}, where explicit
formulae for the $q$-concavity constant of the Lorentz spaces.

Focusing on the projection constant for polynomials on Lorentz spaces $\ell_{r,s}^n$, it is essential to first address the case of $1$-homogeneous polynomials. This case corresponds to the projection constant of the scale of finite-dimensional Lorentz spaces themselves. The following result encapsulates the current state of knowledge on this topic.

\begin{proposition} \label{lor}
For all $1 < r < \infty$ \,and \, $1 \leq s \leq \infty$ one has
\[
\boldsymbol{\lambda}(\ell^n_{r,s})\sim
\begin{cases}
n^{\min\big\{\frac{1}{2}, \frac{1}{r} \big\}},  & r \neq 2\\[2mm]
n^{\frac{1}{2}},&  r = 2, \, \, \, 2 \leq s \leq \infty \\[2mm]
\Big({\frac{n}{\log \,(e + \log n)}}\Big)^{\frac12},&  r = 2, \, \,\,s=1.
\end{cases}
\]
\end{proposition}

\begin{proof}
As explained above,  $\ell_{r,s}$ is $2$-concave (then its dual is $2$-convex) for $r < 2$ and $1 \leq s \leq 2$. Thus the
statement follows from Theorem \ref{lower bound 2-convex for poly}. For the proof of the case  $r < 2$ and $2 \leq s \leq \infty$
see Theorem~\ref{t-final}. The case $r=2$, $s\ge2$ and the case $r>2$ follow from  \eqref{Carsten1} and \eqref{schuett}.
The last equivalence is a~remarkable result due to Kwapie\'n and Sch\"{u}tt \cite{schuettkwapien}, which states that
$
\boldsymbol{\lambda}(\ell^n_{2,1})\sim \big(n/\log\,\log n\big)^{\frac12}, \, n\geq 5\,.
$
\end{proof}

It should be noted that, to the best of our knowledge, the optimal estimate for the projection constant of $\ell_{r,s}^n$ with respect to the parameters $r$ and $s$ remains an open problem in general.  Even in the specific case where $r = 2$ and $1 < s < 2$—a scenario not covered by Proposition \ref{lor}—the asymptotic behavior of $\boldsymbol{\lambda}(\ell_{2,s}^n)$ as a function of the dimension $n$ remains unknown. This presents a challenge, particularly when extending the problem to the more complex issue of determining the optimal estimates for the projection constant of polynomials on Lorentz spaces $\ell_{r,s}^n$ in a general setting.

We conclude with the following observation, which is deduced from combining standard estimates involving Banach-Mazur distances.

\begin{lemma}
For every $1<s<2$, it holds that
\[
\boldsymbol{\lambda}(\ell^n_{2,s})\succ
\begin{cases} \
\frac{1}{(1 + \log n)^{1 - \frac{1}{s}}}\,
\Big({\frac{n}{\log \,(e + \log n)}}\Big)^{\frac12},& 1<s <\frac{4}{3} \\[2mm]
\frac{\sqrt{n}}{(1+\log n)^{\frac{1}{s}-\frac{1}{2}}}\,, & \frac{4}{3} \leq s <2\,.
\end{cases}
\]
\end{lemma}

\begin{proof}
Recall that for any pair of isomorphic Banach spaces $E$ and $F$ one has
\[
\boldsymbol{\lambda}(E) \leq d(E, F)\, \boldsymbol{\lambda}(F)\,.
\]
Thus, for each $n\in \mathbb{N}$, we have the estimates
\[
\boldsymbol{\lambda}({\ell_2^n)} \leq d(\ell_2^n,\, \ell_{2, s}^n)\,\boldsymbol{\lambda}(\ell_{2, s}^n), \quad \,\,\,\,
\boldsymbol{\lambda}{(\ell_{2, 1}^n)} \leq d(\ell_{2, 1}^n, \, \ell_{2, s}^n)\, \boldsymbol{\lambda}(\ell_{2, s}^n)\,,
\]
and hence
\[
\boldsymbol{\lambda}(\ell^n_{2,s}) \geq \max \Bigg\{ \frac{\boldsymbol{\lambda}(\ell_2^n)}{d(\ell_2^n,\, \ell_{2, s}^n)},
\,\,\, \frac{\boldsymbol{\lambda}(\ell_{2, 1}^n)}{d(\ell_{2, 1}^n, \, \ell_{2, s}^n)}\Bigg\}\,.
\]

\noindent
By \eqref{eq: norm identity lorentz}, it follows that
\[
d(\ell_2^n,\,\ell_{2, s}^n) \prec (1 + \log n)^{\frac{1}{s}-\frac{1}{2}}, \quad\,
d(\ell_{2, 1}^n,\,\ell_{2, s}^n )\prec (1 + \log n)^{1-\frac{1}{s}}\,.
\]
The above inequalities, combined with the asymptotic behavior $\boldsymbol{\lambda}(\ell_2^n) \sim \sqrt{n}$ and the asymptotic
for $\boldsymbol{\lambda}(\ell_{2,1}^n)$ from Proposition \ref{lor}, provide the necessary estimates.
\end{proof}

\subsection{Tetrahedral index sets of multi indices}

As previously noted, understanding the unconditional basis constants and projection constants for the space $\mathcal{P}_{J}(\ell_{r,s}^n)$,
where $J$ is a tetrahedral index set, is in fact crucial.

We will first show the upper estimates in the tetrahedral case.

\begin{theorem} \label{t-final}
Let $1 < r < \infty$ and $1 \leq s \leq \infty$. Then, there exists a constant $C=C(r, s)>0$ such that, for any sequence $\big(J_m\big)$
of tetrahedral index sets, each with degree at most~$m$, and every $n$ the following estimates hold{\rm:}
\begin{equation*} 
 \boldsymbol{\chimon}\big(\mathcal{P}_{J_m}(\ell_{r,s}^n)\big)  \leq C^m \Big( \frac{n}{m}\Big)^{(m-1)\min\big\{\frac{1}{2},\frac{1}{r'}\big\}}
\,\,\,\, \text{ and } \,\,\,\,\,\,\,
\boldsymbol{\lambda}\big(\mathcal{P}_{J_m}(\ell_{r,s}^n)\big)  \leq C^m \Big( \frac{n}{m}\Big)^{m\min\big\{\frac{1}{2},\frac{1}{r'}\big\}}\,.
\end{equation*}
\end{theorem}

\begin{proof}
For $2 \leq r < \infty$ both claims  follow from Corollary \ref{K-S for poly}.

For $1 < r \leq 2$, we begin with the upper bound for the projection constant. Using
Theorem~\ref{lambda-dash} and  Corollary~\ref{coro: bound proj constant tetra}, we obtain
\begin{align*}
\boldsymbol{\lambda}\big(\mathcal{P}_{J_m}(\ell_{r,s}^n)\big)
&
\leq
\widehat{\boldsymbol{\lambda}}\big(\mathcal{P}_{J_m}(\ell_{r,s}^n)\big)
\\
&
\le
\sup_{z\in B_{\ell_{r,s}^n}}\,\,\sum_{\alpha \in J_m} c_{\ell_{r,s}^n}(\alpha) \,|z^\alpha |
\le \sup_{z\in B_{\ell_{r,s}^n}}\,\,\sum_{\alpha \in \Lambda_T(\le m)} c_{\ell_{r,s}^n}(\alpha) \,|z^\alpha |
\\
&
\le \sum_{k=0}^{m} \sup_{z\in B_{\ell_{r,s}^n}}\,\,\sum_{\alpha \in \Lambda_T( k)} c_{\ell_{r,s}^n}(\alpha) \,|z^\alpha|
\le \sum_{k=0}^{m} e^k \left(  \frac{\varphi_{\ell_{r',s'}}(n)}{\varphi_{\ell_{r',s'}}(k)}\right)^{k}\,.
\end{align*}
However, since $\varphi_{\ell_{r',s'}}(n) \sim n^{1/r'}$, we need to provide a bound for the sum $\dis\sum_{0 \leq k \leq m} \left( \frac{n}{k} \right)^{\frac{k}{r'}}$.
Notice that $t\mapsto\big( \frac{n}{t}\big)^t$ is increasing in the interval $\big(1,\frac{n}{e}\big)$ and has a maximum at
$t=\frac{n}{e},$ hence
\begin{equation}\label{wuerzburg}
\sum_{0\leq k \leq m} \Big( \frac{n}{k}\Big)^{\frac{k}{r'}} \leq (m+1)\left(\max_{0\leq k \leq m}  \Big(\frac{n}{k}\Big)^k\right)^{\frac{1}{r'}}
\leq   (m+1)\max\left\{e^{\frac{m}{r'}} \,;\,\Big( \frac{n}{m}\Big)^{\frac{m}{r'}}\right\}\,,
\end{equation}
and the claim follows. To prove the upper bound for the unconditional basis constant, note that by Corollary~\ref{main3A} (and by our previous
results), it follows that
\begin{align*}
\boldsymbol{\chimon}\big(\mathcal{P}_{J_m}(\ell_{r,s}^n)\big)
& \leq \boldsymbol{\chimon}\big(\mathcal{P}_{\Lambda_T(\leq m)}(\ell_{r,s}^n)\big) \\
& \le C^m  \max_{1 \leq  k \leq m-1} \boldsymbol{\lambda}\big( \mathcal{P}_{\Lambda_T(k)}(X_n)\big)
\le C^m \max_{1 \leq  k \leq m-1} \Big( \frac{n}{k}\Big)^{\frac{k}{r'}}\,,
\end{align*}
hence the claim is again a consequence of the elementary estimate from \eqref{wuerzburg}.
\end{proof}

Proposition~\ref{lor} shows that the upper estimates mentioned above, specifically for $r=2$ and $s=\infty$, do not provide
the correct asymptotic behavior.  We now turn our attention to the lower estimates.

\begin{theorem} \label{t-final2}
Assume that
\[
\text{$\mathbf{A}: \,2 < r < \infty,  \,2\leq s \leq \infty$ \quad or \quad
$\mathbf{B}: \, 2 \leq r < \infty, \, 1 \leq s \leq 2$ \quad or \quad $\mathbf{C}: \,1 < r < 2 ,\, r \leq s$\,.}
\]
Then, there  is a  constant $C=C(r,s) >0$  such that for every sequence 
$\big(J_m\big)$ of index sets, each with degree at most~$m$ and such that  $\Lambda_T(m) \subset J_m$, and for each $n$ the following estimate holds{\rm:}
\begin{equation*} 
\Big(\frac{n}{m}\Big)^{(m-1)\min\big\{\frac{1}{2},\frac{1}{r'}\big\}}
\leq C^m
\boldsymbol{\chimon}\big(\mathcal{P}_{J_m}(\ell_{r,s}^n)\big)
\,\,\,\, \text{ and } \,\,\,\,\,\,\,
\Big( \frac{n}{m}\Big)^{m\min\big\{\frac{1}{2},\frac{1}{r'}\big\}}
\leq C^m
\boldsymbol{\lambda}\big(\mathcal{P}_{J_m}(\ell_{r,s}^n)\big)\,.
\end{equation*}
Moreover, in the case $2 \leq  r < \infty$ and $r \leq  s$, we (only) know that for all $n$ and $m$ with
$\Lambda_T(m) \subset J_m$
\[
\frac{1}{(\log n)^{\frac{m}{r}-\frac{m}{s} }}\Big( \frac{n}{m} \Big)^{\frac{m-1}{r'}}
\prec_{C^m} \boldsymbol{\chimon}\big(\mathcal{P}_{J_m}(\ell_{r,s}^n)\big)
\]
and
\[
\frac{1}{(\log n)^{\frac{m}{r}-\frac{m}{s} }}\Big( \frac{n}{m} \Big)^{\frac{m}{r'}}
\prec_{C^m}
\boldsymbol{\lambda}\big(\mathcal{P}_{J_m}(\ell_{r,s}^n)\big)\,.
\]
\end{theorem}

\begin{proof}
In case $\mathbf{A}$ both claims are consequences of Theorem \ref{lower bound 2-convex for poly} as $\ell_{r,s}$ is $2$-convex.

In case $\mathbf{B}$, we first address the unconditional basis constant using Proposition~\ref{toblach}. To proceed, we need
to verify that the hypotheses are satisfied. First, recall that
\begin{align} \label{fundfunc}
\varphi_{\ell_{r,s}^n}(n) \sim n^{\frac{1}{r}} \,\,\,\,\,\,\,\,
\text{and} \,\,\,\,\,\,\,\,
\varphi_{(\ell_{r,s}^n)'}(n) \sim \varphi_{\ell_{r',s'}^n}(n)~\sim~n^{\frac{1}{r'}}\,,
\end{align}
hence the first assumptions is satisfied. For the second assumption of this proposition note that, using equation
\eqref{eq: norma id lorentz q < p}, we obtain
\begin{align*}
\|\id \colon \ell_{r,s}^n \to \ell_2^n\| & \leq \|\id\colon \ell_{r,s}^n \to \ell_r^n\| \cdot \|\id \colon \ell_r^n \to \ell_2^n\| \\
& \leq  \|\id \colon \ell_{r}^n \to \ell_2^n\| = n^{\frac{1}{2}-\frac{1}{r}}
= \frac{n^{1-\frac{1}{r} } }{n^{\frac{1}{2} } } =\frac{\|\id \colon \ell_{r,s}^n \to \ell_1^n\|}{\sqrt{n}}\,. \nonumber
\end{align*}
Consequently, we deduce from Proposition~\ref{toblach} that, for $m \leq n$, it holds
\[
\Big(\frac{n}{m}\Big)^{\frac{m-1}{2}}
\prec_{C^m}
\boldsymbol{\chimon}\big( \mathcal{P}_{\Lambda_T(m)}(\ell_{r,s}^n)\big)\,.
\]
Since the case $n \leq m$ is anyway obvious, as the index set would be empty, the claim for the lower bound of the unconditional
basis constant is proved.

Let us consider  case $\mathbf{C}$, so $1\le s \leq r<2$:  For the unconditional basis constant, this is an immediate consequence
of Proposition~\ref{innichen1} and equation
\eqref{eq: norma id lorentz q < p}:
\[
\Big( \frac{n}{m} \Big)^{\frac{m-1}{r'}} = \frac{1}{\|\id\colon \ell_{r,s}^n\to \ell_r^n\|^m} \Big( \frac{n}{m} \Big)^{\frac{m-1}{r'}}
\prec_{C^m} \boldsymbol{\chimon}\big(\mathcal{P}_{\Lambda_T(m)}(\ell_{r,s}^n)\big)\,.
\]

It remains to show the claims for the projection constant in the cases $\mathbf{B}$ and $\mathbf{C}$:
We apply Theorem~\ref{thm: uncond cte vs proj cte} in conjunction with Theorem~\ref{OrOuSe} (note that $\Lambda_T(m) = J_m(m)\cap  \Lambda_T(\leq m)$) and Proposition~\ref{prop: homogeneous part proj constant}. This shows that
\begin{align}\label{end}
\boldsymbol{\chimon}\big(\mathcal{P}_{\Lambda_T(m+1)}(\ell_{r,s}^n)\big)
\prec_{C^m} \boldsymbol{\lambda}\big(\mathcal{P}_{\Lambda_T(m)}(\ell_{r,s}^n)\big)
\prec_{C^m}
\boldsymbol{\lambda}\big(\mathcal{P}_{J_m(m)}(\ell_{r,s}^n)\big)
\leq \boldsymbol{\lambda}\big(\mathcal{P}_{J_m}(\ell_{r,s}^n)\big)\,.
\end{align}
This provides the lower bound for the projection constant, using the bounds obtained for the unconditional basis constant.

Finally, it remains to prove the last statement of the theorem. The estimate for the unconditional basis constant follows
from Proposition~\ref{innichen1}, and for the estimate on the projection constant we again repeat the argument from~\eqref{end}.
\end{proof}

\subsection{Arbitrary index sets of multi indices}

Having discussed the tetrahedral case, we now turn to more general index sets. A natural question that arises is whether
variants of the estimates established for the spaces $\ell_r^n$ in \cite{bayart2019monomial, defant2011bohr} remain valid
for Lorentz spaces $\ell_{r,s}^n$, regardless of the parameter $s$. Theorems~\ref{t-final} and~\ref{t-final2} indicate that
this is indeed the case when we restrict ourselves to tetrahedral indices, at least for $r > 2$. However, there are subtle
differences in more complex scenarios. It turns out that the case of $r > 2$ can be treated with methods similar to those
used for tetrahedral index sets. In contrast, the case where $1 < r \leq 2$ is significantly different and technically much more challenging.

\subsubsection{\bf The case $\pmb{2 \leq r < \infty}$ and $\pmb{1 \leq s \leq \infty}$}

We present a theorem that summarizes our results for the case of general index sets.

\begin{theorem} \label{t-finalII}
Let $2 \leq r < \infty$ and $1 \leq s \leq \infty$. Then, there exists a~constant $C=C(r, s)>0$ such that for any sequence
$\big(J_m\big)$ of index sets, each with degree at most $m$, and for any $n$, the following estimate holds{\rm:}
\begin{equation*} 
 \boldsymbol{\chimon}\big(\mathcal{P}_{J_m}(\ell_{r,s}^n) \big) \leq C^m \Big( 1+\frac{n}{m}\Big)^{\frac{m-1}{2}}
\,\,\,\, \text{ and } \,\,\,\,\,\,\,
\boldsymbol{\lambda}\big(\mathcal{P}_{J_m}(\ell_{r,s}^n) \big) \leq C^m \Big(1+ \frac{n}{m}\Big)^{\frac{m}{2}}\,.
\end{equation*}
Moreover, in each of the cases
$\mathbf{A}: 2 < r < \infty, \,1 \leq s \leq \infty $
or
$\mathbf{B}: r= 2, \,1 \leq s \leq  2$
the preceding estimates are optimal for all $n$ and  all $m$ for which   $\Lambda_T(m) \subset J_m$ (in the sense that under this assumption that the inequality is replaced equivalence $\sim_{C^m}$).
\end{theorem}

\begin{proof}
The upper estimates are special cases  from Theorem~\ref{lower bound 2-convex for poly}. The lower one for the unconditional
basis constants follows from Theorem~\ref{t-final2}, since for all $m $ with $\Lambda_T(m) \subset J_m$ clearly
\[
\boldsymbol{\chimon}\big( \mathcal{P}_{\Lambda_T(m)}(\ell_{r,s}^n)\big)
\leq
\boldsymbol{\chimon} \big(\mathcal{P}_{J(m)} (\ell_{r,s}^n) \big)\,.
\]
 It remains to show the lower bound for the projection constant: By what was proved
in \eqref{verflucht}, we then for all $m$ with $\Lambda_T(m) \subset J_m$ have
\begin{align*}
\boldsymbol{\chimon}\big( \mathcal{P}_{\Lambda_T(m+1)}(\ell_{r,s}^n)\big)
\prec_{C^m}  \boldsymbol{\lambda}\big( \mathcal{P}_{J_m}(\ell_{r,s}^n)\big)\,,
\end{align*}
hence the claim  follows from the result on the unconditional basis constant.
\end{proof}

We already indicated that the case $r=2$ is somewhat special -- at least for $s = \infty$ (see again Proposition~\ref{lor}).
For a better overview, we collect our knowledge for $r=2$ in the following corollary. The statements are all covered by
Theorem~\ref{t-finalII} and Theorem~\ref{t-final2}

\begin{corollary}\label{thm: proy + uncon l2,s}
Let $1 \leq s \leq \infty$. Then, there exists a constant $C=C(s)>0$ such that for any sequence $\big(J_m\big)$ of index sets,
each with degree at most~$m$, the following estimates hold{\rm:}
\[
\boldsymbol{\chimon}\big(\mathcal{P}_{J_m}(\ell_{2,s}^n)\big)  \leq C^m\Big(1+ \frac{n}{m}\Big)^{\frac{m-1}{2}}
\quad and \quad
\boldsymbol{\lambda}\big(\mathcal{P}_{J_m}(\ell_{2,s}^n)\big)  \leq C^m\Big( 1+\frac{n}{m}\Big)^{\frac{m}{2}}\,,
\]
and for $1 \leq s \leq 2$ these estimates  are optimal for all $n$ and $m$ with $\Lambda_T(m) \subset J_m$.

For $2 \leq s \leq \infty$ and for all $n$ and $m$ with $\Lambda_T(m) \subset J_m$ we have that
\[
\frac{1}{(\log n)^{(m-1)(\frac{1}{2}-\frac{1}{s}) }}
\Big( 1+\frac{n}{m} \Big)^{\frac{m-1}{2}}
\prec_{C^m}
\boldsymbol{\chimon}\big(\mathcal{P}_{J_m}(\ell^n_{2,s})\big)
\]
and
\begin{equation} \label{april7} \frac{1}{(\log n)^{m(\frac{1}{2}-\frac{1}{s}) }} \Big( 1+\frac{n}{m} \Big)^{\frac{m}{2}}
\prec_{C^m} \boldsymbol{\lambda}\big(\mathcal{P}_{J_m}(\ell^n_{2,s})\big)\,.
\end{equation}
\end{corollary}

\subsubsection{\bf The case $\pmb{1 < r  \le 2}$ and $\pmb{s\leq r}$}
\label{firstcase}

Notice that estimates of the characteristic  $c_{X_n}(\alpha)$ for all $\alpha$ in the given index set $J$ will play an essential
role in the estimates of the polynomial projection. The following remark regarding estimates of characteristics will be useful
in what follows.

\begin{remark} \label{verysimple}
Let $X_n = (\mathbb{C}^n,\|\cdot\|_{X_n})$ and $Y_n = (\mathbb{C}^n,\|\cdot\|_{Y_n})$ be two Banach spaces,
and  $\alpha \in \mathbb{N}_0^{(\mathbb{N})}$. Then the following estimate holds{\rm:}
\[
c_{Y_n}(\alpha) \leq \|\id \colon X_n  \to Y_n \|^m c_{X_n}(\alpha)\,.
\]
Moreover, for every $1 \leq r \leq \infty$, we have
\[
c_{X_n}(\alpha) \le \big\|m^{-1/r}(\alpha_1^{1/r},\dots,\alpha_n^{1/r})\big\|^m_{X_n} \left( \frac{m^{m}}{\alpha^\alpha} \right)^{1/r}\,.
\]
\end{remark}

The first estimate is immediate from the given definitions. The second one follows by normalizing the vector $m^{-1/r}\alpha^{1/r}$
in $X_n$ and using  it to estimate $c_{X_n}(\alpha)$.

In the study of the estimates of polynomial projection and unconditional basis constants of $\mathcal{P}_J(\ell_{r,s}^n)$, we consider
two subcases,  each requiring different techniques: the first subcase is for $1 \leq s \leq r \leq 2$, and the second subcase is for
$1 < r \leq 2$ with $r \leq s$.

\noindent {\bf The subcase $\pmb{1 \leq s \leq r \leq 2}$:} 

The following result describes the asymptotic behavior of the polynomial projection constant of $\mathcal P_{J}(\ell_{r,s}^n)$ for
a~broad range of index sets $J$, subject to mild restrictions on $m$ and $n$.

\begin{theorem} \label{bound_similar_ell_r}
Let $1\le s \le r\le 2$. Then, there exists a constant $C=C(r, s)>0$ such that 
for all $n\geq m$ with
\begin{equation}\label{restriction}
\log m + \frac{r'}{r}(\log m)^{r(\frac1{s}-\frac1{r})} \le  \log n\,,
\end{equation}
and for all
sequences $\big(J_m\big)$ of index sets,
each with degree at most $m$, 
the following estimates hold{\rm:}
\[
\boldsymbol{\chimon}\big(\mathcal P_{J_m}( \ell_{r,s}^n)\big) \le C^m  \left(\frac{n}{m} \right)^{\frac{m-1}{r'}}
\,\,\quad \text{and}\,\,\quad
 \boldsymbol{\lambda}\big(\mathcal P_{J_m}(\ell_{r,s}^n)\big) \le C^m   \left(\frac{n}{m} \right)^{\frac{m}{r'}}\,.
\]
Moreover,  these estimates  are  optimal for all $n$ and $m$ with $\Lambda_T(m) \subset J_m$
(without restricting to \eqref{restriction}); in particular,  whenever   $m^{1+\delta} \prec_C n$ for some $\delta>0$ in the case
$(r,s)\ne(2,1)$, or $m^{2} \prec_C n$ in the case $(r,s)=(2,1)$.
\end{theorem}

For the  homogeneous case $J_m= \Lambda(m)$ the following theorem presents an upper bound that applies to the entire range, which
will be essential in deriving asymptotically optimal bounds for the Bohr radius in Lorentz spaces.

\begin{theorem} \label{proj lorentz s<r}
Let  $1\le s \le r\le 2$ and $0<\kappa<1$. Then, there exists a constant $C=C(r, s)>0$ such that for all  $n\ge m$, we have
\begin{equation*}\label{proj lorentz s<r - bound}
\widehat{\boldsymbol{\lambda}}\big(\mathcal P_m(\ell_{r,s}^n)\big) \leq C^m \Big(\frac{n}{m}\Big)^{\frac{m}{r'}}
\max\left\{1,\,\frac{\log(m)^{m(\frac1{s}-\frac1{r})}}{n^{\frac{m^\kappa}{2r'}}}\right\}\,.
\end{equation*}
\end{theorem}

To establish the upper estimates in the preceding theorems, we will isolate the following upper estimates for the polynomial
projection constants of $\mathcal{P}_{J_m}(\ell_{r,s}^n)$, in light of Theorem~\ref{lambda-dash}.

\begin{theorem} \label{bound_similar_ell_rBBB}
Let  $1\le s \le r\le 2$. Then there exists a constant $C=C(r, s)$ such that the following results hold true:
\begin{itemize}
\item[(a)] For any sequence $(J_m)$ of index sets, each with degree at most $m$, and for each $n \geq m$ that satisfies
the condition in \eqref{restriction}, we have
\begin{equation*}
\widehat{\boldsymbol{\lambda}}\big(\mathcal P_{J_m}(\ell_{r,s}^n)\big) \le C^m \Big(\frac{n}{m} \Big)^{\frac{m}{r'}}\,.
\end{equation*}
\item[(b)] For all $n \leq m$
\begin{equation*}
\widehat{\boldsymbol{\lambda}}\big(\mathcal P_m(\ell_{r,s}^n)\big)
\le C^m \Big(\frac{n}{m}\Big)^{\frac{m}{r'}}
\max\left\{1,\,\frac{\log(m)^{m(\frac1{s}-\frac1{r})}}{n^{\frac{m^\kappa}{2r'}}}\right\}\,.
\end{equation*}
\end{itemize}
\end{theorem}

Before we proceed to the proof of Theorem~\ref{bound_similar_ell_rBBB} (a), we want to ensure that it, along with our findings
in Section~\ref{Unconditionality}, actually proves Theorem~\ref{bound_similar_ell_r}. Additionally, note that
Theorem~\ref{bound_similar_ell_rBBB} (a), when combined with Theorem~\ref{lambda-dash}, leads to Theorem~\ref{proj lorentz s<r}.

\begin{proof}[Proof of Theorem \ref{bound_similar_ell_r}]
We distinguish the cases $n \le m$ and $n \ge m$. In the first case, $n \le m$, we can immediately conclude from Corollary~\ref{K-S for poly}
\[
\boldsymbol{\chimon}\big(\mathcal P_{J_m}( \ell_{r,s}^n)\big)  \sim_{C^m} 1
\quad\textrm{ and }\quad  \boldsymbol{\lambda}\big(\mathcal P_{J_m}(\ell_{r,s}^n)\big)  \sim_{C^m} 1\,.
\]
For $n \ge m$, the lower bound for the unconditional basis constant follows directly from Proposition~\ref{innichen1}. Meanwhile,
the lower estimate for the projection constant relies on Theorem~\ref{t-final2}, Theorem~\ref{OrOuSe}
(note that we by the assumption on $J_m$ have 
$\Lambda_T(m) =J_m(m)\cap \Lambda_T(\leq\ m)$)
and finally Proposition~\ref{prop: homogeneous part proj constant}, which all together  shows that
\begin{align*}
\Big(\frac{n}{m}\Big)^{\frac{m}{r'}}
\prec_{C^m}\boldsymbol{\lambda}\big(\mathcal{P}_{\Lambda_T(m)}(\ell_{r,s}^n)\big)
\prec_{C^m}
\boldsymbol{\lambda}\big(\mathcal{P}_{J_m(m) }(\ell_{r,s}^n)\big)
\prec_{C^m}\boldsymbol{\lambda}\big(\mathcal{P}_{J_m}(\ell_{r,s}^n)\big)\,.
\end{align*}
Obviously, the   upper bound for the projection constant is a consequence of  Theorem~\ref{lambda-dash}
and Theorem~\ref{bound_similar_ell_rBBB} (which  remains to be proved).
For the upper bound of the unconditional basis constant we may apply Corollary~\ref{main3A}
and the upper bound for the projection constant (already shown) to  conclude that
\begin{align*}
\boldsymbol{\chimon}\big(\mathcal{P}_{J_m}(\ell_{r,s}^n)\big)
& \le \boldsymbol{\chimon}\big(\mathcal{P}_{\leq m}(\ell_{r,s}^n)\big)
\prec_{C^m}  \boldsymbol{\chimon}\big(\mathcal{P}_{m}(\ell_{r,s}^n)\big)
\\
& \prec_{C^m} \max_{k \leq m-1} \boldsymbol{\lambda}\big(\mathcal{P}_{m-1}(\ell_{r,s}^n)\big)
\prec_{C^m}  \Big(\frac{n}{m}\Big)^{\frac{m-1}{r'}}\,. \qedhere
\end{align*}
\end{proof}
We begin the proof of Theorem~\ref{bound_similar_ell_rBBB} and first note that, without loss of generality, we may assume
that $J_m = \Lambda(m)$. Indeed, if we establish this homogeneous case, we can then consider an arbitrary $J_m$ of degree
at most $m$, and observe that
\begin{equation}\label{red1}
\widehat{\boldsymbol{\lambda}}\big(\mathcal P_{J_m}(\ell_{r,s}^n)\big)
\leq 1 +\sum_{k=1}^{m} \widehat{\boldsymbol{\lambda}}\big(\mathcal P_{\Lambda(k)}(\ell_{r,s}^n)\big)\,.
\end{equation}
The hypothesis on $n$ and $m$ from \eqref{restriction} implies
that $\log k + \frac{r'}{r}(\log k)^{r(\frac1{s}-\frac1{r})} \le  \log n$  for all $1 \leq k \leq m$, and hence
\begin{equation}\label{red2}
\widehat{\boldsymbol{\lambda}}\big(\mathcal P_{J_m}(\ell_{r,s}^n)\big)
\prec_{C^m} \sum_{k=1}^{m}  \Big(\frac{n}{k} \Big)^{\frac{k}{r'}} \prec_{C^m}  \Big(\frac{n}{m} \Big)^{\frac{m}{r'}}\,,
\end{equation}
where the last estimate is a straightforward consequence of the following technical remark, which will also be utilized later.
\begin{remark} \label{rem: bound for n/klog(n)}
Given $a>e$, define $f(t) := \big(\frac{a}{t}\big)^t$ for all $t\geq 1$. Then $f$ is increasing on $[1, a/e]$
and decreasing on $[a/e, \infty)$.  Moreover, $f$ reaches its maximum value at $t_0 = a/e$, where $f(t_0)= e^{\frac{a}{e}}$.
\end{remark}

\hspace{20 mm}

After this reduction of the proof of Theorem~\ref{bound_similar_ell_rBBB} to \( J = \bigcup_m\Lambda(m) \), the  idea is to
leverage the flexibility in the definition of \(\widehat{\boldsymbol{\lambda}}(\mathcal{P}_{\Lambda(m)}(\ell_{r,s}^n))\).
Specifically, we break down the sum
\[
\sum_{\alpha \in J_m} c_{\ell_{r,s}^n}(\alpha) |z^{\alpha}|
\]
into smaller components. This decomposition relates to the number of variables involved in each monomial. The key insight is
that if an index $\alpha$ involves a large number of variables, then $c_{\ell_{r,s}^n}(\alpha)$ closely resembles $c_{\ell_{r}^n}(\alpha)$,
placing us in a~classical scenario. Conversely, if the number of variables in $\alpha$ is limited, we encounter a troublesome
logarithmic term in
$
c_{\ell_{r,s}^n}(\alpha)
$
compared to $c_{\ell_{r}^n}(\alpha)$.
To address this, we analyze the number of monomials associated with a fixed number of variables. The underlying philosophy is that
the indices leading to poor estimates are relatively few, allowing for a compensatory effect. However, managing all these elements
simultaneously introduces significant technical challenges, requiring a great deal of subtlety. The primary difficulty lies in
balancing the number of indices with the estimates obtained and ensuring that all components can be cohesively integrated.

To accomplish this, we need to consider specific sets of multi-indices that define subspaces of polynomials based on the number of
variables involved in each monomial. We define, for $1 \le L \le m \leq n$, the set of indices
\[
\Lambda^L(m,n):=\left\{\alpha \in \Lambda(m,n) : \;\; \vert\{ i : \alpha_i \neq 0\} \vert = L \right\}\,.
\]
In other words,  $\alpha\in \Lambda^L(m,n)$ whenever the monomial $z^\alpha$ involves exactly $L$ variables.
Given a Banach space $X_n=(\mathbb C^n,\|\cdot\|)$, we denote
\[
\mathcal{P}_{m,L}(X_n):=  \mathcal{P}_{\Lambda^L(m,n)}(X_n)\,.
\]
With this notation, the set of tetrahedral \( m \)-homogeneous polynomials can be denoted by $\mathcal{P}_{m,m}(X_n)$. The following
estimate for the cardinality of the sets $\Lambda^{L}(m,n)$ is crucial.

\begin{lemma}\label{cardinal J^L} 
For $1 \le L \le m \le n$, we have
\[
|\Lambda^{L}(m,n)|\sim_{C^m} \Big(\frac{n}{L}\Big)^{L} \sim_{C^m} \binom{n}{L}\,.
\]
\end{lemma}

\begin{proof}
Any multi-index $\alpha \in \Lambda^L(m,n)$ can be expressed as the sum of a tetrahedral index $\beta \in \Lambda_T(L,n)$
and another multi-index whose support is contained within the support of $\beta$. Furthermore, if $\alpha$ has $L-k$
coordinates equal to $1$, then the remaining $k$ non-zero coordinates of $\alpha$ are at least $2$, and so
$
L-k+2k\le m, \quad or,\quad k\le m-L\,.
$
Thus, the decreasing reordering of $\alpha$ can be expressed as
\[
\alpha^*=\beta^* \, + \,(\alpha^*-\beta^*)
=(1,\dots,1,0,\dots)+(\alpha_1^*-1,
\dots,\alpha_k^*-1,0\dots)\,.
\]
Therefore, as any $\alpha^*$ can be decomposed as a sum of $\beta^* \in \Lambda_T(L,n)$ and $\alpha^* - \beta^* \in \Lambda(m-L,m-L)$,
we have
\begin{align*}
|\Lambda^{L}(m,n)|\le |\Lambda_T(L,n)|\cdot|\Lambda(m-L,m-L)|\le \binom{n}{L} \,\Big(1+\frac{m-L}{m-L}\Big)^{m-L}
\prec_{C^m} \Big(\frac{n}{L}\Big)^{L}\,.
\end{align*}
The lower bound follows from the fact that we can define an injection from
\[
\Lambda_T(L,n) \to \Lambda^{L}(m,n), \quad \alpha \mapsto (\alpha_1 + m - L, \alpha_2, \dots, \alpha_n).
\]
Therefore,
$
|\Lambda^{L}(m,n)| \ge |\Lambda_T(L,n)| \sim_{C^m} \left(\frac{n}{L}\right)^{L},
$
which completes the argument.
\end{proof}

As already mentioned, to prove Theorem~\ref{bound_similar_ell_rBBB} (a), we may concentrate on the index set $J = \bigcup_m \Lambda(m)$.
Now, observe that for $1 \le L \le m \le n$,
\begin{equation*}\label{union}
\Lambda^L(m,n) \subset \bigcup_{k=0}^k \Lambda^L(k,n),
\end{equation*}
implying
\begin{equation}\label{dashunion}
\widehat{\boldsymbol{\lambda}}\big(\mathcal{P}_{m}(\ell_{r,s}^n)\big)
\le \sum_{L=0}^m \widehat{\boldsymbol{\lambda}}\big(\mathcal P_{m,L}(\ell_{r,s}^n)\big)\,.
\end{equation}
So, given \(1 \leq L \leq m\), our next goal is to provide upper estimates for
\[
\widehat{\boldsymbol{\lambda}}\big(\mathcal P_{m,L}(\ell_{r,s}^n)\big), \quad 1 \leq L \leq m\,.
\]
From Lemma~\ref{cardinal J^L} we obtain a first  simple bound.

\begin{remark}\label{rem: bound coef functional ell_r,s}
For all $1 \leq L \leq m$
\[
\widehat{\boldsymbol{\lambda}}\big(\mathcal P_{m,L}(\ell_{r,s}^n)\big)
\prec_{C^m} \Big(\frac{n}{L}\Big)^{\frac{L}{r'}}(\log m)^{m(\frac1{s}-\frac1{r})}\,.
\]
\end{remark}

\begin{proof}
Observe that for a given $\alpha \in \Lambda^L(m,n)$, using 
Remark~\ref{verysimple}, \eqref{eq: norm identity lorentz}
and Lemma~\ref{cardinal J^L}, we obtain
\begin{equation*}\label{eq: bound c_alpha J^L}
c_{\ell_{r,s}^n}(\alpha)
\le {\big\|m^{-1/r}(\alpha_1^{1/r},\dots,\alpha_n^{1/r})\big\|_{r,s}^{m}} \, \left( \frac{m^{m}}{\alpha^\alpha} \right)^{1/r}\prec_{C^m}
(\log L)^{m(\frac1{s}-\frac1{r})}|[\alpha]|^{1/r}\,.
\end{equation*}
Thus, applying H\"older's inequality along with the binomial formula, we have 
\begin{align*}
\widehat{\boldsymbol{\lambda}}\big(\mathcal P_{m,L}(\ell_{r,s}^n)\big)
 & \prec_{C^m}
\log(m)^{m(\frac1{s}-\frac1{r})} \left(\sup_{z \in B_{\ell_{r,s}^n}}\sum_{\alpha \in \Lambda^{L}(m,n)} |[\alpha]||z^{\alpha r}| \right)^\frac{1}{r}\left|
\Lambda^{L}(m,n) \right|^{1/r'} \nonumber\\
& \prec_{C^m} (\log m)^{m(\frac1{s}-\frac1{r})} \sup_{z \in B_{\ell_{r,s}^n}}\|z\|_{\ell_r}^{m}\left| \Lambda^{L}(m,n) \right|^{1/r'} \prec_{C^m} \Big(\frac{n}{L}\Big)^{\frac{L}{r'}}
(\log m)^{m(\frac1{s}-\frac1{r})}\,.\qedhere
\end{align*}
\end{proof}

Unfortunately, the bound given in the previous remark is inadequate for proving the two bounds from Theorem~\ref{bound_similar_ell_rBBB}.
A more precise estimation is required - in fact, we for all index sets need precisely the estimate given in Theorem~\ref{bound_similar_ell_rBBB},
(a), that is: Let  $1\le s \le r\le 2$, and  $1\le m\le n$ such that $\log m + \frac{r'}{r}(\log m)^{r(\frac1{s}-\frac1{r})}\le  \log n$.
Then, for any $0\le L\le m$
\begin{equation} \label{proj lorentz s<r - ell_r bound}
\widehat{\boldsymbol{\lambda}}\big(\mathcal P_{m,L}(\ell_{r,s}^n)\big)  \prec_{C^m} \Big(\frac{n}{m}\Big)^{\frac{m}{r'}}\,.
\end{equation}
To achieve this, we divide the interval of all possible \(L\)'s into three subintervals:
\[
\Big[1, \frac{m}{2}\Big], \quad \Big[\frac{m}{2}, s_m\Big], \quad \Big[s_m, m\Big],
\quad
\text{where $s_m =\Big(1-\frac{1}{(\log m)^{r(\frac1{s}-\frac1{r})}}\Big)$}\,
\]
and handle each of these three cases separately.

We now proceed to the proof of~\eqref{proj lorentz s<r - ell_r bound}, which requires some additional preparation.
In fact, this lemma directly follows from Lemma~\ref{lem: bound for big L}, Lemma~\ref{lem: second bound for middle L},
and Lemma~\ref{lemma L>m/2}, which are presented below. We start by examining the case when \( L \) is considered
'big'—that is, when the polynomials are 'almost tetrahedral'.

\begin{lemma}\label{rem: bound for coord fun}
Let  $1\le s \le r$ and $1 \le L < m$ with $\frac{m}{2} \leq L$. Then, there exists a~constant $C=C(r,s)>0$ such that,
for each $\alpha  \in \Lambda_{L}(m,n)$, the following estimate holds{\rm:}
\[
c_{\ell_{r,s}^n}(\alpha)
\prec_{C^m} \left( \Big(\frac{(m-L)}{m}\Big)^{1/r}\log(m-L)^{\frac1{s}-\frac1{r}}+1 \right)^m  |[\alpha]|^{1/r}\,.
\]
\end{lemma}

\begin{proof}
Before starting, let us recall that $\|M^{-1/r}\beta^{1/r}\|_{\ell_r}=1$ for any $\beta\in\Lambda(M,N)$, and if $\beta$
is tetrahedral, then by \eqref{fundfunc}
\begin{equation}\label{norm for alpha tetrahedral}
\big\|M^{-1/r}\alpha^{1/r}\big\|_{r,s} = \big\|M^{-1/r}(\underbrace{1,\dots,1}_{M},0,\dots)\big\|_{r,s}\sim_C 1\,.
\end{equation}
Given $\alpha  \in \Lambda^{L}(m,n)$, for $1 \le L < m$ with $\frac{m}{2} \leq L$, note $\alpha$ has at least $2L-m$ coordinates which are equal to $1$. Indeed, if  $k$ denotes the numbers  $1$'s then $\alpha$ has $L-k$ coordinates which are greater than or equal to $2$, so that $m\ge  2(L-k)+k$. Thus,
\[
\alpha^*=(\alpha_1^*,\dots,\alpha_{m-L}^*,\underbrace{1,\dots,1}_{2L-m},0,\dots)\,.
\]
Since the degree of $(\alpha_1^*,\dots,\alpha_{m-L}^*)$
equals $2(m-L)$, it follows from \eqref{norm for alpha tetrahedral} and \eqref{eq: norm identity lorentz} that
\begin{align*}
\nonumber \big\|m^{-1/r}(\alpha_1^{1/r},\dots,\alpha_n^{1/r})\big\|_{r,s} &\le  \big\|m^{-1/r}((\alpha_1^*)^{1/r},\dots,(\alpha_{m-L}^*)^{1/r})\big\|_{r,s}+
\big\|m^{-1/r}(\underbrace{1,\dots,1}_{2L-m},0,\dots)\big\|_{r,s}\\
\nonumber
&\le \big\|(2(m-L))^{-1/r}((\alpha_1^*)^{1/r},\dots,(\alpha_{m-L}^*)^{1/r})\big\|_{r} \, \Big(\frac{2(m-L)}{m}\Big)^{1/r}\log(m-L)^{\frac1{s}-\frac1{r}}
\\
&
\quad +\big\|(2L-m)^{-1/r}(\underbrace{1,\dots,1}_{2L-m},0,\dots)\big\|_{r,s} \nonumber \\
&\prec_C \Big(\frac{2(m-L)}{m}\Big)^{1/r}\log(m-L)^{\frac1{s}-\frac1{r}}+1\,.
\end{align*}
Using that  $z = \frac{m^{-1/r}\alpha^{1/r}}{\| m^{-1/r}\alpha^{1/r} \|_{\ell_{r,s}^n}} \in B_{\ell_{r,s}^n}$, we obtain
\begin{equation*}\label{eq: bound for coord fun}
c_{\ell_{r,s}^n}(\alpha)  \le {\big\|m^{-1/r}(\alpha_1^{1/r},\dots,\alpha_n^{1/r})\big\|_{r,s}^{m}}
\left( \frac{m^{m}}{\alpha^\alpha} \right)^{1/r}
 \prec_{C^m} \left( \Big(\frac{2(m-L)}{m}\Big)^{1/r}\log(m-L)^{\frac1{s}-\frac1{r}}+1 \right)^m e^{m/r} |[\alpha]|^{1/r}\,.
 \qedhere
\end{equation*}
\end{proof}

\smallskip

\begin{lemma}\label{lem: bound for big L}
Let $1 \leq s \leq r \leq 2$, and assume $2\le m\le n$  and $L\ge  m\left(1-\frac{1}{(\log m)^{r(\frac1{s}-\frac1{r})}}\right)$.
Then, there exists a constant $C=C(r,s)>0$ such that, the following estimate holds{\rm:}
\begin{equation*}
\widehat{\boldsymbol{\lambda}}\big(\mathcal P_{m,L}(\ell_{r,s}^n)\big)
\, \prec_{C^m} \,  \Big(\frac{n}{L}\Big)^{\frac{L}{r'}} \, \prec_{C^m} \, \Big(\frac{n}{m}\Big)^{\frac{m}{r'}}\,.
\end{equation*}
\end{lemma}

\begin{proof}
The case $m=L$ is the tetrahedral case which is  contained in (the proof of) Theorem \ref{t-final}. We may thus assume $1\le L<m.$
Since $m-L\le \frac{m}{(\log m)^{r(\frac1{s}-\frac1{r})}}$, Lemma \ref{rem: bound for coord fun} yields
\[
c_{\ell_{r,s}^n}(\alpha)
\prec_{C^m}  |[\alpha]|^{1/r}\,, \quad \alpha \in \Lambda^L(m,n)\,.
\]
Therefore with an application of H\"older's inequality and Lemma \ref{cardinal J^L} we obtain,
\begin{align*}
\sup_{z\in B_{\ell_{r,s}^n}}\,\,\sum_{\alpha \in \Lambda^L(m,n)}  c_{\ell_{r,s}^n}(\alpha) \,|z^\alpha |
& \prec_{C^m} \sup_{z\in B_{\ell_{r,s}^n}}\,\,\sum_{\alpha \in \Lambda^L  (m,n)} |z^\alpha| |[\alpha]|^{1/r}
\\
&
\le \sup_{z\in B_{\ell_{r,s}^n}}\,\,\left( \sum_{\alpha \in \Lambda^L  (m,n)} |z^\alpha|^r |[\alpha]| \right)^{1/r} | \Lambda^L  (m,n) |^{1/r'}
\\
& \le \sup_{z\in B_{\ell_{r,s}^n}}\,\,\| z  \|_{\ell_r}^m \,\, | \Lambda^L  (m,n) |^{1/r'}
 \prec_{C^m}  \Big(\frac{n}{L}\Big)^{\frac{L}{r'}} \prec_{C^m}  \Big(\frac{n}{m}\Big)^{\frac{m}{r'}}\,.\label{final bound for L>m-sqrt(m)}
\qedhere
\end{align*}
\end{proof}

Lemma \ref{lem: bound for big L} provides a bound for \( \widehat{\boldsymbol{\lambda}}\big(\mathcal{P}_{m,L}(\ell_{r,s}^n)\big) \) when \( L \) is 'big', meaning that \( L \) is comparable to \( m \). The following two lemmas offer bounds for the 'middle' values of \( L \).

\begin{lemma}\label{lem: first bound for middle L }
Let $1 \leq s \leq r \leq 2$ and $1<\frac{m}{2}<L\le m\left(1-\frac{1}{(\log m)^{r(\frac1{s}-\frac1{r})}}\right)$. Then, there exists
a constant $C=C(r,s)>0$ such that, the following estimate holds{\rm:}
\begin{equation*}
\widehat{\boldsymbol{\lambda}}\big(\mathcal P_{m,L}(\ell_{r,s}^n)\big)
\,\prec_{C^m}
\, \Big(\frac{n}{m}\Big)^{m/r'} \log(m-L)^{m(\frac1{s}-\frac1{r})}
\Big(\frac{m}{n}\Big)^{\frac{m-L}{r'}}\left(\frac{m-L}{m}\right)^{\frac{m}{r}}\,.
\end{equation*}
\end{lemma}
\begin{proof}

Note first that our hypothesis on $L$ implies that $\Big(\frac{(m-L)}{m}\Big)^{1/r}\log(m-L)^{\frac1{s}-\frac1{r}}\succ_C 1.$
By Lemma \ref{rem: bound for coord fun}, for any $\alpha\in\Lambda^L(m,n)$
\begin{align*}\label{bound norm in l_rs few variables}
c_{\ell_{r,s}^n}(\alpha)
\prec_{C^m}
\Big(\frac{m-L}{m}\Big)^{m/r}\log(m-L)^{m(\frac1{s}-\frac1{r})}|[\alpha]|^{1/r}\,.
\end{align*}
Combining  H\"older's inequality with the binomial formula  and Lemma \ref{cardinal J^L}, we obtain
\begin{align*}
\widehat{\boldsymbol{\lambda}}\big(\mathcal P_{m,L}(\ell_{r,s}^n)\big) \, & \prec_{C^m} \Big(\frac{m-L}{m}\Big)^{m/r}\log(m-L)^{m(\frac1{s}-\frac1{r})} \sup_{z \in B_{\ell_{r,s}^n}}\sum_{\alpha \in \Lambda^{L}(m,n)} |[\alpha]|^{1/r}|z^\alpha|
\nonumber\\
& \prec_{C^m}
\left(\Big(\frac{m-L}{m}\Big)^{m/r}\log(m-L)^{m(\frac1{s}-\frac1{r})} \left| \Lambda^{L}(m,n) \right|^{1/r'} \right) \sup_{z \in B_{\ell_{r,s}^n}}\|z\|_{\ell_r}^{m}\nonumber\\
& \prec_{C^m}  \Big(\frac{m-L}{m}\Big)^{m/r}\log(m-L)^{m(\frac1{s}-\frac1{r})}  \left( \frac{n}{L} \right)^{L/r'}
 \sup_{z \in B_{\ell_{r,s}^n}} \|z\|^m_{\ell_{r}^n}
 \nonumber\\
& \prec_{C^m}  \Big(\frac{m-L}{m}\Big)^{m/r}\log(m-L)^{m(\frac1{s}-\frac1{r})}  \left( \frac{n}{m} \right)^{L/r'}  \nonumber\\
& =  \Big(\frac{n}{m}\Big)^{m/r'}
\log(m-L)^{m(\frac1{s}-\frac1{r})}
\Big(\frac{m}{n}\Big)^{\frac{m-L}{r'}}\left(\frac{m-L}{m}\right)^{\frac{m}{r}}\,. \qedhere
\end{align*}
\end{proof}

The next lemma consists of two parts. The first part will be instrumental in proving~\eqref{proj lorentz s<r - ell_r bound}
(which, as mentioned above, is necessary to complete the proof of Theorem \ref{bound_similar_ell_r}), while the second part
will be used to prove Theorem~\ref{proj lorentz s<r}.

\begin{lemma}\label{lem: second bound for middle L}
Let $1 \leq s \leq r \leq 2$ and $1<\frac{m}{2}<L<m\left(1-\frac{1}{(\log m)^{r(\frac1{s}-\frac1{r})}}\right)$. Then there
exists a constant $C=C(r,s)>0$ such that the following statements are true{\rm:}
\begin{itemize}
\item[(i)] For $ \log n \ge \log m + \frac{r'}{r}(\log m)^{r(\frac1{s}-\frac1{r})}$,
\begin{align*}
\widehat{\boldsymbol{\lambda}}\big(\mathcal P_{m,L}(\ell_{r,s}^n)\big)\,&\prec_{C^m} \Big(\frac{n}{m}\Big)^{\frac{m}{r'}\left(1-\frac1{(\log m)^{r(\frac1{s}-\frac1{r})}}\right)}\,.
\end{align*}
In particular, this holds for $n\ge m^{1+\delta},$ provided $\delta\ge \frac{r'}{r(\log m)^{2-\frac{r}{s}}}$.
\item[(ii)] For $ \log n \le \log m + \frac{r'}{r}(\log m)^{r(\frac1{s}-\frac1{r})}$ and  $ m\le n$, $0<\kappa<1$,
\begin{align*}
\widehat{\boldsymbol{\lambda}}\big(\mathcal P_{m,L}(\ell_{r,s}^n)\big)  \, &\prec_{C^m}\Big(\frac{n}{m}\Big)^{\frac{m}{r'}}\max\left\{\frac{(\log m)^{m(\frac1{s}-\frac1{r})}}{n^{\frac{m^\kappa}{2r'}}};\,1\right\}\,.
\end{align*}
\end{itemize}
\end{lemma}

\begin{proof}
By Lemma~\ref{lem: first bound for middle L },
\begin{equation*}
\widehat{\boldsymbol{\lambda}}\big(\mathcal P_{m,L}(\ell_{r,s}^n)\big) \,\prec_{C^m}
\, \Big(\frac{n}{m}\Big)^{m/r'} \log(m-L)^{m(\frac1{s}-\frac1{r})}
\Big(\frac{m}{n}\Big)^{\frac{m-L}{r'}}\left(\frac{m-L}{m}\right)^{\frac{m}{r}}\,.
\end{equation*}
Denote now  $m-L=tm$, with $(\log m)^{-r(\frac1{s}-\frac1{r})}\le t\le \frac12.$ Then, taking the $m$-th root and rearranging,
\begin{equation}\label{L as a function of t}
\widehat{\boldsymbol{\lambda}}\big(\mathcal P_{m,L}(\ell_{r,s}^n)\big)^{1/m}\Big(\frac{m}{n}\Big)^{1/r'}(\log m)^{\frac1{r}-\frac1{s}}
\prec_C t^{\frac1{r}}\left(\frac{m}{n}\right)^{t/r'}\,.
\end{equation}
Let $f(t):=\frac1{r}\log(t)-\frac{t}{r'}\log\frac{n}{m}$ for all $t>0$. Then, $f$ has a global maximum at
$t_0=\left(\frac{r}{r'}\log\frac{n}{m}\right)^{-1}$.  Now, since we are looking for the maximum of $f$ for
$(\log m)^{-r(\frac1{s}-\frac1{r})}\le t\le \frac12$, we have three possibilities:
\begin{enumerate}
\item[$(a)$] \label{case1} $t_0>\frac12$, in which case we  consider $t^*:=\frac{1}{2}$. This is the case when $n<e^{\frac{2r'}{r}}m$.
\item[$(b)$] \label{case3} $\log(m)^{-r(\frac1{s}-\frac1{r})} \le t_0 \le \frac{1}{2}$, in which case we  consider $t^*:=t_0$. This
is the case when $m e^{\frac{2r'}{r}}\le n$  and $\log n\le \log m +\frac{r'}{r}(\log m)^{r(\frac1{s}-\frac1{r})}$.
\item[$(c)$] \label{case2} $t_0<(\log m)^{-r(\frac1{s}-\frac1{r})}$, in which case we  consider $t^*:=(\log m)^{-r(\frac1{s}-\frac1{r})}$.
This is the case when $\log n\le \log m +\frac{r'}{r}(\log m)^{r(\frac1{s}-\frac1{r})}$.
\end{enumerate}
Replacing $t$ by $t^*$ in \eqref{L as a function of t},  $(c)$ gives us the bound stated in $(i)$. In the other two cases, replacing
$t$ by $t^*$ in \eqref{L as a function of t}, gives us:

\begin{itemize}
\item[$(a)$] for $e^{-\frac{2r}{r'}}n\le m\le n$,
\begin{align*}
\widehat{\boldsymbol{\lambda}}\big(\mathcal P_{m,L}(\ell_{r,s}^n)\big)   &\prec_{C^m}
(\log m)^{m(\frac1{s}-\frac1{r})}\,.
\end{align*}
\item[$(b)$] for $m\le e^{-\frac{2r}{r'}}n$ and $ \log n \le \log m + \frac{r'}{r}(\log m)^{r(\frac1{s}-\frac1{r})}$,
\[
\widehat{\boldsymbol{\lambda}}\big(\mathcal P_{m,L}(\ell_{r,s}^n)\big)
\prec_{C^m} \Big(\frac{n}{m}\Big)^{\frac{m}{r'}} \frac{\left(\log m\right)^{m(\frac1{s}-\frac1{r})}}{\left(\log \frac{n}{m}\right)^{\frac{m}{r}}}\,.
\]
\end{itemize}
Let $0 < \kappa < 1$. Then, for both cases (a) and (b) above, we have $n \prec_C m^2$, which implies $n^{\frac{m^\kappa}{2r'}} \prec_{C^m} 1$.
Therefore, we deduce the bound stated in (ii) for both cases.
\end{proof}

It remains to prove the case $L < \frac{m}{2}$. The key point in this scenario is that, although we do not have a strong bound for
$c_\alpha(\ell_{r,s}^n)$, the cardinality of $\Lambda^L(m,n)$ is sufficiently small to compensate.

\begin{lemma}\label{lemma L>m/2}
Let $1 \leq s \leq r \leq 2$ and  $L\le \frac{m}{2}$. Then there exists a constant $C=C(r,s)>0$ such that for  each $n \geq m \log(m)^{2r'(\frac{1}{s}-\frac{1}{r})}$, we have
\begin{equation*} \label{good bound for L>m/2}
\widehat{\boldsymbol{\lambda}}\big(\mathcal P_{m,L}(\ell_{r,s}^n)\big)
 \,\prec_{C^m} \left(\frac{n}{m} \right)^{\frac{m}{r'}}\,.
\end{equation*}
\end{lemma}

\begin{proof}
Since $s \le r$, it follows from the Remark \ref{rem: bound coef functional ell_r,s} that
\begin{align*}
\widehat{\boldsymbol{\lambda}}\big(\mathcal P_{m,L}(\ell_{r,s}^n)\big)
\prec_{C^m}  \log(L)^{m(\frac1{s}-\frac1{r})}  \left( \frac{n}{L} \right)^{L/r'} \label{bound depending on L for L>m/2} = \Big(\frac{n}{m}\Big)^{m/r'}(\log L)^{m(\frac1{s}-\frac1{r})}\Big(\frac{m^m}{n^{m-L}L^L}\Big)^{\frac{1}{r'}}.\nonumber
\end{align*}
For $n \geq m (\log m)^{2r'(\frac{1}{s}-\frac{1}{r})}$ and  $L\le \frac{m}{2}$ (and hence $2(m-L)\ge m$), so
\begin{align*}
\widehat{\boldsymbol{\lambda}}\big(\mathcal P_{m,L}(\ell_{r,s}^n)\big)
& \prec_{C^m} \Big(\frac{n}{m}\Big)^{m/r'}(\log L)^{m(\frac1{s}-\frac1{r})}\Big(\frac{m^m}{n^{m-L}L^L}\Big)^{\frac{1}{r'}}\\
&
\le \Big(\frac{n}{m}\Big)^{m/r'}\log(m)^{m(\frac1{s}-\frac1{r})}\Big(\frac{m^m}{m^{m-L}L^L}\Big)^{\frac{1}{r'}}(\log m)^{-2(m-L)(\frac1{s}-\frac1{r})}\\
&
\le \Big(\frac{n}{m}\Big)^{m/r'}\Big(\frac{m}{L}\Big)^{\frac{L}{r'}}\prec_{C^m} \Big(\frac{n}{m}\Big)^{m/r'}\,. \qedhere
\end{align*}
\end{proof}

We  are now in position to prove~\eqref{proj lorentz s<r - ell_r bound} (still needed to complete the proof of Theorem
\ref{bound_similar_ell_r}).

\begin{proof}[Proof of Equation~\eqref{proj lorentz s<r - ell_r bound}]
For $m = 1$ (and thus $L=1$) we have $ c_{\ell_{r,s}^n}(\alpha) = 1$, then
\begin{align*}
\widehat{\boldsymbol{\lambda}}\big(\mathcal P_{m,L}(\ell_{r,s}^n)\big)
& \le \sup_{z\in B_{\ell_{r,s}^n}}\,\,\sum_{j = 1}^n  |z_j | = \varphi_{\ell_{r,s}}(n) \prec_{C} n^{\frac1{r'}} \prec_{C^m} \Big(\frac{n}{m}\Big)^{\frac{m}{r'}}\,.
\end{align*}
In the  case $2 \leq m$ and $1 \leq L \leq m$, all desired estimates follow from  Lemma~\ref{lem: bound for big L},
Lemma~\ref{lem: second bound for middle L}$(i)$, and Lemma~\ref{lemma L>m/2}.
\end{proof}
As explained above, this completes the proof of Theorem~\ref{bound_similar_ell_r}. We conclude this subsection by presenting the

\begin{proof}[Proof of Theorem~$\ref{bound_similar_ell_rBBB}$], second estimate] Returning to \eqref{dashunion}, we observe that
it suffices to show that $\widehat{\boldsymbol{\lambda}}\big(\mathcal{P}_{m,L}(\ell_{r,s}^n)\big)$ satisfies the bound in
Theorem~\ref{bound_similar_ell_rBBB} (b) for all $L \le m$.

Lemma~\ref{lem: bound for big L} implies the bound for $L\ge m\big(1-1/(\log m)^{r(\frac1{s}-\frac1{r})}\big)$. For
$\frac{m}{2}\le L\le m\big(1-1/(\log m)^{r(\frac1{s}-\frac1{r})}\big)$ we use Lemma~\ref{lem: second bound for middle L}$(ii)$.
Finally, for  $L\le \frac{m}{2}$, note that since $L\le \frac{m}{2}\le m-\frac{m^\kappa}{2}\le n$, it holds
\[
\binom{n}{L} \le \binom{2n}{L} \le \binom{2n}{[m-\frac{m^\kappa}{2}]}\,.
\]
Thus, by Remark \ref{eq: bound c_alpha J^L}, for every $z \in \mathbb{C}^n$, we have
\begin{align*}
\sum_{\alpha \in \Lambda^{L}(m,n)} c_{\ell_{r,s}^n}(\alpha)  |z^\alpha|
& {\prec_{C^m}} (\log m)^{m(\frac1{s}-\frac1{r})}  \binom{2n}{[m-m^\kappa/2]}^{1/r'}  \|z\|_{\ell_r}^{m}\nonumber
\\
& {\prec_{C^m}} (\log m)^{m(\frac1{s}-\frac1{r})} \left( \Big(\frac{n}{m-m^\kappa/2}\Big)^{m-m^\kappa/2} \right)^{1/r'}  \|z\|_{\ell_r}^{m}\nonumber
\\
&=\|z\|_{\ell_r}^{m}\Big(\frac{n}{m}\Big)^{m/r'}(\log m)^{m(\frac1{s}-\frac1{r})}\Big(\frac{m-m^\kappa/2}{n}\Big)^{\frac{m^\kappa}{2r'}}
\Big(\frac{m}{m-m^\kappa/2}\Big)^{\frac{m}{r'}}\nonumber
\\
& \prec_{C^m} \|z\|_{\ell_{r,s}}^{m}\Big(\frac{n}{m}\Big)^{m/r'}\Big(\frac{(\log m)^{m(\frac1{s}-\frac1{r})}}{n^{\frac{m^\kappa}{2r'}}}\Big)\,.\nonumber \qedhere
\end{align*}
\end{proof}

To achieve our goal in this subsection, only one subcase remains.

\noindent {\bf The subcase $\pmb{1 < r \le 2}$ and $\pmb{r \leq s}$:} \label{secondcase}

 We follow a different method. Instead of partitioning the Banach
spaces $\mathcal P_J(\ell_{r,s}^n)$ into smaller pieces as previously, we use a~specific decomposition of multi-indices. This approach allows
us to control the expression $\widehat{\boldsymbol{\lambda}}\big(\mathcal{P}_J(\ell_{r,s}^n)\big)$ expressing it as a product of terms involving
monomials of lower degrees, for which we can establish suitable bounds.

\begin{theorem} \label{proj lorentz s>r}
Let $1 < r \leq 2$ and $r \leq s$. Then, there exists a constant $C=C(r, s)>0$ such that for any sequence $\big(J_m\big)$
of index sets, each with degree at most~$m$, and any $n$ the following estimate holds{\rm:}
\[
\widehat{\boldsymbol{\lambda}}\big(\mathcal P_{J_m}(\ell_{r,s}^n)\big) \le C^m \Big(\frac{n}{m} \Big)^{\frac{m}{r'}}\,,
\]
provided that $m$ and $n$ satisfy the condition
\[
\frac{n}{e(\log n)^{r'(\frac1{r}-\frac1{s})}}\ge~m\,.
\]
In particular, we have{\rm:}
\[
\boldsymbol{\chimon}\big(\mathcal P_{J_m}(\ell_{r,s}^n)\big) 
\prec_{C^m}   \Big(\frac{n}{m} \Big)^{\frac{m-1}{r'}} \quad \textrm{and}\quad\,
{\boldsymbol{\lambda}}\big(\mathcal P_{J_m}(\ell_{r,s}^n)\big)\prec_{C^m} \Big(\frac{n}{m} \Big)^{\frac{m}{r'}}\,.
\]
\end{theorem}

We start by noting that, regarding the first estimate on the polynomial projection constant, we may assume without loss of generality
that $J_m = \Lambda(m)$. This follows as in the first subcase of  Section~\ref{firstcase}; indeed, the hypothesis on $m$ and $n$ implies that
$\frac{n}{e(\log n)^{r'(\frac1{r}-\frac1{s})}}\ge~k$ for all $1 \leq k \leq m$. Hence the argument again follows
like in \eqref{red1} and \eqref{red2} using Remark~\ref{rem: bound for n/klog(n)}.

Since for all $s\geq r$ and $\alpha \in J$, we have $c_{\ell_{r,s}^n}(\alpha) \leq c_{\ell_{r}^n}(\alpha)
\prec_{C^m} |[\alpha]|^{1/r}$. Thus to prove Theorem~\ref{proj lorentz s>r} we need to show proper bounds for the sum
\begin{equation*} 
\sum_{\alpha \in \Lambda (m,n)} \vert z \vert^{\alpha} \vert [\alpha] \vert^{1/r}\,.
\end{equation*}
To do this, we will analyse smaller components of this sum: the \emph{tetrahedral} and the \emph{even} part, and use the bounds obtained for each of these parts to conclude something about sums which involve general monomials. As mentioned, this technique was developed by Galicer, Mansilla, Muro, and Sevilla-Peris \cite{galicer2021monomial} (see also \cite{mansilla2019thesis}) for studying 'sets of monomial convergence'. Additionally, this decomposition method was employed in \cite{defant2024asymptotic} to investigate the projection constant of Boolean cube function spaces.

Recall that for each $m,n$ the set of even multi indices is given by
\[
\Lambda_E(m,n) = \big\lbrace \alpha \in \Lambda(m,n) :  \alpha_i \text{ is even for $1 \le i \le n$ }  \big\rbrace.
\]
Clearly, $m$ is even whenever $\alpha \in \Lambda_E(m,n)$. Then for every $\alpha \in \Lambda_E(m,n)$ there is a~unique
$\beta \in \Lambda(m/2,n)$ such that $\alpha = 2 \beta$. Given $\alpha \in \Lambda (M,N)$, the tetrahedral part and the even part are
defined as,
\[
\big(\alpha_{T} \big)_{i} = \begin{cases}
1 & \text{ if } \alpha_{i} \text{ is odd} \\
0 & \text{ if } \alpha_{i} \text{ is even}
\end{cases}
\quad\,\,\, \text{ and } \quad\,\,\, \alpha_E=\alpha-\alpha_T\,.
\]
If $0 \leq k \leq M$ denotes the number of odd entries in  $\alpha$, then $\alpha_{T} \in \Lambda_T(k,N)$  and $\alpha_E \in \Lambda_E(M-k,N)$.
Since $(\alpha_E)_i \le \alpha_i$ for every $i$, it follows that $\alpha_E! \le \alpha!$. On the other hand, because $\alpha_T! = 1$,
we have $\alpha_T! \alpha_E! \le \alpha!$. Therefore,
\begin{align}\label{eq: cardinality decomposition even-tetra}
|[\alpha]| = \frac{M!}{\alpha!} \le \frac{M!}{\alpha_T! \alpha_E!} = \frac{M!}{(M-k)! k!} \frac{k!}{\alpha_T!} \frac{(M-k)!}{\alpha_E!}
=  \binom{M}{k} |[\alpha_T]||[\alpha_E]| \le 2^M |[\alpha_T]||[\alpha_E]|\,.
\end{align}

\smallskip

We are now ready to prove a key lemma which play a key role in the proof of Theorem \ref{proj lorentz s>r}.

\begin{lemma} \label{3 cases}
Let $1<r\leq 2$ and $r\leq s$. Then, there exists a constant $C=C(r,s)>0$ such that for all $m,n$ and all $z \in \mathbb{C}^n$,  the following estimates hold{\rm:}
\begin{equation} \label{estimate1}
\sum_{\alpha \in \Lambda_T(m,n)} |z^\alpha| |[\alpha]|^{\frac{1}{r}}  \le C^m  \| z \|_{\ell_{r,s}}^m \left(\frac{n^m}{m!}\right)^{\frac{1}{r'}}\,,
\end{equation}
\begin{equation} \label{estimate2}
\sum_{\alpha \in \Lambda_E(m,n)} |z^\alpha| |[\alpha]|^{\frac{1}{r}}
\le 2^{\frac{m}{2r}}\| z \|_{\ell_r}^m  \le C^m  \| z \|_{\ell_{r,s}}^m (\log n)^{(\frac{1}{r}-\frac{1}{s})m}\,,
\end{equation}
\begin{equation} \label{estimate3}
\sum_{\alpha \in \Lambda_E(m,n)} |z^\alpha| |[\alpha]|^{\frac{1}{r}}
\le C^m 
\begin{cases}
\left(\frac{n}{m}\right)^{m/r'} & \text{ for } m \le \frac{n}{e(\log n)^{r'(\frac1{r}-\frac1{s})}} \\
(\log n)^{m(\frac1{r}-\frac1{s})}  & \text{else.}
\end{cases}
\end{equation}
\end{lemma}

\begin{proof}
We start with the first estimate, which is straightforward when $n=1$. Thus, we can assume $n \geq 2$. Given $\alpha \in \Lambda_T(m,n)$,
note that $\alpha !=1$ and $|[\alpha]|= m!$. Then
\begin{align*}
\sum_{\alpha \in \Lambda_T(m,n)} |z^\alpha| |[\alpha]|^{\frac{1}{r}}
= \sum_{\alpha \in \Lambda_T(m,n)} | z^\alpha| |[\alpha]| \frac{1}{|[\alpha]|^{\frac{1}{r'}}}
= \Big( \sum_{k=1}^n |z_k| \Big)^{m} \frac{1}{m!^{\frac{1}{r'}}}
\prec_{C^m} n^{m/r'} \| z \|_{\ell_{r,s}}^m \frac{1}{m!^{\frac{1}{r'}}}\,,
\end{align*}
where in the last estimate \eqref{fundfunc} is used. For the proof of the second inequality let us recall first that
for each $\alpha \in \Lambda_E(m,n)$ there is a unique $\beta \in {\Lambda}(m/2,n)$
such that $\alpha = 2 \beta$. Applying an obvious inequality $2^k \le \frac{(2k)!}{k!^2} \le 2^{2k}$ true for each positive integer
$k$, we obtain
\[
|[\alpha]| = \frac{m!}{\alpha_1 ! \cdots \alpha_n!} = \Big(\frac{(m/2)!}{\beta_1 ! \cdots \beta_n!}\Big)^2 \frac{m!}{(m/2)!(m/2)!} \prod_{i=1}^n \frac{\beta_i! \beta_i!}{(2\beta_i)!} \le 2^{m/2}|[\beta]|^2\,.
\]
This implies that
\[
\frac{m!}{(m/2)!(m/2)!} \prod_{i=1}^n \frac{\beta_i! \beta_i!}{(2\beta_i)!} \le 2^m \prod_{i=1}^n \frac{1}{2^{\beta_i}} = 2^{m/2}\,.
\]
Since $2/r \ge 1$, the $\ell_1$-norm bounds the $\ell_{2/r}$-norm), so we get
\begin{align*}
\sum_{\alpha \in \Lambda_E(m,n)} |z^\alpha| |[\alpha]|^{\frac{1}{r}}
&
\le \dis 2^{\frac{m}{2r}}\sum_{\beta \in {\Lambda}(m/2,n)} |(z^2)^\beta| |[\beta]|^{2/r}
= 2^{\frac{m}{2r}} \dis\sum_{\beta \in {\Lambda}(m/2,n)} \Big( |(z^r)^\beta| |[\beta]| \Big)^{2/r}
\\&
\le 2^{\frac{m}{2r}}\Big( \dis\sum_{\beta \in {\Lambda}(m/2,n)} |(z^r)^\beta| |[\beta]| \Big)^{2/r}
=  2^{\frac{m}{2r}}\Big( \sum_{l=1}^n |z_l|^r \Big)^{m/r}
\\&
\prec_{C^m} 2^{\frac{m}{2r}}\|z\|_{\ell_{r,s}}^m (\log n)^{(\frac{1}{r}-\frac{1}{s})m}\,,
\end{align*}
where we for the last inequality used \eqref{eq: norm identity lorentz}. This concludes the proof of the first two inequalities.

To prove the third estimate we apply \eqref{eq: cardinality decomposition even-tetra} and estimates \eqref{estimate1} and
\eqref{estimate2} to get
\begin{align*}
\sum_{\alpha \in \Lambda(m,n)} |z^\alpha| |[\alpha]|^{\frac{1}{r}}
& =  \sum_{k = 0}^m \sum_{\alpha_{T} \in \Lambda_T(k,n)} \sum_{\alpha_{E} \in \Lambda_E(m-k,n)}|z^{(\alpha_{T} + \alpha_{E})}| |[\alpha_{T} + \alpha_{E}]|^\frac{1}{r}\\
& \leq 2^{\frac{m}{r}} \sum_{k = 0}^m
\left( \sum_{\alpha_{T} \in \Lambda_T(k,n)} |z^\alpha_{T} | | [\alpha_{T}]|^{\frac{1}{r}} \right)
\left( \sum_{\alpha_{E} \in \Lambda_E(m-k,n)}|z^\alpha_{E}| |[\alpha_{E}]|^\frac{1}{r} \right) \\
& \leq 2^{\frac{3m}{2r}} \sum_{k = 0}^m
\left(n^{k/r'} \| z \|_{\ell_{r,s}}^k \frac{1}{k!^{\frac{1}{r'}}}  \right)
\left((\log n)^{(m-k)(\frac{1}{r} - \frac{1}{s})} \| z \|_{\ell_{r,s}}^{m-k} \right) \\
& \le 2^{\frac{3m}{2r}} \| z \|_{\ell_{r,s}}^m m \dis\max_{k = 1, \ldots, m}
n^{k/r'} \frac{1}{k^{\frac{k}{r'}}}
 (\log n)^{(m-k)(\frac{1}{r} - \frac{1}{s})}.
\end{align*}
Using Remark~\ref{rem: bound for n/klog(n)} with $a = \frac{n}{(\log n)^{r'(\frac1{r}-\frac1{s})}}$, we have that
for $m \le a/e $,  the maximum value in the previous estimate is attained when  $k=m$. Therefore,
\begin{align*}
\sum_{\alpha \in \Lambda(m,n)} |z^\alpha| |[\alpha]|^{\frac{1}{r}}
& \prec_{C^m}  \| z \|_{\ell_{r,s}}^m n^{m/r'} \frac{1}{m^{\frac{m}{r'}}}.
\end{align*}
For $m >a/e $ we have that
$\dis\max_{ k = 1, \ldots, m} \left( \frac{a}{k}  \right)^k \le  \left( \frac{a}{a/e} \right)^{a/e} \le e^{a/e} \le e^m$,
and consequently
\begin{align*}
\sum_{\alpha \in \Lambda(m,n)} |z^\alpha| |[\alpha]|^{\frac{1}{r}}
& \prec_{C^m} 2^{\frac{m}{r}} \| z \|_{\ell_{r,s}}^m m \dis\max_{k = 1, \ldots, m}
n^{k/r'} \frac{1}{k^{\frac{k}{r'}}}
(\log n)^{(m-k)(\frac{1}{r} - \frac{1}{s})}\\
& \prec_{C^m} \| z \|_{\ell_{r,s}^m} (\log n)^{m(\frac{1}{r} - \frac{1}{s})} \left( \dis\max_{ k = 1, \ldots, m} \left( \frac{n}{k (\log n)^{r'(\frac{1}{r} - \frac{1}{s})}}  \right)^k \right)^{1/r'}
\prec_{C^m} \|z\|_{\ell_{r,s}^m} (\log n)^{m(\frac{1}{r} - \frac{1}{s})}.
\end{align*}
This completes the proof.
\end{proof}

\smallskip

\begin{proof} [Proof of Theorem \ref{proj lorentz s>r}]
We start proving the first estimate of the theorem. Since $s\geq r$, it follows from the second part Remark~\ref{verysimple} and \eqref{eq: norm identity lorentz} that
\[
c_{\ell_{r,s}^n}(\alpha) \leq \|m^{-1/r}\alpha^{1/r}\|_{\ell_{r,s}^n} \left( \frac{m^{m}}{\alpha^\alpha} \right)^{1/r}
\leq C\,\|m^{-1/r}\alpha^{1/r}\|_{\ell_{r}^n} \left( \frac{m^{m}}{\alpha^\alpha} \right)^{1/r} =
C \left( \frac{m^{m}}{\alpha^\alpha} \right)^{1/r} \leq C e^{m/r} |[\alpha]|^{1/r}\,.
\]
Additionally,  for each $\alpha \in \Lambda(m,n)$ there is a pair $(\ii,j_{m})$, with $\ii \in \Jj(m-1,j_{m})$
and $j_m\in\{1,\dots,n\}$, so that $z^\alpha = z_{j_m}z_\ii$. Therefore,
\begin{equation*}
\begin{split}
\sum_{\alpha \in \Lambda(m,n)} c_{\ell_{r,s}^n}(\alpha)  |z^\alpha|
& \prec_{C^m} \sum_{\alpha \in \Lambda(m,n)}  |z^\alpha | |[\alpha]|^{1/r} =  C^{m} \sum_{j\in \mathcal{J}(m,n)} |z_j| |[j]|^{1/r}
\\
&
= C^{m} \sum_{j_{m}=1}^n |z_{j_{m}}| \sum_{\ii \in \Jj(m-1,j_{m})} |z_\ii| |[(\ii,j_{m})]|^{\frac{1}{r}}
\\
&  \prec_{C^m}  \sum_{j_{m}=1}^n |z_{j_{m}}|   \sum_{\ii \in \Jj(m-1,j_{m})} |z_\ii| |[\ii]|^\frac{1}{r}\,,
\end{split}
\end{equation*}
where the last inequality is due to the fact that $|[(\ii,j_{m})]| \le m |[\ii]|$ for every $ \ii \in \Jj(m-1, j_{m})$.
Recall that $\varphi_{\ell_{r, s}}(k) \equiv k^{1/r'}$ and so
\[
\sum_{j_m =1}^n |z_{j_m}| \prec \|z\|_{\ell_{r, s}^n} n^{\frac{1}{r'}}\,,
\]
Combining with Lemma~\ref{3 cases}, in the range $\frac{n}{e(\log n)^{r'(\frac1{r}-\frac1{s})}}\ge m$, we obtain
\begin{align*}
\sum_{j_{m}=1}^n |z_{j_{m}}|
\sum_{\ii \in \Jj(m-1,j_{m})} |z_\ii| |[\ii]|^\frac{1}{r}
&
\prec_{C^m} \Vert z \Vert_{\ell_{r,s}}^{m-1} \sum_{j_{m}=1}^n |z_{j_{m}}|  \left(\frac{n}{m}\right)^{\frac{m-1}{r'}}
\prec_{^m} \Vert z \Vert_{\ell_{r,s}}^{m-1} \left(\frac{n}{m}\right)^{\frac{m}{r'}}
\\
&
\prec C^{m} m^{\frac{1}{r'}} \|z\|_{\ell_{r,s}^{n}}^m \Big(\frac{n}{m}\Big)^{\frac{m}{r'}}\,\Big(\frac{n}{m}\Big)^{\frac{m-1}{r'}}
\prec C^m \|z\|_{\ell_{r, s}}^{m} \Big(\frac{n}{m}\Big)^{\frac{m}{r'}}\,.
\end{align*}
The final assertions follows from  Theorem~\ref{lambda-dash} regarding the projection constant, and
Corollary~\ref{main3A} concerning the unconditional basis constant.
\end{proof}

In similar manner, we obtain the following result.

\begin{theorem} \label{proj lorentz s>r, big m}
Let $1 < r \leq 2$ and $r \leq s$. 
Then, there exists a constant $C=C(r, s)>0$ such that for any sequence $\big(J_m\big)$
of index sets, each with degree at most~$m$, and for any $n$, the following estimate holds{\rm:}
\[
\widehat{\boldsymbol{\lambda}}\big(\mathcal P_{J_m}(\ell_{r,s}^n)\big) \le C^m (\log n)^{m(\frac{1}{r} - \frac{1}{s})}\,,
\]
provided that $m$ and $n$ satisfy the condition
\[
\frac{n}{e(\log n)^{r'(\frac1{r}-\frac1{s})}}\le m\,.
\]
In particular, we have{\rm:}
\[
\boldsymbol{\chimon}\big(\mathcal P_{J_m}(\ell_{r,s}^n)\big) \prec_{C^m}  (\log n)^{m(\frac{1}{r} - \frac{1}{s})}  \quad \textrm{and}\quad\,
{\boldsymbol{\lambda}}\big(\mathcal P_{J_m}(\ell_{r,s}^n)\big) \prec_{C^m} (\log n)^{m(\frac{1}{r} - \frac{1}{s})}\,.
\]
\end{theorem}

\subsection{Bohr radii} \label{section: main bohr radii}

Finally, we apply our previous results to derive asymptotic estimates for the Bohr radii of the space of multivariate analytic
polynomials on finite-dimensional Lorentz spaces $\ell_{r,s}^n$. The following theorem provides a strong improvement over
\cite[Corollary 10]{defant2018bohr}.

\begin{theorem}\label{thm: main bohr radii}
Let $1<r<\infty$, $1 \le s \le \infty $, and let $J$ be an index set for which  $\Lambda_{T}(m) \subset J$ for all $m$.
Then the following statements are true{\rm:}
\begin{itemize}
\item[(i)] If\, $2 < r < \infty$ and $1 \le s \le \infty$, then $K(B_{\ell_{r,s}^n},J)\sim_C \sqrt{\frac{\log n}{n}}$. Moreover,
if\, $2 < r < \infty$ and $r \le s $, then
\[
\lim_{n \to \infty} \frac{K(B_{\ell_{r,s}^n},J)}{\sqrt{\frac{\log n}{n}}} = 1\,.
\]
\item[(ii)] If\, $1 < r \le 2$ and  $1 \le s \le r$, then  $K(B_{\ell_{r,s}^n},J)\sim_C \left( \frac{\log n}{n} \right)^{1-\frac{1}{r}}$\,.
\item[(iii)] If\, $1<r\le2$ and $r < s $, then  $\left( \frac{\log n}{n} \right)^{1-\frac{1}{r}}
\prec_C K(B_{\ell_{r,s}^n},J) \prec_C\frac{(\log n)^{1-\frac{1}{s}}}{n^{1-\frac{1}{r}}}$.
\end{itemize}
\end{theorem}

For the full index set $J= \N_{0}^{(\N)}$ statement (i)  of Theorem \ref{thm: main bohr radii} was first presented in  \cite[Corollary 10, (i)]{defant2018bohr},
though without the assertion about the limit. The other two cases in Theorem~\ref{thm: main bohr radii} improve the results in  \cite[Corollary 10, (ii) and (iii)]{defant2018bohr}.
Notably, (ii) of Theorem~\ref{thm: main bohr radii} is asymptotically optimal, which contrasts with \cite[Corollary 10, (ii)]{defant2018bohr}).

By the monotonicity property as described in  \eqref{rem: bohr radii monotony} it suffices to prove the upper bounds  in the 'tetrahedral case',
and the lower bounds
in the  'full case'. However, the lower estimate in the tetrahedral case appears to be technically less complex than the lower estimate in the full case. This motivates us
to provide, up to a certain extent, two separate proofs for these two situations.

\begin{proof}
Proof of (i). Recall that $\ell_{r,s}$  is $2$-convex whenever $2 < r < \infty, \,  2 \leq s \leq \infty$. Hence in this
situation the result is a special case of Theorem~\ref{thm: 2 convex bohr radiiA}. But in the case $2 < r < \infty,\, 1 \leq s \leq 2$
we also know from Theorem \ref{lower bound 2-convex for poly} that
$
\big( 1+\frac{n}{m}\big)^{\frac{m-1}{2}}\sim_{C^m}
\boldsymbol{\chimon}(\Pp_{J_m}( X_n)) \,,
$
for all  $m, n$, and so exactly the same arguments as in the proof of  Theorem~\ref{thm: 2 convex bohr radiiA} work.
To get the the limit formula, it suffices to note that the assumptions of Theorem \ref{limits++} are satisfied in the case of
$2< r \le s$.

\noindent
Proof of (ii) and (iii) (the upper bounds). We show them for the tetrahedral case only,  the general case follows by
monotonicity from \eqref{rem: bohr radii monotony}.  In fact, both  cases  are immediate from Theorem~\ref{applyhedral},
since under the assumption of $(ii)$, we have
$
\|\id\colon \ell_{r,s}^n \to \ell_r^n\| = 1$, and on the other hand assuming $(iii)$ it follows
$
\|\id\colon \ell_{r,s}^n \to \ell_r^n\| \prec_C (\log n)^{\frac{1}{r}-\frac{1}{s}}$.

\noindent Proof of (ii) and (iii) (the lower bounds in the tetrahedral case): By Theorem~\ref{applyhedral} we have that
\begin{align*}
K\big(B_{\ell_{r,s}^n},\Lambda_T\big)
& \succ_C \inf_{m\leq n}
\left(\frac{ \varphi_{\ell_{r',s'}^n(m-1)}}{\varphi_{\ell_{r',s'}^n}(n)} \right)^{\frac{m-1}{m}} \sim_C\inf_{m\leq n}
\frac{ (m-1)^{\frac{1}{r'}\frac{m-1}{m}}}{n^{\frac{1}{r'}\frac{m-1}{m}}}
\\
& \sim_C \inf_{m\leq n}
\left( \frac{m n^{\frac{1}{m}} }{n} \right) ^{\frac{1}{r'}} =\frac{1}{n^{\frac{1}{r'}}}\inf_{m\leq n}
\Big(m n^{\frac{1}{m}}\Big)^{\frac{1}{r'}}
\,,
\end{align*}
hence minimizing the last infimum (using Remark~\ref{rem: maximum of f}) leads to the  claim.

Finally, given $r$ and $s$ like in (ii) or (iii), it remains to prove the lower bounds for $ K(B_{\ell_{r,s}^n}, J)$, for any index set
$J\subset \NN_{0}^{(\NN)}$ for which  $\Lambda_{T}(m) \subset J$ for all $m$. But, again looking at  \eqref{rem: bohr radii monotony},
we may assume that  $J = \NN_0^{(\NN)}$. To achieve the bounds we are looking for, it is central to use Theorem \ref{thm: Bohr vs unc} and
good bounds for the unconditional basis constant for the space of $m$-homogeneous polynomials in each of these cases.

\noindent
$1 < r \le 2$ and $1 \le s \le r$: By  Theorem \ref{proj lorentz s<r} with $\kappa = 1/2$ and Theorem \ref{thm: uncond cte vs proj cte},
since $\frac{1}{s}-\frac{1}{r}\leq\frac{1}{r'}$, we know that,
\begin{align*}
\boldsymbol{\chimon}\big(\mathcal P_m( \ell_{r,s}^n)\big) & \prec_{C^m} \Big(\frac{n}{m} \Big)^{(m-1)/r'} \max \left\{ \Big( \frac{  \log(m)^{m}}{n^{\sqrt{m}-1}}\Big)^{1/r'},1 \right\} \\
& \prec_{C^m} \left( n^{1/r'}\Big(\frac{1}{mn^{1/m}} \Big)^{1/r'}\max \left\{ \Big(\frac{\log m}{n^{(\sqrt{m}-1)/m}}\Big)^{1/r'},1 \right\} \right)^m\\
& = \left( n^{1/r'}\max \left\{\Big(\frac{\log m}{mn^{1/\sqrt{m}}} \Big)^{1/r'},\Big(\frac{1}{mn^{1/m}} \Big)^{1/r'}\right\}\right)^m\\
&\prec_{C^m} \left( n^{1/r'}\max\left\{\Big(\frac{1}{\sqrt{m}n^{1/\sqrt{m}}} \Big)^{1/r'},\Big(\frac{1}{mn^{1/m}} \Big)^{1/r'}\right\}\right)^m.
\end{align*}
Optimizing with  Remark \ref{rem: maximum of f}  shows
$
\dis\sup_{m}  \boldsymbol{\chimon}(\mathcal P_m( \ell_{r,s}^n))^{1/m} \prec_C \Big( \frac{n}{\log n} \Big)^{1/r'},
$
and hence Theorem \ref{thm: Bohr vs unc} gives
$
\Big( \frac{n}{\log n} \Big)^{1/r'} \prec_C K(B_{X_n})$, as desired.

\noindent $1 < r \le 2$ and $ r<s$: \,By Theorems \ref{proj lorentz s>r} and \ref{proj lorentz s>r, big m}, for $m \leq n$, we have
\[
\boldsymbol{\chimon}\big(\Pp_m(\ell^n_{r,s})\big) \prec_{C^m} \begin{cases}
 \Big( \frac{n}{m n^{1/m}} \Big)^{m/r'} & m\le  \frac{n}{e (\log n)^{r'(\frac{1}{r} - \frac{1}{s})}},\\
(\log n)^{m(\frac1{r}-\frac1{s})} & \textrm{ otherwise.}
\end{cases}
\]
Choose now some  $n_0 = n_0(r,s)$  such that $\log n \le \frac{n}{e (\log n)^{r'(\frac{1}{r} - \frac{1}{s})}}
$ for all $n \ge n_0$. Then by Remark \ref{rem: maximum of f}, for each $n \ge n_0$ and $ m
\le \frac{n}{e (\log n)^{r'(\frac{1}{r} - \frac{1}{s})}}$, we have
\[
\boldsymbol{\chimon}\big(\Pp_m(\ell^n_{r,s})\big) \prec_{C^m} \left( \frac{n}{\log n} \right)^{m(1 - \frac{1}{r})}\,.
\]
On the other hand, for all $n$ and every $ m \ge \frac{n}{e (\log n)^{r'(\frac{1}{r} - \frac{1}{s})}}$, it holds
\[
\boldsymbol{\chimon}\big(\Pp_m(\ell^n_{r,s})\big) \prec_{C^m} (\log n)^{m(\frac{1}{r} - \frac{1}{s})}
\prec_{C^m} \left( \frac{n}{\log n} \right)^{m(1 - \frac{1}{r})}\,.
\]
Combining both estimates, Proposition~\ref{prop: Km vs uncond in Pm} and Theorem \ref{thm: Bohr vs unc} complete the proof.
\end{proof}

\end{document}